\documentclass[12pt]{article}
\usepackage[margin=1.1in]{geometry}
\usepackage{amsmath,amssymb,amsthm,mathtools}
\usepackage{enumitem}
\usepackage{microtype}
\usepackage{graphicx}
\usepackage{booktabs}
\usepackage{xcolor}
\usepackage{placeins}
\usepackage[round]{natbib}
\usepackage[hidelinks]{hyperref}
\usepackage[nameinlink,noabbrev]{cleveref}

\newtheorem{assumption}{Assumption}
\newtheorem{theorem}{Theorem}
\newtheorem{proposition}{Proposition}
\newtheorem{corollary}{Corollary}
\newtheorem{lemma}{Lemma}
\newtheorem{observation}{Observation}
\theoremstyle{remark}
\newtheorem{remark}{Remark}

\crefname{assumption}{assumption}{assumptions}
\Crefname{assumption}{Assumption}{Assumptions}
\crefname{theorem}{theorem}{theorems}
\Crefname{theorem}{Theorem}{Theorems}
\crefname{proposition}{proposition}{propositions}
\Crefname{proposition}{Proposition}{Propositions}
\crefname{corollary}{corollary}{corollaries}
\Crefname{corollary}{Corollary}{Corollaries}
\crefname{lemma}{lemma}{lemmas}
\Crefname{lemma}{Lemma}{Lemmas}
\crefname{observation}{observation}{observations}
\Crefname{observation}{Observation}{Observations}
\crefname{remark}{remark}{remarks}
\Crefname{remark}{Remark}{Remarks}

\newcommand{\OmegaL}{\overline{\Omega}_L}
\newcommand{\OmegaLo}{\Omega_L}
\newcommand{\DeltaL}{\Delta_L}
\newcommand{\R}{\mathbb R}
\newcommand{\N}{\mathbb N}
\newcommand{\E}{\mathbb E}
\newcommand{\Pp}{\mathbb P}
\newcommand{\calB}{\mathcal B}
\newcommand{\calZ}{\mathcal Z}
\newcommand{\calM}{\mathcal M}
\newcommand{\norm}[1]{\left\lVert #1\right\rVert}
\newcommand{\ind}[1]{\mathbf 1\{#1\}}
\newcommand{\dd}{d}
\DeclareMathOperator{\Lip}{Lip}
\DeclareMathOperator{\dist}{dist}
\DeclareMathOperator{\pers}{pers}
\DeclareMathOperator{\Pers}{Pers}

\DeclareMathOperator{\Var}{Var}
\DeclareMathOperator{\OT}{OT}
\DeclareMathOperator{\adm}{adm}

\title{Weighted persistence intensity regression}
\author{%
Matteo Pegoraro\thanks{Faculty of Informatics, Universit\`a della Svizzera italiana, Via Giuseppe Buffi 13, 6900 Lugano, Switzerland. Email: \texttt{matteo.pegoraro@usi.ch}.}
\and
Mario Beraha\thanks{Department of Economics, Management and Statistics, University of Milano-Bicocca, Piazza dell'Ateneo Nuovo 1, 20126 Milan, Italy. Email: \texttt{mario.beraha@unimib.it}.}}
\date{}

\begin{document}
\maketitle

\begin{abstract}
Persistence diagrams summarize the multiscale topological structure of data, and in applications they often arrive paired with covariates. We develop nonparametric methodology and theory for estimating the expected weighted persistence diagram conditional on a Euclidean covariate. Representing each weighted diagram as a finite random measure on a compact window, we take the density of its conditional expectation as the regression target, the conditional weighted persistence intensity. For a conditional double-kernel estimator we establish finite-sample sup-norm rates with a matching minimax lower bound, uniform rates in partial optimal transport, and an unbiased-risk cross-validation criterion for bandwidth selection. Simulations with analytically known intensities corroborate the theory and show that cross-validation selects the oracle candidate bandwidth in the exact-intensity design. The method is illustrated by studying how radial geometry in cerebral artery trees varies with age.
\end{abstract}

\noindent\textbf{Keywords:} conditional intensity; kernel regression; optimal transport; persistence diagram; topological data analysis.

\section{Introduction}\label{sec:intro}

Many modern data sets consist of complex objects rather than observations in a
fixed-dimensional Euclidean space: a single observation may be a network, a point cloud, a surface, an image, a function, or a spatial field.
The scientifically relevant variation may concern its geometric organization (connectivity, branching, loops, cavities, or spatial arrangement) and the observed objects may also differ because of position, orientation, parametrization, or discretization. Topological data analysis provides tools for extracting multiscale geometric information while allowing scientifically irrelevant variation to be removed or made invariant at the representation stage
\citep{edelsbrunner2010computational,wasserman2018topological,
pegoraro2026functional,domanin2024persistence}.
Among these tools, persistence diagrams are one of the most commonly used representations for summarizing the multiscale geometric structure of complex data. A persistence diagram is typically obtained for each statistical unit and, in many applications, is observed together with explanatory covariates. This naturally motivates statistical methods that relate variation in the diagrams themselves to variation in those covariates.

Accordingly, in this paper, we consider independent pairs $(Z_1,D_1),\ldots,(Z_n,D_n)$, where $D_i$ is a persistence diagram and $Z_i\in\calZ\subset\R^d$, and study a basic regression question: how does the average topological structure represented by $D_i$ vary with $Z_i$?
Existing statistical work has primarily considered diagrams drawn from a single
unconditional distribution, or has used diagrams as explanatory variables after mapping them into a vector space. We instead develop a regression framework in which the persistence diagram is the response.

A running example is provided by the cerebral artery data of \citet{bullitt2010effects}.
The data contain three-dimensional arterial trees reconstructed from $98$ subjects, together with subject-level information including age.
Existing analyses have shown that topological summaries of these arterial trees are associated with age \citep{bendich2016persistent,biscio2019accumulated}.
Our regression framework allows us to determine which regions of the persistence diagram account for this association. We also map the corresponding topological features back to the arterial structures that generated them. \Cref{fig:intro-schematic} illustrates the construction.

\begin{figure}[ht!]
    \centering
    \includegraphics[width=\linewidth]{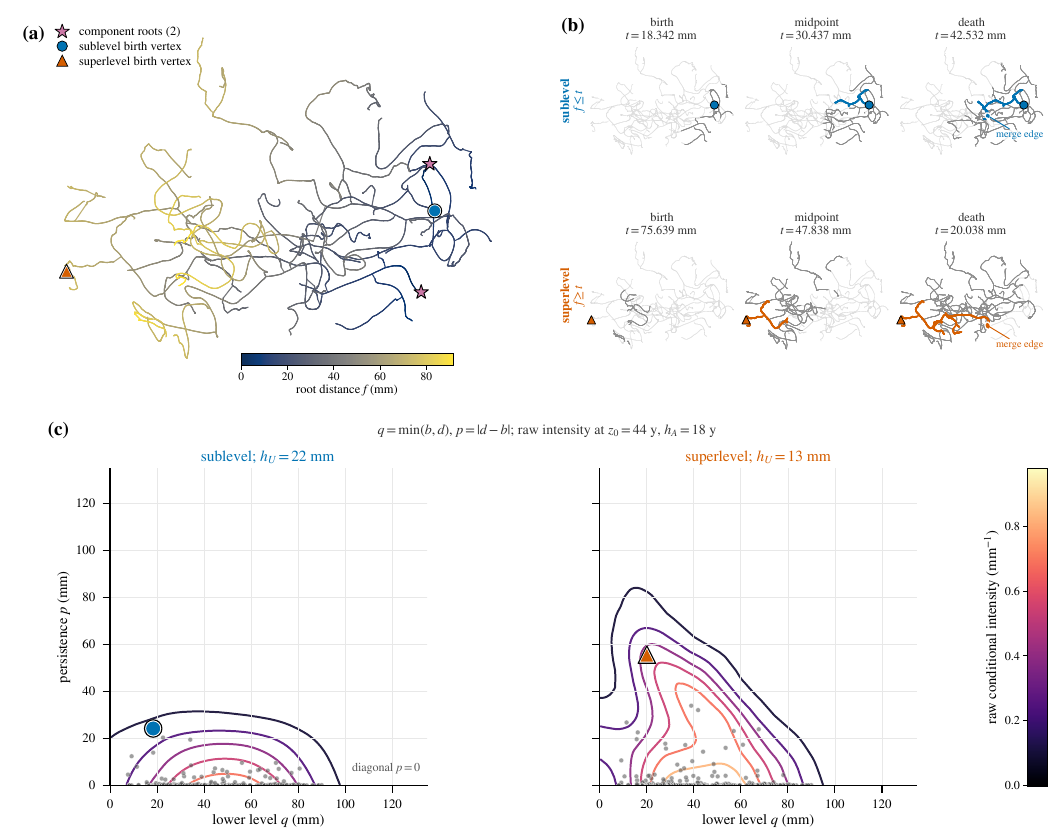}
    \caption{Cerebral-artery pipeline for the deterministic representative NormalA074, arterial system 3. (a) The tree is colored by Euclidean distance from its two deposited connected-component roots; colored symbols show the birth vertex and merge edge of the most persistent valid pair in each filtration. (b) Exact sublevel and superlevel snapshots at the selected birth, death, and contextual thresholds. (c) Finite $H_0$ diagrams in $q=\min(b,d)$, $p=|d-b|$ coordinates (mm), with the selected pairs and conditional-intensity contours at $z_0=44$ years. The fitted bandwidths are $h_A=18$ years and $h_U=22$ mm (sublevel), 13 mm (superlevel).}
    \label{fig:intro-schematic}
\end{figure}

Observe from \Cref{fig:intro-schematic} that each persistence diagram is a collection of points in the plane. Following \citet{chen2015persistence, chazal2019density, divol2021estimation, wu2024estimation}, we represent each diagram as a discrete measure, so that its expectation can be understood as an expected measure, or first moment measure \citep{daley2008introduction}. This provides a natural way to formulate the regression problem above: our inferential target is the expected diagram measure conditional on the covariates. We propose a simple double-kernel Nadaraya--Watson estimator of this conditional mean measure. When it admits a density with respect to Lebesgue measure, we establish minimax-optimal convergence rates for the corresponding conditional intensity in supremum norm; more generally, we obtain faster convergence rates directly for the conditional mean measure in partial optimal transport \citep{divol2021estimation,wu2024estimation}, without requiring absolute continuity. These ideas are developed in more detail below.

\subsection{Persistence diagrams as statistical data}

Persistence diagrams arise from persistent homology, which studies the evolution of topological features (such as connected components, loops, and higher-dimensional cavities) along a filtration, namely a nested family of representations indexed by a scale parameter. The birth and death scales of these features are recorded in the persistence diagram as a multiset of points $(b,d)$ with $b<d$. Further background on filtrations, persistent homology, and persistence diagrams is provided by \citet{edelsbrunner2010computational}.

Persistence diagrams possess a natural but nonlinear geometry. Wasserstein-type distances compare their points through partial matchings, allowing unmatched features to be sent to the diagonal and thereby accommodating diagrams with different cardinalities. For suitable constructions, stability results show that the map from the original statistical units to their persistence diagrams is Lipschitz with respect to such distances \citep{cohen2007stability,cohen2010lipschitz}. In other words, small perturbations of the filtered object lead to controlled perturbations of its diagram in these metrics. These distances therefore provide a natural geometry for comparing persistence diagrams.

Partial optimal transport geometry supports direct comparisons and Fr\'echet-type summaries \citep{mileyko2011probability,turner2014frechet}, but makes standard linear statistical operations less immediate. A large literature therefore maps diagrams into vector or function spaces before applying conventional statistical or machine-learning tools.
Examples include persistence landscapes \citep{bubenik2015statistical}, persistence
images \citep{adams2017persistence}, persistence spheres
\citep{pegoraro2026persistence}, accumulated persistence functions
\citep{biscio2019accumulated}, kernel embeddings
\citep{reininghaus2015stable,kusano2016persistence,carriere2017sliced},
and learned representations \citep{hofer2017deep}. In most of this literature the
diagram is used as an explanatory variable for classification or scalar-response
prediction.

A different inferential question concerns the persistence diagrams themselves. Classical statistical TDA often studies the uncertainty involved in estimating the persistence diagram associated with a single underlying object from finite or noisy data, with inference aimed at distinguishing persistent features from sampling noise
\citep{fasy2014confidence}. In the repeated-object setting considered here, instead, a persistence diagram is obtained for each statistical unit, and inference concerns the distribution of these diagrams across statistical units.

Under the measure interpretation introduced above, one natural way to summarize this distribution is through its first moment, the expected persistence measure, and, when this measure admits a density, through the corresponding persistence intensity. Kernel estimators of persistence intensities were introduced by \citet{chen2015persistence}; \citet{chazal2019density} established conditions for the existence and smoothness of the density of an expected persistence measure; \citet{divol2021estimation} studied estimation of the expected persistence measure directly in partial optimal transport; and \citet{wu2024estimation} obtained uniform rates for persistence intensity and normalized persistence-density functions. At a different level, Bayesian nonparametric models have been proposed to describe the full law of a random persistence diagram
\citep{maroulas2019nonparametric,maroulas2020bayesian}.

Although they target different aspects of the distribution of a random diagram, these works share with ours the same repeated-diagram sampling perspective. The statistical unit carries a persistence diagram, and accuracy improves as the number of independent diagrams increases, while the amount of raw data used to construct any individual diagram may remain fixed. The distinction is that these works consider diagrams drawn from a single unconditional distribution, whereas we introduce explanatory covariates and study a family of conditional expected measures indexed by $z\in\calZ$.

\subsection{Overview of the methodology and theoretical results}

We now make this conditional extension precise. Following the measure interpretation introduced above, let $w$ be a nonnegative weight on the persistence plane that vanishes on the diagonal. For a persistence diagram $D$, we define the weighted persistence measure and its conditional mean measure by
\[
\mu_D^w
=
\sum_{x\in\operatorname{supp}(D)}
a_x w(x)\delta_x,
\qquad
a_x\in\mathbb{N},
\qquad
\Lambda_z^w(A)
:=
\mathbb{E}\!\left[\mu_D^w(A)\mid Z=z\right].
\]
Diagonal-vanishing weights are standard in the construction of stable linear representations of persistence diagrams
\citep{kusano2016persistence,divol2019choice}: they control the contribution of features with very low persistence, which may occur in large numbers near the diagonal, without requiring them to be discarded.

The conditional mean measure $\Lambda_z^w$ has a direct interpretation. For every region $A$ of the persistence plane, $\Lambda_z^w(A)$ is the expected weighted mass of persistence features falling in $A$ among statistical units with covariate value $z$. It is the conditional analogue of the expected persistence measure studied by
\citet{chen2015persistence,chazal2019density,divol2021estimation,
wu2024estimation}.

When $\Lambda_z^w$ is absolutely continuous with respect to Lebesgue measure on the persistence plane, we call its density the conditional weighted persistence intensity, in line with the existing literature, and estimate it using the double-kernel Nadaraya--Watson--type estimator anticipated above. Informally, our estimator first smooths each diagram over the persistence coordinates and then averages the resulting surfaces over statistical units whose covariates are close to $z$. This yields a valid estimator that is straightforward to compute and requires no parametric model for the full distribution of a random diagram.

Under this assumption, we study recovery of the conditional weighted persistence intensity in supremum norm, uniformly over both the covariate and persistence domains. Although demanding, this criterion provides local control: a discrepancy concentrated in a small region of the persistence plane may correspond to a scientifically meaningful class of topological features. It also connects intensity estimation to other standard summaries and geometries. The result of \citet{wu2024estimation} implies that the supremum-norm distance between two persistence intensities controls the partial optimal transport distance between their corresponding mean measures, so supremum-norm consistency guarantees convergence under one of the main geometries used to compare persistence measures. Supremum-norm control further yields uniform guarantees for expected linear representations, including persistent Betti summaries and persistence surfaces. Under anisotropic H\"older regularity, our finite-sample bound gives the rates in \Cref{cor:holder-rate}; in the isotropic case of smoothness $s$, the rate is $(\log n/n)^{s/(2s+d+2)}$, the classical exponent for nonparametric regression in $d+2$ effective dimensions. A matching lower bound on the bounded-weighted-mass model class, together with the class-uniform upper bound furnished by the same estimator, shows that this rate is minimax optimal in supremum norm \Cref{thm:lower-bound}.

Note, however, that absolute continuity is needed only to define the conditional weighted persistence intensity, not to construct the estimator from the observed diagrams. Indeed, the same kernel-smoothed surface $\widehat\lambda_{w,n}(z,\cdot)$ is always defined from the observed diagrams and induces the measure-valued estimator
\[
\widehat\Lambda_{z,n}^{w,h_U}(A)
:=
\int_A \widehat\lambda_{w,n}(z,u)\,\dd u,
\qquad
A\in\calB(\OmegaL).
\]
We also study $\widehat\Lambda_{z,n}^{w,h_U}$ directly as an estimator of $\Lambda_z^w$ in partial optimal transport. Rather than requiring pointwise recovery of a Lebesgue density, this formulation compares the estimated and target measures and allows the conditional mean measure to contain atoms, be supported on lower-dimensional sets, or otherwise be singular with respect to Lebesgue measure. This occurs, for example, for standard $H_0$ \v{C}ech and Vietoris--Rips filtrations, where all classes are born at zero and the conditional mean measure can be supported on the segment $\{(0,d):0<d\leq L\}$. Beyond this increased generality, the direct measure-level analysis yields a different statistical rate: under a H\"older condition on the map $z\mapsto\Lambda_z^w$ in partial optimal transport, the stochastic rate for the $q$th-power partial-transport loss depends on the covariate dimension $d$, but not on the two-dimensional persistence coordinate.

In line with the existing literature, we also consider a normalized version of the weighted persistence measure, obtained by dividing each diagram measure by its total weighted mass before taking the conditional expectation. This leads to a related but distinct inferential target: whereas the unnormalized target records both the amount and the location of expected weighted persistence mass, the normalized target describes its average relative distribution within a diagram. The same techniques yield uniform convergence results for the corresponding conditional density (\Cref{thm:normalized-rate}).

For practical implementation, we derive an unbiased leave-one-out criterion for jointly selecting the covariate and persistence bandwidths (\Cref{prop:cv-risk}). We also relax some of the assumptions used in the main uniform theory: a truncation argument replaces the bounded-mass assumption with suitable moment conditions (\Cref{prop:truncation}), while trimming and clipping yield integrated consistency under weaker assumptions on the covariate design (\Cref{prop:trimmed}). The simulation studies examine estimation accuracy and bandwidth selection.

\subsection{Application to cerebral artery trees}

We now return to the cerebral artery data used to motivate the regression problem above. We apply the methodology to the data of \citet{bullitt2010effects}, using age as the explanatory covariate. For each rooted arterial-tree component, we compute the extended persistence of the Euclidean distance from the supplied root. Its finite sublevel and superlevel parts describe complementary local geometric events: inward-to-outward and outward-to-inward radial reversals, respectively, and are analysed separately.

The application asks whether age-associated changes can be localized within these persistence summaries and related back to the geometry of the original arterial trees. We estimate conditional intensities both for the raw persistence-weighted measures and after standardizing each arterial system by its reconstructed centreline length, where the latter analysis asks which changes remain when persistence mass is expressed per unit of observed vasculature. Simultaneous age-contrast bands are then used to identify regions of the persistence plane whose expected mass varies with age. Via the generating vertex of each persistence pair, selected regions can be mapped exactly back to the corresponding radial reversals on the three-dimensional centrelines. In particular, the analysis isolates a localized age-associated decrease among outward-to-inward radial reversals after system-length standardization. The resulting findings are compared with the conclusions of existing analyses of this benchmark data set in \Cref{sec:arteries-literature}.

\section{Conditional weighted intensity regression}\label{sec:estimator}

\subsection{Weighted persistence measures}\label{sec:framework}

Fix a filtration scale bound $L>0$ and write
\[
\OmegaLo:=\{(b,d):0\leq b<d\leq L\},\qquad
\OmegaL:=\{(b,d):0\leq b\leq d\leq L\},
\]
with diagonal boundary
\[
\DeltaL:=\{(b,b):0\leq b\leq L\}.
\]

For $u=(b,d)\in\OmegaL$, define its persistence (or lifetime) by
\[
\pers(u):=d-b.
\]
Thus $\dist_2(u,\DeltaL)=\pers(u)/\sqrt{2}$ for the Euclidean norm and
$\dist_\infty(u,\DeltaL)=\pers(u)/2$ for the max norm. A persistence diagram supported in $\OmegaLo$ is identified with the locally finite counting measure
\[
\mu_D=\sum_{x\in D} a_x\delta_x,\qquad a_x\in\N,
\]
where multiplicities are included in $a_x$.

A diagonal-aware weight is a measurable map
\[
w:\OmegaL\to[0,\infty),\qquad w|_{\DeltaL}=0,
\]
and the corresponding weighted persistence measure is
\[
\mu_D^w(A):=\int_A w(u)\dd\mu_D(u)=\sum_{x\in D\cap A}a_xw(x),\qquad A\in\calB(\OmegaL).
\]
Unlike $\mu_D$, the weighted measure $\mu_D^w$ is not usually a counting measure.

Typical examples are
\[
w(b,d)=\min(1,a(d-b)^\tau),\qquad \tau\geq 1, a>0,
\]
or, in birth-persistence coordinates $(b,p)=(b,d-b)$,
\[
w(b,p)=\min(1,ap^\tau).
\]
Such weights are standard in the stability theory of linear representations of persistence diagrams \citep{kusano2016persistence,divol2019choice}.

\subsection{The estimand}

We observe independent and identically distributed pairs
\[
(Z_i,D_i),\qquad i=1,\ldots,n,
\]
where $Z_i\in\calZ\subset\R^d$ is a covariate and $D_i$ is a persistence diagram. Let $\mu_i^w:=\mu_{D_i}^w$. The conditional weighted mean measure is
\[
\Lambda_z^w(A):=\E[\mu_i^w(A)\mid Z_i=z],\qquad A\in\calB(\OmegaL).
\]
When $\Lambda_z^w$ is absolutely continuous with respect to Lebesgue measure on $\OmegaL$, write
\[
\Lambda_z^w(A)=\int_A \lambda_w(z,u)\dd u.
\]
The target $\lambda_w(z,u)$ is the conditional weighted persistence intensity. If an unweighted conditional intensity $\lambda(z,u)$ exists, then $\lambda_w(z,u)=w(u)\lambda(z,u)$. Existence and smoothness of (unconditional) persistence intensities for diagrams built from random point clouds are established in \citet{chazal2019density}.

\subsection{Double-kernel estimator and boundary correction}\label{sec:nw}

Let $K_Z:\R^d\to[0,\infty)$ and $K_U:\R^2\to[0,\infty)$ be kernels. Define the normalizing factors
\[
c_{Z,h}(z):=\int_\calZ h^{-d}K_Z\!\left(\frac{z'-z}{h}\right)\dd z',\qquad
c_{U,h}(u):=\int_{\OmegaL} h^{-2}K_U\!\left(\frac{v-u}{h}\right)\dd v,
\]
and normalized kernels
\begin{equation}\label{eq:normalized-kernels}
\kappa_{Z,h}(z,z'):=\frac{h^{-d}K_Z((z'-z)/h)}{c_{Z,h}(z)},\qquad
\kappa_{U,h}(u,v):=\frac{h^{-2}K_U((v-u)/h)}{c_{U,h}(u)}.
\end{equation}
As shown in \Cref{obs:kernel-bounds}, \Cref{ass:kernels,ass:geometry-Z}
guarantee that the normalizing factors are finite and bounded away from zero
uniformly over their respective domains for all sufficiently small bandwidths,
so that
\begin{equation}\label{eq:kernel-normalization}
\int_\calZ \kappa_{Z,h}(z,z')\dd z'=1,
\qquad
\int_{\OmegaL}\kappa_{U,h}(u,v)\dd v=1.
\end{equation}
We use $h_Z$ for the covariate-space bandwidth and $h_U$ for the persistence-plane bandwidth. For each weighted diagram, define its kernelized response through kernel smoothing,
\[
Y^w_{i,h_U}(u):=\int_{\OmegaL}\kappa_{U,h_U}(u,v)\dd\mu_i^w(v),
\]
and, at locations $z$ where the denominator is positive, define the Nadaraya--Watson estimator \citep{nadaraya1964estimating,watson1964smooth}
\begin{equation}\label{eq:nw-estimator}
\widehat\lambda_{w,n}(z,u)
:=\frac{\sum_{i=1}^n \kappa_{Z,h_Z}(z,Z_i)Y^w_{i,h_U}(u)}{\sum_{i=1}^n \kappa_{Z,h_Z}(z,Z_i)}.
\end{equation}
When the denominator is zero, set $\widehat\lambda_{w,n}(z,u)=0$. Under the assumptions of \Cref{thm:supnorm-rate}, the denominator is uniformly positive on the high-probability event used in the proof, so this convention does not affect the stated bound.

For fixed $z$, the factor $h_Z^{-d}/c_{Z,h_Z}(z)$ is common to every term in the numerator and denominator of \eqref{eq:nw-estimator}, and therefore cancels exactly. Thus, whenever the denominator is nonzero,
\[
\widehat\lambda_{w,n}(z,u)
=\frac{\sum_{i=1}^n K_Z((Z_i-z)/h_Z)Y^w_{i,h_U}(u)}
{\sum_{i=1}^n K_Z((Z_i-z)/h_Z)}.
\]
Thus the normalization of the covariate kernel is an algebraic convenience and does not change the untrimmed Nadaraya--Watson estimator. We retain this normalization because, after division by $n$, the expected denominator remains bounded above and below by constants independent of $h_Z$. Without the factor $h_Z^{-d}$, its expectation would be proportional to $h_Z^d$ and would therefore vanish as $h_Z\to0$. This rescaling simplifies the proof bounds and the formulation of the denominator-trimming rule in \Cref{sec:trimmed}. In contrast, $c_{U,h_U}(u)$ appears inside the kernelized response $Y^w_{i,h_U}(u)$ and does not cancel. It corrects for the portion of the persistence-plane kernel that falls outside $\OmegaL$ when $u$ is near the boundary, preventing attenuation from kernel truncation and permitting the stated uniform bias bound up to that boundary.
Expanding the kernelized response also gives
\[
\widehat\lambda_{w,n}(z,u)
=\frac{\sum_i\kappa_{Z,h_Z}(z,Z_i)\sum_{x\in D_i}a_xw(x)\kappa_{U,h_U}(u,x)}{\sum_i\kappa_{Z,h_Z}(z,Z_i)}.
\]
Under the compact-support condition of \Cref{ass:kernels}, only observations satisfying $\|Z_i-z\|\leq h_Z$ receive nonzero covariate weight. The denominator normalizes these covariate weights; it does not normalize by the number of points in each persistence diagram. Dividing by the total diagram mass would change the estimand from an intensity to a normalized density, the object studied separately in \Cref{sec:normalized}.

\subsection{Assumptions}\label{sec:assumptions}

The following group of assumptions yields a clean uniform theorem;
\Cref{sec:truncation} explains weaker variants.

\begin{assumption}[Admissible diagonal weight]\label{ass:weight}
The function $w:\OmegaL\to[0,\infty)$ is bounded and Lipschitz. More generally, when a $\tau$-total-persistence condition is used, we allow
\begin{equation}\label{eq:weight-order}
w(u)\leq C_w\pers(u)^\tau
\end{equation}
for some $\tau>0$.
\end{assumption}

\begin{assumption}[Kernels and bandwidth sequences]\label{ass:kernels}
The kernels $K_Z:\R^d\to[0,\infty)$ and $K_U:\R^2\to[0,\infty)$ are
bounded and Lipschitz, integrate to one, and satisfy
\[
\operatorname{supp}(K_Z)\subseteq \overline B_d(0,1),
\qquad
\operatorname{supp}(K_U)\subseteq \overline B_2(0,1).
\]
They are uniformly positive on neighbourhoods of the origin: there exist
$\rho_{K,Z},\rho_{K,U},k_Z,k_U>0$ such that
\begin{align*}
K_Z(x)&\geq k_Z\quad\text{whenever }\norm{x}\leq \rho_{K,Z},\\
K_U(y)&\geq k_U\quad\text{whenever }\norm{y}\leq \rho_{K,U}.
\end{align*}
Writing $h_Z=h_{Z,n}$ and $h_U=h_{U,n}$, the bandwidth sequences are positive
and satisfy
\[
h_Z\to0,
\qquad
h_U\to0,
\qquad
\frac{nh_Z^d h_U^2}{\log n}\to\infty.
\]
\end{assumption}

The unit-ball support convention entails no loss of generality: any fixed
compact support, including the box support of a product kernel, can be
accommodated by a fixed rescaling of the kernel and its bandwidth. Positivity
near the origin ensures that a nonvanishing amount of kernel mass is retained
at domain boundaries. The final condition requires the effective sample size
in a joint covariate--persistence smoothing neighbourhood to dominate the
logarithmic cost of uniform estimation.

\begin{assumption}[Covariate-domain geometry]\label{ass:geometry-Z}
The covariate domain $\calZ\subset\R^d$ is compact and satisfies a uniform interior cone condition: there exist constants $\rho_Z>0$ and $\eta_Z>0$ such that, for every $z\in\calZ$, there is a unit vector $\xi_z\in\mathbb S^{d-1}$ for which
\begin{equation}\label{eq:interior-cone-Z}
\left\{
z+t v:\ 0<t<\rho_Z,\ v\in\mathbb S^{d-1},\
\norm{v-\xi_z}<\eta_Z
\right\}
\subset \calZ.
\end{equation}
\end{assumption}

The cone condition rules out arbitrarily thin boundary cusps. More precisely,
it implies that there exist $r_0,a_0>0$ such that
\begin{equation}\label{eq:covariate-volume-lower}
\operatorname{Leb}_d\{B(z,r)\cap\calZ\}\geq a_0r^d
\qquad
\forall z\in\calZ,\quad 0<r\leq r_0.
\end{equation}
The condition holds for compact convex domains with nonempty interior and, more
generally, for bounded Lipschitz domains.

\begin{assumption}[Design density for the uniform theorem]\label{ass:design}
The law $P_Z$ has a density $p_Z$ on $\calZ$ satisfying
\[
0<c_Z\leq p_Z(z)\leq C_Z<\infty\qquad \forall z\in\calZ.
\]
\end{assumption}

The lower bound $p_Z\geq c_Z$ is appropriate for a uniform theorem over all $z\in\calZ$. It is stronger than necessary for the integrated loss; the trimmed estimator of \Cref{sec:trimmed} weakens it to local mass conditions on high-probability subsets of $\calZ$.

The logarithmic bandwidth consequences and normalized-kernel bounds used
below are collected, with their proofs, in
\Cref{app:preliminary-proofs}.

\begin{assumption}[First-order mean regularity]\label{ass:first-moment}
The conditional weighted mean measure $\Lambda^w_z$ is absolutely continuous with respect to Lebesgue measure on $\OmegaL$, with density $\lambda_w:\calZ\times\OmegaL\to[0,\infty)$, and $\lambda_w$ is bounded and uniformly continuous.
\end{assumption}

\begin{assumption}[Envelope alternatives]\label{ass:envelope}
We use two envelope regimes.  The main unnormalized theorem assumes
\textup{(E1)}, whereas the normalized theory of \Cref{sec:normalized}
uses \textup{(E2)}.
\begin{enumerate}[label=\textup{(E\arabic*)}]
\item \textbf{Bounded weighted mass:} there is $M_w<\infty$ such that
$\mu_i^w(\OmegaL)\leq M_w$ almost surely.
\item \textbf{Normalized weighted diagrams:} replace $\mu_i^w$ by the probability measure $\widetilde\mu_i^w$ defined in \Cref{sec:normalized}, whose total mass is one automatically.
\end{enumerate}
\end{assumption}

\begin{remark}[Bounded weighted mass from total persistence]\label{rem:mass-from-pers}
Write
\[
\Pers_\tau(D):=\sum_{x\in D}a_x\pers(x)^\tau.
\]
If $w(u)\leq C_w\pers(u)^\tau$ and $\Pers_\tau(D_i)\leq M$ almost surely,
then $\mu_i^w(\OmegaL)\leq C_wM$. Thus \Cref{ass:envelope}(E1) follows from a
bounded $\tau$-total-persistence assumption, which is often more natural than
a bound on the raw number of diagram points. \Cref{sec:truncation} replaces
(E1) altogether by a moment condition.
\end{remark}

\section{Estimation theory}\label{sec:theory}

\subsection{Finite-sample uniform rate}\label{sec:supnorm-rate}

Let
\[
\omega_{\lambda_w}(r):=\sup\{|\lambda_w(z,u)-\lambda_w(z',u')|:\norm{z-z'}+\norm{u-u'}\leq r\}
\]
denote the joint modulus of continuity of the target. For $\delta\in(0,1)$, define the rate functions
\begin{align}
r_n(h_Z,h_U,\delta)&:=\sqrt{\frac{\log(C n/(\delta h_Z^d h_U^2))}{nh_Z^d h_U^2}}+\frac{\log(C n/(\delta h_Z^d h_U^2))}{nh_Z^d h_U^2},
\label{eq:rate-rn}\\
r_{Z,n}(h_Z,\delta)&:=\sqrt{\frac{\log(C n/(\delta h_Z^d))}{nh_Z^d}}+\frac{\log(C n/(\delta h_Z^d))}{nh_Z^d}.
\label{eq:rate-rzn}
\end{align}

\begin{theorem}[Finite-sample sup-norm rate]\label{thm:supnorm-rate}
Assume \crefrange{ass:weight}{ass:envelope}\textup{(E1)}. Then there is a constant $C_1<\infty$, depending only on the model constants and the kernels, such that, for every $\delta\in(0,1)$ and all sufficiently large $n$, with probability at least $1-\delta$,
\begin{equation}\label{eq:supnorm-rate}
\sup_{z\in\calZ,u\in\OmegaL}|\widehat\lambda_{w,n}(z,u)-\lambda_w(z,u)|\leq C_1(\omega_{\lambda_w}(h_Z+h_U)+r_n(h_Z,h_U,\delta)).
\end{equation}
Therefore, under \Cref{ass:kernels},
\[
\sup_{z\in\calZ,u\in\OmegaL}|\widehat\lambda_{w,n}(z,u)-\lambda_w(z,u)|\overset{P}{\to}0,
\qquad\text{and hence}\qquad
\mathcal{L}_{\mathrm{IU}}(\widehat\lambda_{w,n},\lambda_w)\overset{P}{\to}0.
\]
\end{theorem}

The proof, given in \Cref{app:supnorm}, follows a
numerator--denominator decomposition. A standard empirical-process
inequality for classes with polynomial covering numbers controls the
empirical numerator uniformly over $\calZ\times\OmegaL$; the required
envelope, variance, and entropy bounds are verified explicitly for the
random-measure response. The conditional second-moment bound follows
from bounded weighted mass, boundedness of the conditional mean density,
and the Cauchy--Schwarz inequality. The
persistence-plane normalization together with the population
Nadaraya--Watson denominator makes the deterministic ratio a genuine
local average of $\lambda_w$, even at the boundary, and an exact
centered ratio identity propagates the numerator and denominator
errors to the ratio.

\begin{corollary}[Anisotropic H\"older rate]\label{cor:holder-rate}
Suppose, in addition to the assumptions of \Cref{thm:supnorm-rate}, that there exist constants $s_Z,s_U\in(0,1]$ and $L_Z,L_U<\infty$ such that
\[
|\lambda_w(z,u)-\lambda_w(z',u')|\leq L_Z\norm{z-z'}^{s_Z}+L_U\norm{u-u'}^{s_U}
\]
for all $z,z'\in\calZ$ and $u,u'\in\OmegaL$. Then, for any bandwidths satisfying \Cref{ass:kernels},
\[
\sup_{z\in\calZ,u\in\OmegaL}|\widehat\lambda_{w,n}(z,u)-\lambda_w(z,u)|=O_P\left(h_Z^{s_Z}+h_U^{s_U}+\sqrt{\frac{\log n}{nh_Z^d h_U^2}}+\frac{\log n}{nh_Z^d h_U^2}\right).
\]
In particular, choosing
\begin{equation}\label{eq:holder-bandwidths}
h_Z\asymp\left(\frac{\log n}{n}\right)^{1/(2s_Z+d+2s_Z/s_U)},
\qquad
h_U\asymp\left(\frac{\log n}{n}\right)^{1/(2s_U+2+ds_U/s_Z)}
\end{equation}
gives
\[
\sup_{z\in\calZ,u\in\OmegaL}|\widehat\lambda_{w,n}(z,u)-\lambda_w(z,u)|=O_P\left[\left(\frac{\log n}{n}\right)^{1/(2+d/s_Z+2/s_U)}\right].
\]
In the isotropic case $s_Z=s_U=s$, this becomes
\[
\sup_{z\in\calZ,u\in\OmegaL}|\widehat\lambda_{w,n}(z,u)-\lambda_w(z,u)|=O_P\left[\left(\frac{\log n}{n}\right)^{s/(2s+d+2)}\right].
\]
\end{corollary}

The exponent $s/(2s+d+2)$ is the classical sup-norm exponent for nonparametric regression in $d+2$ effective dimensions \citep{tsybakov2009introduction,gine2016mathematical}: the covariate contributes $d$ dimensions and the persistence plane contributes two. For $d=0$ the rate matches the unconditional persistence-intensity rates of \citet{wu2024estimation}, and \Cref{sec:pot-rate} shows that the two persistence dimensions disappear altogether when the error is measured in partial optimal transport. For smoothness $s>1$, higher-order bias reduction would require a local-polynomial or higher-order-kernel modification, together with corresponding boundary correction. We do not pursue that extension because the nonnegative local-constant estimator \eqref{eq:nw-estimator} is the one that directly admits the measure-valued interpretation used in \Cref{sec:pot-rate}.

The next result shows that the exponent of \Cref{cor:holder-rate} is not an
artefact of the proof technique on the bounded-weighted-mass model underlying
the main uniform theorem. For $s\in(0,1]$ and $L_H>0$, let
$\norm{g}_\infty:=\sup_{z\in\calZ,u\in\OmegaL}|g(z,u)|$ for every bounded
function $g:\calZ\times\OmegaL\to\mathbb R$, and let
$\Sigma(s,L_H)$ denote the class of functions
$g:\calZ\times\OmegaL\to\mathbb R$ satisfying
\[
|g(z,u)-g(z',u')|\leq L_H\left(\norm{z-z'}+\norm{u-u'}\right)^s
\]
for all $z,z'\in\calZ$ and $u,u'\in\OmegaL$.

\begin{theorem}[Minimax lower bound under bounded weighted mass]\label{thm:lower-bound}
Fix $d\geq1$, $s\in(0,1]$, $L_H>0$, and $M_w>0$. Suppose that $w$ satisfies \Cref{ass:weight} and that $w(u)>0$ for every $u\in\OmegaLo$. Let $\mathcal P_{\mathrm{bm}}(s,L_H,M_w)$ be the class of joint laws $P$ of $(Z,D)$ satisfying the following conditions:
\begin{enumerate}
\item $Z$ is uniformly distributed on $\calZ=[0,1]^d$;
\item $D$ is a persistence diagram supported in $\OmegaLo$ and
\[
\mu_D^w(\OmegaL)\leq M_w
\]
almost surely;
\item the conditional weighted mean measure admits a density $\lambda_{w,P}$ belonging to $\Sigma(s,L_H)$.
\end{enumerate}
Then there exist constants $c,c_0>0$, depending only on $d$, $s$, $L_H$, $M_w$, $w$, and $L$, such that, for all sufficiently large $n$,
\[
\inf_{\widehat\lambda}\sup_{P\in\mathcal P_{\mathrm{bm}}(s,L_H,M_w)}P^{\otimes n}\left(\norm{\widehat\lambda-\lambda_{w,P}}_\infty\geq c\left(\frac{\log n}{n}\right)^{s/(2s+d+2)}\right)\geq c_0,
\]
where the infimum is over all measurable estimators based on $(Z_1,D_1),\ldots,(Z_n,D_n)$.
\end{theorem}

The lower bound is matched uniformly over this class by the estimator already
studied above. Indeed, fix kernels satisfying \Cref{ass:kernels}, choose
\[
h_Z\asymp h_U\asymp
\left(\frac{\log n}{n}\right)^{1/(2s+d+2)},
\qquad
\varrho_n:=\left(\frac{\log n}{n}\right)^{s/(2s+d+2)},
\]
and fix $\delta\in(0,1)$. Then there is $C_\delta<\infty$, depending
only on $d,s,L_H,M_w,w,L$, the kernels, and $\delta$, such that, for all
sufficiently large $n$,
\begin{equation}\label{eq:bounded-mass-uniform-upper}
\sup_{P\in\mathcal P_{\mathrm{bm}}(s,L_H,M_w)}
P^{\otimes n}\left(
\norm{\widehat\lambda_{w,n}-\lambda_{w,P}}_\infty
>C_\delta\varrho_n
\right)
\leq\delta.
\end{equation}
The uniformity of the upper-bound constants is verified in
\Cref{app:uniform-upper}. Consequently,
\Cref{thm:lower-bound} and \eqref{eq:bounded-mass-uniform-upper} establish
the minimax rate $\varrho_n$ in supremum norm over
$\mathcal P_{\mathrm{bm}}(s,L_H,M_w)$. The proof of the lower bound is in
\Cref{app:lower-bound}.

\subsection{Partial optimal transport rate for the same estimator}\label{sec:pot-rate}

The estimator \eqref{eq:nw-estimator} also defines the finite measure
\[
\widehat\Lambda_{z,n}^{w,h_U}(A)
:=\int_A\widehat\lambda_{w,n}(z,u)\dd u,
\qquad A\in\calB(\OmegaL),
\]
and therefore remains meaningful when $\Lambda_z^w$ has no Lebesgue
density. For $q\in[1,\infty)$, let $\adm(\nu,\xi)$ be the finite transport
plans whose off-diagonal marginals are $\nu$ and $\xi$, with unmatched mass
allowed to move to or from the diagonal $\DeltaL$, and define
\[
\OT_q(\nu,\xi)
:=\left\{\inf_{\pi\in\adm(\nu,\xi)}
\int_{\OmegaL\times\OmegaL}\|u-v\|^q\dd\pi(u,v)\right\}^{1/q}.
\]
We use the Euclidean ground norm; its relation to the customary max norm is
given in \Cref{rem:ground-norm}. Write
\[
\omega_{\OT,q}(r)
:=\sup_{\|z-z'\|\le r}\OT_q^q(\Lambda_z^w,\Lambda_{z'}^w),
\qquad
r_{U,h}(v):=\int_{\OmegaL}\kappa_{U,h}(u,v)\dd u,
\]
with $d_{\DeltaL,2}(v):=\dist_2(v,\DeltaL)=\pers(v)/\sqrt2$, set
\[
\mathfrak b_{U,q}(h)
:=\sup_{z\in\calZ}\int_{\OmegaL}|r_{U,h}(v)-1|
d_{\DeltaL,2}(v)^q\dd\Lambda_z^w(v).
\]
The first quantity controls covariate bias in transport, whereas the second
records the possible mass defect caused by boundary normalization of the
persistence kernel. Checkable sufficient conditions for both are collected in
\Cref{app:pot}; recall the covariate-only rate $r_{Z,n}$ from
\eqref{eq:rate-rzn}.

\begin{theorem}[Partial optimal transport rate for the double-kernel estimator]\label{thm:pot-rate}
Assume \Cref{ass:weight,ass:geometry-Z,ass:design,ass:envelope}\textup{(E1)}.
Assume the kernel conditions in \Cref{ass:kernels}, but replace the
density-estimation bandwidth condition by
\[
h_Z\downarrow0,\qquad h_U\downarrow0,
\qquad \frac{nh_Z^d}{\log n}\to\infty.
\]
Suppose that $\omega_{\OT,q}(r)\to0$ as $r\downarrow0$ and
$\mathfrak b_{U,q}(h_U)\to0$, and define
\[
\chi_q(N):=
\begin{cases}
1,&q>1,\\
\log(2+N),&q=1.
\end{cases}
\]
Then there is $C<\infty$, depending only on the model constants, the kernels,
$q$, and $L$, such that, for every fixed $\delta\in(0,1)$ and all sufficiently
large $n$, with probability at least $1-\delta$,
\begin{equation}\label{eq:pot-rate}
\sup_{z\in\calZ}\OT_q^q
\left(\widehat\Lambda_{z,n}^{w,h_U},\Lambda_z^w\right)
\le C\left\{
h_U^q+\mathfrak b_{U,q}(h_U)+\omega_{\OT,q}(h_Z)
+\frac{\chi_q(nh_Z^d)}{\sqrt{nh_Z^d}}
+r_{Z,n}(h_Z,\delta)
\right\}.
\end{equation}
In particular,
$\sup_{z\in\calZ}\OT_q(\widehat\Lambda_{z,n}^{w,h_U},\Lambda_z^w)
\overset P\longrightarrow0$.
\end{theorem}

The terms in \eqref{eq:pot-rate} respectively describe persistence smoothing,
boundary mass defect, covariate bias, empirical-measure fluctuation, and
design fluctuation. For $q>1$, the fourth term is dominated by $r_{Z,n}$;
for $q=1$, it yields at most one additional factor $(\log n)^{1/2}$.
Unlike the sup-norm rate, the stochastic part depends on the local number
$nh_Z^d$ of diagrams and not on the persistence-plane dimension.

\begin{corollary}[H\"older rate in partial optimal transport]\label{cor:pot-holder-rate}
Let $q>1$ and suppose the assumptions of \Cref{thm:pot-rate} hold. If, for
some $\beta>0$ and $L_{\OT}<\infty$,
\[
\OT_q^q(\Lambda_z^w,\Lambda_{z'}^w)
\le L_{\OT}\|z-z'\|^\beta
\qquad(z,z'\in\calZ),
\]
and $\mathfrak b_{U,q}(h_U)=O(h_U^q)$, then the choices
\[
h_Z\asymp\left(\frac{\log n}{n}\right)^{1/(2\beta+d)},
\qquad
h_U^q\lesssim
\left(\frac{\log n}{n}\right)^{\beta/(2\beta+d)}
\]
give
\[
\sup_{z\in\calZ}
\OT_q^q\left(\widehat\Lambda_{z,n}^{w,h_U},\Lambda_z^w\right)
=O_P\!\left[
\left(\frac{\log n}{n}\right)^{\beta/(2\beta+d)}
\right],
\]
or equivalently
\[
\sup_{z\in\calZ}
\OT_q\left(\widehat\Lambda_{z,n}^{w,h_U},\Lambda_z^w\right)
=O_P\!\left[
\left(\frac{\log n}{n}\right)^{\beta/(q(2\beta+d))}
\right].
\]
\end{corollary}

Thus the transport rate depends on the covariate dimension $d$, but not on
the two-dimensional persistence coordinate: persistence smoothing contributes
approximation error, but no factor $(nh_Z^dh_U^2)^{-1/2}$. This is the gain
from estimating the conditional mean measure rather than its density.

\subsection{Normalized weighted persistence density}\label{sec:normalized}

To describe the relative location rather than the expected amount of
persistence mass, consider the normalized target of
\citet{wu2024estimation}. Assume $w>0$ on $\OmegaLo$ and
$\mu_D^w(\overline{\Omega}_L)>0$ almost surely, and write
\[
\widetilde\mu_D^w
:=
\frac{\mu_D^w}{\mu_D^w(\overline{\Omega}_L)},
\qquad
\widetilde\Lambda_z^w(A)
:=
\E[\widetilde\mu_D^w(A)\mid Z=z]
=\int_A\widetilde\lambda_w(z,u)\dd u
\]
whenever the conditional mean measure has a Lebesgue density. Define
\[
\widetilde Y_{i,h_U}(u)
:=\int_{\OmegaL}\kappa_{U,h_U}(u,v)\dd\widetilde\mu_i^w(v)
\]
and estimate this density by
\begin{equation}\label{eq:normalized-nw-estimator}
\widehat{\widetilde\lambda}_{w,n}(z,u)
:=\frac{\sum_{i=1}^n
\kappa_{Z,h_Z}(z,Z_i)\widetilde Y_{i,h_U}(u)}
{\sum_{i=1}^n\kappa_{Z,h_Z}(z,Z_i)},
\end{equation}
with the same zero-denominator convention as in \eqref{eq:nw-estimator}.
Let $\omega_{\widetilde\lambda_w}$ denote the joint modulus of continuity
defined as in \Cref{sec:supnorm-rate}.

\begin{theorem}[Uniform normalized-density rate]
\label{thm:normalized-rate}
Assume \Cref{ass:kernels,ass:geometry-Z,ass:design} and that
$\widetilde\lambda_w$ is bounded and uniformly continuous. Then there is
$C<\infty$, depending only on the fixed model quantities, such that, for
every $\delta\in(0,1)$ and all sufficiently large $n$, with probability at
least $1-\delta$,
\[
\norm{
\widehat{\widetilde\lambda}_{w,n}
-\widetilde\lambda_w
}_\infty
\leq
C\left\{
\omega_{\widetilde\lambda_w}(h_Z+h_U)
+r_n(h_Z,h_U,\delta)
\right\}.
\]
Consequently, the left-hand side converges to zero in probability.
\end{theorem}

The proof is in \Cref{app:normalized}. If
$0<\Pp\{\mu_D^w(\overline{\Omega}_L)>0\}<1$, the result instead applies
under the retained law
$\Pp^+(\cdot)=\Pp(\cdot\mid\mu_D^w(\overline{\Omega}_L)>0)$, provided the
assumptions hold under $\Pp^+$. Under the H\"older conditions of
\Cref{cor:holder-rate}, it gives the same rates as the unnormalized estimator.

\subsection{Relaxing bounded mass and uniform design}
\label{sec:truncation}\label{sec:trimmed}

We record two extensions of \Cref{thm:supnorm-rate}: truncation replaces the
bounded-mass envelope, while denominator trimming permits a design density that
is not bounded away from zero.

\paragraph{Unbounded weighted mass.}
Set $M(D):=\mu_D^w(\OmegaL)$. For $M_n\uparrow\infty$, let
$\lambda_{w,n}^{\mathrm{tr}}$ be the density of
\[
A\longmapsto
\E\!\left[\mu_D^w(A)\ind{M(D)\leq M_n}\mid Z=z\right].
\]
\begin{proposition}[Moment alternative via truncation]\label{prop:truncation}
Assume \crefrange{ass:weight}{ass:first-moment}, but not
\Cref{ass:envelope}\textup{(E1)}. Suppose, for a constant $C_Y$ independent
of $n$, that
\[
n\Pp\{M(D)>M_n\}\to0,
\qquad
\frac{M_n\log n}{nh_Z^dh_U^2}\to0,
\qquad
\norm{\lambda_{w,n}^{\mathrm{tr}}-\lambda_w}_\infty\to0,
\]
and
\[
\sup_{z\in\calZ,u\in\OmegaL}
\E\!\left[(Y^w_{h_U}(u))^2\mid Z=z\right]
\leq C_Yh_U^{-2}.
\]
Then $\norm{\widehat\lambda_{w,n}-\lambda_w}_\infty\overset{P}{\to}0$.
\end{proposition}

\begin{proposition}[Conditionally Poisson diagrams]\label{prop:poisson-truncation}
Assume \crefrange{ass:weight}{ass:first-moment}, with $w>0$ on $\OmegaLo$.
Conditionally on $Z=z$, let $D$ be a Poisson random measure with possibly
sigma-finite intensity $\lambda_w(z,u)/w(u)$. If
\[
\frac{nh_Z^dh_U^2}{(\log n)^2}\to\infty,
\]
then $M_n=A\log n$, for every sufficiently large $A$, satisfies the
conditions of \Cref{prop:truncation}. Consequently,
$\norm{\widehat\lambda_{w,n}-\lambda_w}_\infty\overset{P}{\to}0$.
\end{proposition}

\paragraph{Vanishing design density.}
When $p_Z$ is not bounded away from zero, consider the integrated uniform loss
\begin{equation}\label{eq:integrated-loss}
\mathcal L_{\mathrm{IU}}(\widehat\lambda,\lambda_w)
:=
\int_\calZ\sup_{u\in\OmegaL}
|\widehat\lambda(z,u)-\lambda_w(z,u)|\dd P_Z(z)
\leq\norm{\widehat\lambda-\lambda_w}_\infty.
\end{equation}
Write
\[
\begin{aligned}
\widehat D_n(z)&:=\frac1n\sum_i\kappa_{Z,h_Z}(z,Z_i),
&D_h(z)&:=\E\widehat D_n(z),\\
\widehat N_n(z,u)&:=\frac1n\sum_i
\kappa_{Z,h_Z}(z,Z_i)Y^w_{i,h_U}(u).&&
\end{aligned}
\]
and, for $\rho_n\downarrow0$ and $B<\infty$, define
\[
\widehat\lambda^{\rho_n,B}_{w,n}(z,u)
:=\left\{\frac{\widehat N_n(z,u)}
{\widehat D_n(z)\vee\rho_n}\right\}\wedge B.
\]

\begin{proposition}[Trimmed integrated consistency]\label{prop:trimmed}
Assume \Cref{ass:weight,ass:kernels,ass:geometry-Z,ass:first-moment} and
\Cref{ass:envelope}\textup{(E1)}, and suppose $p_Z\leq C_Z<\infty$ on
$\calZ$. For every $\varepsilon>0$, suppose there exist a measurable
$A_\varepsilon\subset\calZ$, $m_{\varepsilon,n}>0$, and
$\delta_{\varepsilon,n}\downarrow0$ such that
\[
\begin{aligned}
P_Z(A_\varepsilon)&>1-\varepsilon,
&\inf_{z\in A_\varepsilon}D_h(z)&\geq m_{\varepsilon,n},\\
\frac{r_n(h_Z,h_U,\delta_{\varepsilon,n})}{m_{\varepsilon,n}}&\to0,
&\rho_n&\leq\frac{m_{\varepsilon,n}}2
\quad\text{eventually}.
\end{aligned}
\]
If $B>\norm{\lambda_w}_\infty$, then
\[
\mathcal L_{\mathrm{IU}}
(\widehat\lambda^{\rho_n,B}_{w,n},\lambda_w)
\overset{P}{\to}0.
\]
\end{proposition}

The truncation conditions, Poisson calculations, and proofs are given in
\Cref{app:truncation-details,app:truncation,app:poisson-truncation,app:trimmed}.
The local-mass condition permits $p_Z$ to vanish on negligible boundary
regions; rates require a quantitative version of this condition.

\subsection{Empirical verifications}\label{sec:empirical-verifications}

We examined two empirical implications of the theory in simulations with
analytically known conditional weighted intensities; \Cref{sec:simulations}
gives the complete designs and results. First, in the conditional Poisson
experiment summarized in \Cref{tab:empirical-verifications}, the relative
sup loss at $n=10^5$ was $3.3$--$4.1\%$ for $d=1$ and
$6.1$--$6.9\%$ for $d=2$, increasing to $15.1$--$17.0\%$ for $d=4$.
Thus accuracy deteriorates with the effective dimension $d+2$, while the
empirical exponents at the reported bandwidths remain positive. Severe
undersmoothing behaves differently: with $c_{\mathrm{bw}}=0.10$ and $d=4$,
the effective local kernel sample size is too small at the simulated sample
sizes and the fitted exponents are negative, consistent with the
effective-sample-size condition in \Cref{ass:kernels}.

\begin{table}[t]
\centering
\caption{Forward-study verification against analytically known conditional weighted intensities.
For each data-generating process (DGP) and covariate dimension $d$,
$c_{\mathrm{bw}}$ is the multiplier minimizing the mean integrated sup loss
at $n=10^5$ among $\{0.10,0.25,0.50,0.75\}$. Losses are averages over $100$ replicates. The
empirical exponent $\widehat r_\infty$ is minus the slope from regressing log
integrated sup loss on $\log n$ over $n\in\{10^2,10^3,10^4,10^5\}$; positive
values indicate decreasing loss.}
\label{tab:empirical-verifications}
\small
\begin{tabular}{lccccc}
\toprule
DGP & $d$ & $c_{\mathrm{bw}}$ & Int. sup & Rel. sup (\%) & $\widehat r_\infty$\\
\midrule
location & 1 & 0.75 & 0.320 & 3.5 & 0.385\\
         & 2 & 0.50 & 0.591 & 6.9 & 0.290\\
         & 4 & 0.50 & 1.383 & 16.5 & 0.178\\
mass     & 1 & 0.75 & 0.333 & 3.3 & 0.406\\
         & 2 & 0.50 & 0.623 & 6.1 & 0.316\\
         & 4 & 0.50 & 1.724 & 17.0 & 0.177\\
mixed    & 1 & 0.75 & 0.411 & 4.1 & 0.365\\
         & 2 & 0.50 & 0.694 & 6.3 & 0.291\\
         & 4 & 0.50 & 1.689 & 15.1 & 0.178\\
\bottomrule
\end{tabular}
\end{table}

Second, in a mixed-process experiment with $d=1$ and $n=1{,}000$, the
unbiased-risk criterion selected $c_{\mathrm{bw}}=0.50$, the oracle candidate
on the chosen grid, in all $12$ replicates. The selected and oracle fits
therefore had the same mean integrated squared error, $0.124$ (sd $0.028$),
compared with $0.574$ ($0.052$) at the neighbouring undersmoothed multiplier
$0.25$ and $0.190$ ($0.025$) at the oversmoothed multiplier $0.75$. This small,
coarse-grid experiment is not an oracle-efficiency study, but it verifies that
the criterion distinguishes the variance- and bias-dominated regimes.

\FloatBarrier
\section{Cerebral artery trees: age-associated changes in rooted morphology}
\label{sec:arteries}

We analyse the 98-subject cerebral-artery tree release distributed by
\citet{bendich2016persistent}, which derives from the magnetic-resonance angiography
study of \citet{bullitt2010effects} and its subsequent morphology-based quality control
\citep{aydin2011visualizing}. Age associations have been found repeatedly in this
benchmark archive. Our aim is to determine which topological events, as parametrized by
persistence diagrams, carry that association and to localize them on the observed blood
vessels. For each rooted artery tree, we compute extended persistence of Euclidean
distance from the supplied root, regress the resulting persistence measures continuously
on age, and retain the generating vertex of every pair for exact backmapping. Complete
preparation, tuning, inferential, and sensitivity details are in
\Cref{app:arteries-details}.

\subsection{Data, descriptor, and analysis}
\label{sec:arteries-data}
\label{sec:arteries-method}

The archive contains four labelled arterial systems for each of 98 non-pathological
subjects aged 19--79 years. Its deposited centreline graphs comprise 509 connected tree
components, each with a supplied root. For a component \(T\subset\mathbb R^3\) with root
\(r\), the Topological Morphology Descriptor \citep{kanari2018topological} is constructed
from
\[
    f(x)=\lVert x-r\rVert,\qquad x\in T.
\]
Because \(f\) need not vary monotonically along a branch, its local minima and maxima
record inward-to-outward and outward-to-inward radial reversals. Extended persistence
\citep{cohen2009extending} gives finite pairs for both fields. We display either type at
\(u=(q,p)\), where \(q=\min\{b,d\}\) is the lower endpoint and \(p=|d-b|\) is its
persistence; for a local maximum the generator radius is \(q+p\). There are
\(183{,}035\) local-minimum and \(185{,}592\) local-maximum pairs. The global
minimum--maximum class of each
component is analysed separately in \Cref{app:arteries-global}.

With persistence weight \(w(q,p)=p\), let \(\mathcal J_{isg}\) index the pairs of subject
\(i\), system \(s\), and field \(g\in\{\mathrm{min},\mathrm{max}\}\). We compare the raw
measure with persistence mass per 1,000 mm of observed centreline:
\[
\begin{aligned}
 \mu_{ig}^{\mathrm{raw}}
   &=\sum_{s=1}^4\sum_{j\in\mathcal J_{isg}}
       p_{isgj}\delta_{(q_{isgj},p_{isgj})},\\
 \mu_{ig}^{\mathrm{len}}
   &=\sum_{s=1}^4\frac{1000}{L_{is}}
       \sum_{j\in\mathcal J_{isg}}
       p_{isgj}\delta_{(q_{isgj},p_{isgj})},
\end{aligned}
\]
where \(L_{is}\) is the polyline length of system \(s\). Standardization is performed
within system before the four measures are summed; it changes mass but neither diagram
coordinates nor the measure into a probability distribution.

For \(m\in\{\mathrm{raw},\mathrm{len}\}\), we estimate the conditional intensity
\(\lambda_g^m(a,u)\), the density of \(\E[\mu_{ig}^m\mid A_i=a]\), and report the
upper-minus-lower-quartile contrast
\[
    \widehat\Delta_g^m(u)
    =\widehat\lambda_g^m(57,u)-\widehat\lambda_g^m(31.25,u).
\]
These ages are evaluation anchors, not fitted groups: all 98 subjects enter the
continuous-age regression. Diagram and age bandwidths are selected sequentially, after
which 4,999 complete-subject bootstrap resamples give a studentized 95\% band that is
simultaneous over a predeclared inference set and both finite fields. The selected
bandwidths, exact criteria, inference set, and fixed-bandwidth bootstrap statistic are
reported in \Cref{app:arteries-bandwidths,app:arteries-inference,tab:arteries-primary}.

\subsection{Results and spatial interpretation}
\label{sec:arteries-results}
\label{sec:arteries-spatial}

For the raw measures, both finite fields have broad negative selected regions
(\Cref{fig:arteries-primary}). Between ages 31.25 and 57, fitted feature count decreases
by 25.8\% in the principal local-minimum region and by 24.3\% in the principal
local-maximum region; fitted persistence mass decreases by 27.5\% and 33.0\%,
respectively. These contrasts combine changes in reversal geometry with differences in
the amount of reconstructed centreline.

After system-length standardization, no region is selected for local minima. The
local-maximum mask has five components; its well-supported principal component R1 is
localized at
\[
    71.2\le q\le85.9\ \mathrm{mm},
    \qquad 0\le p\le14.7\ \mathrm{mm}.
\]
Within R1, fitted feature count per unit system length changes from 115.83 to 100.51
(-13.2\%) and fitted persistence mass from 32.85 to 27.41 (-16.6\%). Thus a localized
decrease among outward-to-inward radial reversals remains after subjects are compared at
a common observed centreline length. The four smaller components and their support are
given in \Cref{app:arteries-secondary}.

\begin{figure}[tbp]
\centering
\includegraphics[width=\textwidth]{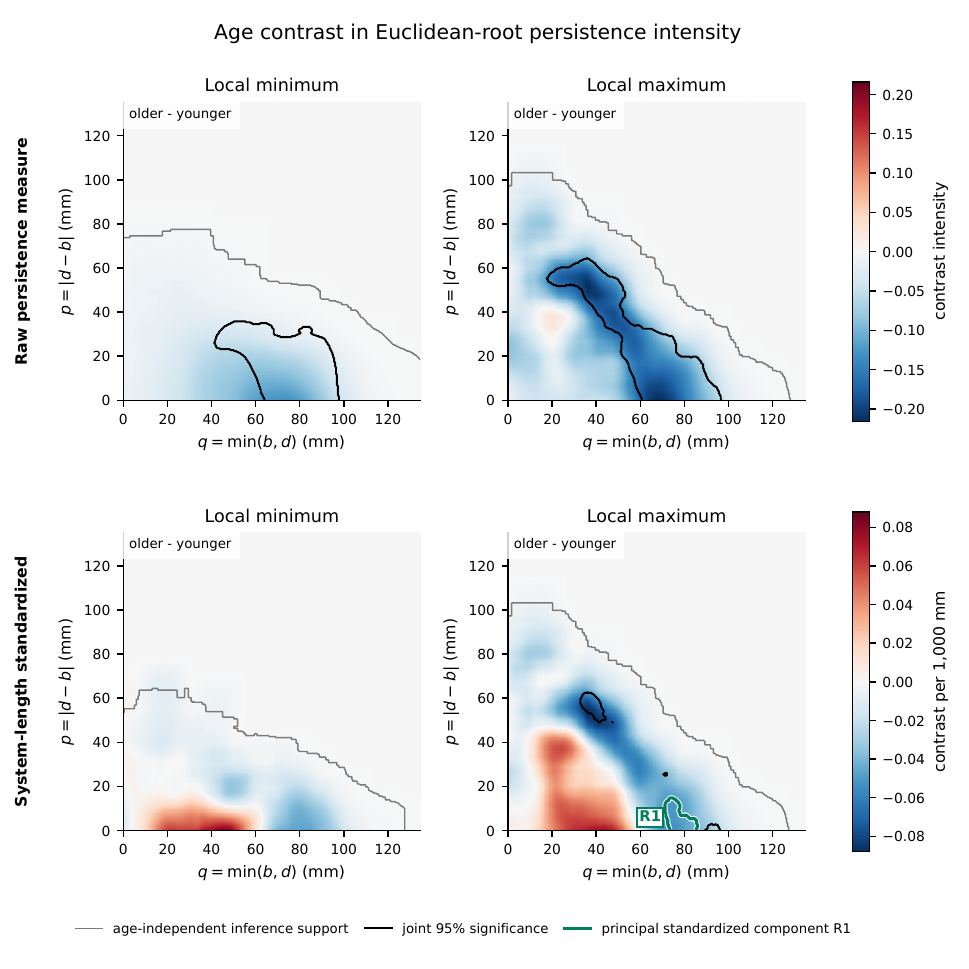}
\caption{Age-\(57\) minus age-\(31.25\) conditional weighted persistence-intensity contrasts.
Columns show local-minimum and local-maximum classes; rows show raw and
system-length-standardized measures. Grey contours delimit the age-independent inference
set, black contours delimit locations selected by the joint 95\% band, and the dark-green
contour identifies R1.}
\label{fig:arteries-primary}
\end{figure}

To assess whether R1 can be recovered from  radius alone, we replace each
local-maximum point \((q,p)\) by the one-dimensional coordinate \(r=q+p\), the Euclidean
distance of its generating maximum from the supplied root, while keeping the same
persistence weight \(w(q,p)=p\) used in the main analysis. Thus only the location of each
feature is projected from two dimensions to one. The resulting analysis selects no radius
interval at \(90\%\), \(95\%\), or \(99\%\), under either weighting scheme. The
two-dimensional R1 contrast is therefore not recovered by a radius-only inferential
summary. Radius-matched and
neighbouring-bandwidth diagnostics are reported in
\Cref{app:arteries-secondary,app:arteries-robustness,app:arteries-backmapping}.

The global minimum--maximum classes yield a complementary one-dimensional result.
Restricting the analysis to the largest-persistence component in each subject--system
gives exactly four global classes per subject. Relative to age \(31.25\), the fitted
unit-weighted intensity at age \(57\) is higher over maximum root radii
\(77.3\)--\(91.5\) mm and lower over \(97.9\)--\(104.6\) mm; both intervals remain
selected at \(99\%\). Because the number of classes is fixed, these contrasts describe a
redistribution of the fitted maximum-radius distribution rather than a change in class
count. The persistence-weighted analysis selects the same intervals and gives a \(4.0\%\)
decrease in fitted total global persistence. The complete analysis and its backmapping
figure are reported in \Cref{app:arteries-global,fig:arteries-global}.

In degree-zero persistence, each finite local-maximum class is generated at a
local-maximum vertex of the centreline. We retain the index of this vertex during the
persistence computation, so every selected class can be mapped exactly back to its
generator. Because the archive has
no common anatomical registration
\citep{aydin2009principal,bendich2016persistent}, we subtract each component's supplied
root but do not rotate, rescale, or anatomically align subjects; the resulting locations
are explicitly root-relative. R1 contains \(18{,}861\) observed generators concentrated
near an \(80\) mm root radius. Since R1 restricts both the lower merge level \(q\) and the
radial excursion \(p\), its generator radii \(q+p\) occupy a narrow interval, and a
band-like geometry is therefore expected for this particular selected region
(\Cref{fig:arteries-spatial}). This is not a generic consequence of the Euclidean-root
filtration: a selected region spanning a more articulated set of \((q,p)\) values, or
several substantially different ranges of \(q+p\), need not map to a layer-like spatial
geometry. A within-subject, radius-matched comparison still finds directional localization
within the observed vascular directions \((p_{\mathrm{MC}}=10^{-5})\), but this
post-selection diagnostic is not a test of an age-by-direction interaction.

The sign and broad location of the finite-field results persist under neighbouring
bandwidths, alternative age anchors and inference sets, and complete
leave-one-subject-out retuning; all four arterial systems contribute negatively in R1
when fitted separately at the pooled bandwidths
(\Cref{app:arteries-robustness}).

\begin{figure}[tbp]
\centering
\includegraphics[width=\textwidth]{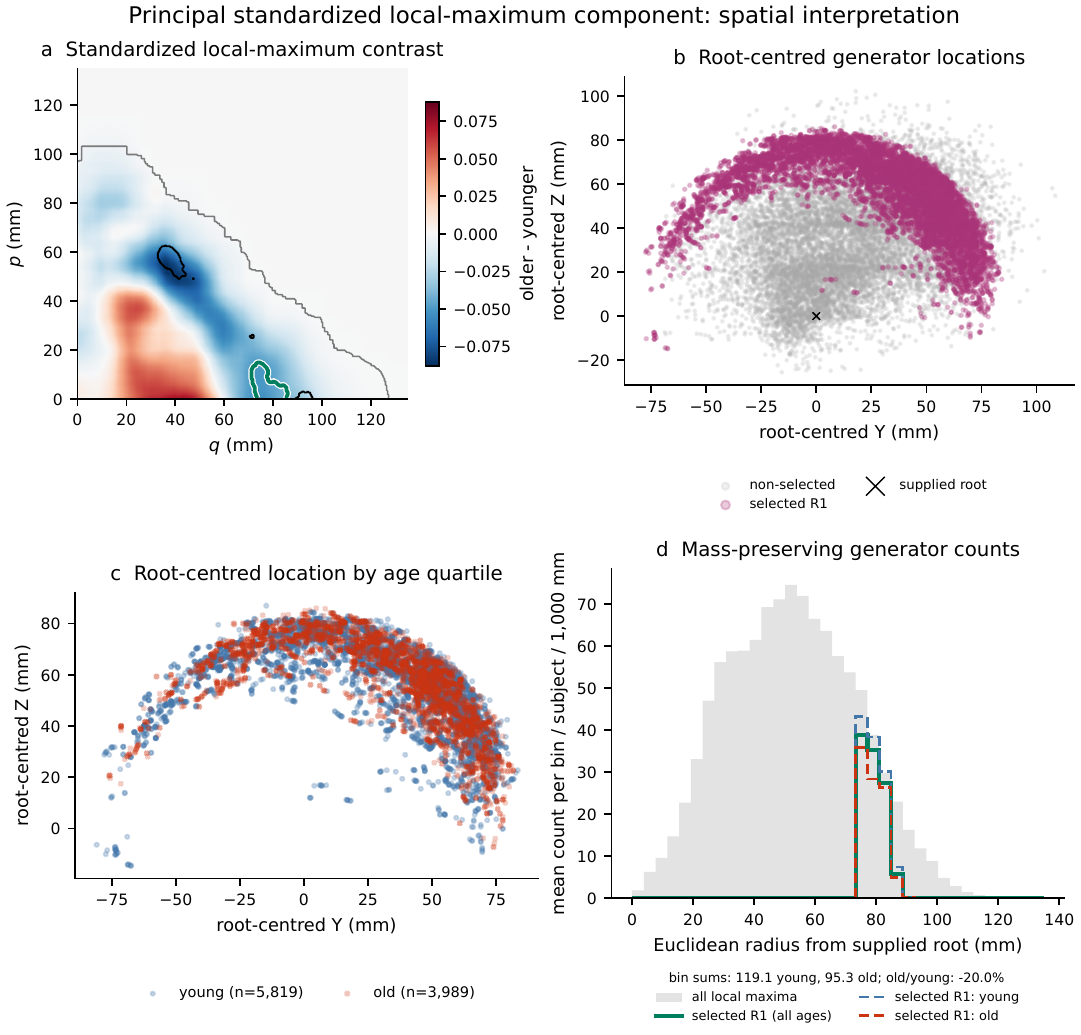}
\caption{Root-relative spatial interpretation of R1. Panels identify R1, compare its
generators with non-selected local maxima, display selected generators in the lower and
upper age quartiles, and show mass-preserving radial histograms. These are descriptive
backmaps of the region selected by the simultaneous persistence-plane analysis, not
atlas-anatomical localizations.}
\label{fig:arteries-spatial}
\end{figure}

The displayed young and old R1 histograms have standardized generator-count masses 119.1
and 95.3, a descriptive 20.0\% decrease. The selected generators are predominantly
ordinary degree-two centreline points rather than branch or terminal vertices. Together,
the persistence-plane contrast and exact backmapping identify a decrease in localized
outward-to-inward radial reversals per unit system length. Assigning the region to named
arterial territories would require anatomical registration or expert landmarks;
longitudinal data would be needed to track the same change within individuals.

\FloatBarrier

\subsection{Relation to previous analyses}
\label{sec:arteries-literature}

Earlier analyses of this archive found age associations in conventional vessel
attributes, combinatorial and metric tree structure, branching behaviour, functional
summaries, and persistence-based representations
\citep{bullitt2010effects,wang2007object,chang2013generalized,shen2014functional,bendich2016persistent,agerberg2025algebraic,matuk2024topogeometric}.
These studies establish a reproducible multiscale signal, but they do not combine
simultaneous selection in persistence space with exact generator-level localization on
the original centrelines. Our standardized result narrows that signal to ordinary
centreline reversals rather than simply to branch or terminal loss, while the global
classes retain a secondary redistribution of maximum root radius.

The peripheral radii found here are qualitatively consistent with the cortical-proximity
effect of \citet{skwerer2014tree}, but root-centred backmapping neither assigns a named
vessel nor establishes an angular age contrast. The closest geometric comparison is
\citet{guo2022statistical}, who found age effects after total-length normalization using
elastic shape graphs and visualized the edges with the largest fitted changes. That
framework requires cross-subject graph registration; ours compares within-subject radial
reversals in persistence space and backmaps them without asserting segment
correspondence. A fuller comparison with the benchmark literature is given in
\Cref{app:arteries-literature}.

\section{Discussion}\label{sec:discussion}

Statistical analyses of persistence diagrams have largely followed two complementary
routes. One treats diagrams as predictors, usually after mapping them into a vector or
function space; the other studies repeated diagrams from a single, unconditional
population. Much less attention has been given to settings in which the persistence
diagram is itself the response and is observed together with explanatory covariates.
This article develops a nonparametric framework for that problem. By representing each
weighted diagram as a finite random measure, we take its conditional mean measure as the
primary regression target and, when this measure is absolutely continuous, its density
as the conditional weighted persistence intensity. The resulting double-kernel
estimator smooths over both the covariate and persistence domains without specifying a
parametric law for the random diagram.

The theory clarifies both the statistical cost of this conditional formulation and the
role of the chosen target. For smooth conditional intensities, the estimator achieves a
uniform sup-norm rate with a matching minimax lower bound under bounded weighted mass
(\Cref{thm:lower-bound}). Direct estimation of the conditional mean measure in partial
optimal transport accommodates atoms and lower-dimensional support, and its stochastic
term depends on the covariate dimension rather than on the two-dimensional persistence
domain (\Cref{thm:pot-rate}). The normalized formulation provides a complementary
description of how persistence mass is distributed within a diagram, separately from
its total amount (\Cref{thm:normalized-rate}). For implementation, the unbiased
leave-one-out criterion in \Cref{prop:cv-risk} supplies a data-driven choice of the two
bandwidths. The empirical verifications behave consistently with these results: accuracy
deteriorates with effective dimension, severe undersmoothing becomes unstable when too
few diagrams receive appreciable local weight, and the cross-validation criterion
distinguishes the variance- and bias-dominated regimes
(\Cref{sec:empirical-verifications}).

The cerebral-artery analysis illustrates what is gained by making the diagram the
response rather than reducing it to a scalar summary. The raw analysis detects broad
age-associated reductions in persistence mass, whereas standardization by observed
centreline length isolates a more specific decrease among outward-to-inward radial
reversals. Generator-level backmapping then returns the selected region of the
persistence plane to the centreline events that produced it and shows that the retained signal is concentrated among ordinary centreline
reversals rather than simply among branch or terminal vertices. At the same time, the
interpretation must respect the limits of the data. The backmapped locations are
root-relative rather than atlas-anatomical, the study is cross-sectional, and the
reported contrasts describe age association rather than change within an individual or
a causal effect of ageing.

These results also delineate the scope of the framework. The conditional mean measure
is a first-order summary: it describes expected weighted persistence mass, but not the
full conditional law of a diagram or the dependence among its points. Note that, for marked
diagrams---including ordinary homology in several dimensions or the different fields of
extended persistence---the natural target is a collection
$\{\lambda_w^{(k)}(z,\cdot):k\in\calM\}$ and the present arguments apply separately to each
fixed mark. Joint modelling of dependence across marks, or inference on higher-order
features of the diagram distribution, however, requires additional targets and regularity
conditions.

Several directions for future work therefore arise naturally. On the theoretical side,
it remains to establish minimax lower bounds for direct estimation in partial optimal
transport and to determine whether the logarithmic factor at $q=1$ in
\Cref{thm:pot-rate} is intrinsic. A complementary inferential problem is to justify
fixed-bandwidth simultaneous bootstrap bands while explicitly accounting for the
diagonal boundary. For bandwidth selection, an oracle inequality or a Lepski-type
adaptation result would strengthen the unbiased-risk identity by quantifying the
performance of the selected estimator. On the modelling side, extensions to dependent
or longitudinal diagrams would be particularly valuable when several diagrams are
observed from the same unit. In the artery application, longitudinal imaging combined
with anatomical registration would make it possible to track changes within subjects
and to relate the detected radial reversals to named vascular territories. More broadly,
the same regression-and-backmapping strategy can be used in materials microstructure,
dynamical systems, biomedical shape analysis, and simulation ensembles whenever
persistence coordinates retain a direct scientific interpretation
\citep{wasserman2018topological}.

\section*{Supplementary materials}
\addcontentsline{toc}{section}{Supplementary materials}
The online supplementary materials contain \Cref{app:forward}, with the technical
arguments and proofs for the estimation theory; \Cref{sec:cv}, with bandwidth-selection
details and its proof; \Cref{app:simulation-results}, with the exact simulation designs
and numerical tables;
\Cref{app:arteries-details}, with the complete cerebral-artery implementation and
sensitivity analyses; and a code archive containing the forward-study integrated
and relative sup, $L^1$, and integrated squared error curves and reproducing the
estimators, bandwidth selection, tables, and figures.

\section*{Acknowledgements}
M.P. acknowledges support from the Fondo Istituzionale per la Ricerca of Università della Svizzera italiana through the project \emph{Using Topological Data Analysis to Understand Microglia Shape Variability in Space and Time}.

\section*{Disclosure statement}
The authors report there are no competing interests to declare.

\section*{Data availability statement}
All simulated data can be regenerated from the seeded scripts in the supplementary code
archive. The cerebral artery trees and the original analysis files are distributed with the
supplement to \citet{bendich2016persistent}; the accompanying code prepares the rooted
Euclidean TMD pairs and reproduces the conditional-intensity analysis reported here.

\appendix

\section[Proofs for the estimation theory]{Proofs for \Cref{sec:theory}}\label{app:forward}

We use Bernstein's inequality in the following form. If
$X_1,\ldots,X_n$ are independent and centered,
$|X_i|\leq B$, and
$n^{-1}\sum_i\E[X_i^2]\leq\sigma^2$, then, for every $t>0$,
\begin{equation}\label{eq:bernstein-form}
\Pp\left(
\left|\frac1n\sum_{i=1}^nX_i\right|
>
\sqrt{\frac{2\sigma^2t}{n}}
+
\frac{Bt}{3n}
\right)
\leq 2e^{-t}.
\end{equation}
Every bounded subset of $\R^m$ has an $\eta$-net of cardinality at most
$C\eta^{-m}$ for $0<\eta\leq1$. Thus, if a pointwise Bernstein bound fails
with probability at most $2e^{-t}$, a union bound with
$t=\log(2|\mathcal N|/\alpha)$ makes it simultaneous on a finite net
$\mathcal N$ with probability at least $1-\alpha$; a Lipschitz estimate
then extends it to the full parameter domain. We use this construction in
later scalar arguments. For the denominator and random-measure numerator
needed by \Cref{thm:supnorm-rate}, we use the following literature result.

\begin{proposition}[Uniform empirical-process bound]
\label{prop:uniform-empirical-process}
Let $X_1,\ldots,X_n$ be independent with common law $P$, let $(T,d_T)$ be
a separable metric space, and let
$\mathcal G=\{g_t:t\in T\}$ be a class of measurable, $P$-centered
functions such that $t\mapsto g_t(x)$ is continuous for every $x$.
Suppose that, for constants $B,\sigma>0$,
\[
\sup_{g\in\mathcal G}\norm{g}_\infty\leq B,
\qquad
\sup_{g\in\mathcal G}Pg^2\leq\sigma^2,
\qquad 0<\sigma\leq B,
\]
and that there are $A\geq e$ and $v_0>0$ such that, for every probability
measure $Q$ and every $0<\eta<B$,
\begin{equation}\label{eq:uniform-entropy-condition}
N\bigl(\mathcal G,L^2(Q),\eta\bigr)
\leq
\left(\frac{AB}{\eta}\right)^{v_0}.
\end{equation}
Here $N(\mathcal G,L^2(Q),\eta)$ is the minimum number of
$L^2(Q)$-balls of radius $\eta$ needed to cover $\mathcal G$.
Let $P_n:=n^{-1}\sum_{i=1}^n\delta_{X_i}$.
For $\delta\in(0,1)$, put
\[
H_{\mathcal G,\delta}
:=v_0\log(AB/\sigma)+\log(1/\delta).
\]
Then there is a universal constant $C$ such that, with probability at
least $1-\delta$,
\begin{equation}\label{eq:uniform-empirical-process-bound}
\sup_{g\in\mathcal G}|(P_n-P)g|
\leq
C\left(
\sqrt{\frac{\sigma^2H_{\mathcal G,\delta}}{n}}
+\frac{B H_{\mathcal G,\delta}}{n}
\right).
\end{equation}
The class, the common law $P$, and the constants $A,B,\sigma$ may depend on
$n$.
\end{proposition}

\begin{proof}
This is the direct combination of Theorems D.6 and D.7 in the supplementary
material of \citet{wu2024estimation}. Their Theorem D.6 controls the empirical
supremum by its expectation, a $\sigma\sqrt{\log(1/\delta)/n}$ term, and a
$B\log(1/\delta)/n$ term. Under
\eqref{eq:uniform-entropy-condition}, their Theorem D.7 bounds the
expectation by the corresponding square-root and linear terms containing
$v_0\log(AB/\sigma)$. Combining the two bounds gives
\eqref{eq:uniform-empirical-process-bound}. The cited results are
nonasymptotic, so they apply at each fixed $n$ even when the class,
constants, and common law vary with $n$.
\end{proof}

\subsection{Auxiliary observations and proofs}\label{app:preliminary-proofs}

\begin{observation}[Logarithmic bandwidth consequences]\label{obs:bandwidth-logs}
Whenever $nh_Z^d/\log n\to\infty$, the bound
\begin{equation}\label{eq:covariate-log-consequence}
\log(1/h_Z)
\leq \frac1d\log\!\left(\frac{n}{\log n}\right)
=O(\!\log n)
\end{equation}
holds for all sufficiently large $n$. Under \Cref{ass:kernels}, the stronger
joint bound
\begin{equation}\label{eq:bandwidth-log-consequence}
\log(1/h_Z)+\log(1/h_U)
\leq d\log(1/h_Z)+2\log(1/h_U)
\leq \log\!\left(\frac{n}{\log n}\right)
\leq \log n
\end{equation}
also holds for all sufficiently large $n$.
\end{observation}

\begin{proof}[Proof of \Cref{obs:bandwidth-logs}]
The first effective-sample-size condition implies
$h_Z^d\geq(\log n)/n$ eventually, which gives
\eqref{eq:covariate-log-consequence}. Under \Cref{ass:kernels}, similarly,
$h_Z^dh_U^2\geq(\log n)/n$ eventually. Taking logarithms gives
\eqref{eq:bandwidth-log-consequence}, once $h_Z,h_U\leq1$.
\end{proof}

\begin{observation}[Bounds for the normalized kernels]\label{obs:kernel-bounds}
Suppose that $K_U$ satisfies the kernel conditions imposed on it in
\Cref{ass:kernels}. Then the persistence-domain normalizing factor satisfies,
for all sufficiently small $h$,
\begin{equation}\label{eq:u-kernel-lb}
0<\underline c_U\leq c_{U,h}(u)\leq\overline c_U<\infty
\qquad \forall u\in\OmegaL.
\end{equation}
If, in addition, \Cref{ass:geometry-Z} holds and $K_Z$ satisfies the kernel
conditions imposed on it in \Cref{ass:kernels}, then
\begin{equation}\label{eq:z-kernel-lb}
0<c_0\leq c_{Z,h}(z)\leq C_0<\infty
\qquad \forall z\in\calZ.
\end{equation}
Therefore, the normalized kernels in \eqref{eq:normalized-kernels} are
well defined under the corresponding conditions and satisfy
\eqref{eq:kernel-normalization}. Under the same respective conditions, they
also obey, for all sufficiently small bandwidths,
\begin{equation}\label{eq:kernel-norm-bounds}
\begin{aligned}
\sup_{z\in\calZ}
\|\kappa_{Z,h_Z}(z,\cdot)\|_{L^\infty(\calZ)}
&\leq C h_Z^{-d},
&
\sup_{z\in\calZ}
\int_\calZ\kappa_{Z,h_Z}(z,z')^2\dd z'
&\leq C h_Z^{-d},\\
\sup_{u\in\OmegaL}
\|\kappa_{U,h_U}(u,\cdot)\|_{L^\infty(\OmegaL)}
&\leq C h_U^{-2},
&
\sup_{u\in\OmegaL}
\int_{\OmegaL}\kappa_{U,h_U}(u,v)^2\dd v
&\leq C h_U^{-2}.
\end{aligned}
\end{equation}
They satisfy the first-argument Lipschitz bounds
\begin{align}
|\kappa_{Z,h}(z,z'')-\kappa_{Z,h}(z',z'')|
&\leq C h^{-d-1}\norm{z-z'},\label{eq:kappa-z-lipschitz}\\
|\kappa_{U,h}(u,v)-\kappa_{U,h}(u',v)|
&\leq C h^{-3}\norm{u-u'}.\label{eq:kappa-u-lipschitz}
\end{align}
If the covariate-domain conditions just stated and \Cref{ass:design} hold,
then
\begin{equation}\label{eq:covariate-kernel-second-moment}
\sup_{z\in\calZ}
\int_\calZ\kappa_{Z,h_Z}(z,z')^2p_Z(z')\dd z'
\leq C h_Z^{-d}.
\end{equation}
\end{observation}

\begin{proof}[Proof of \Cref{obs:kernel-bounds}]
We first prove the normalizer bounds. Under the additional covariate-domain
conditions stated in the observation, by
\eqref{eq:covariate-volume-lower}, the change of variables
$y=(z'-z)/h$, and the lower bound for $K_Z$ in \Cref{ass:kernels},
\[
\begin{aligned}
c_{Z,h}(z)
&=\int_{(\calZ-z)/h}K_Z(y)\dd y\\
&\geq
k_Zh^{-d}\operatorname{Leb}_d\{B(z,h\rho_{K,Z})\cap\calZ\}
\geq k_Za_0\rho_{K,Z}^d
\end{aligned}
\]
uniformly in $z$, provided $h\rho_{K,Z}\leq r_0$. The fixed triangular domain
$\OmegaL$ is the closure of a bounded Lipschitz domain and therefore satisfies
the analogous uniform volume bound
\[
\operatorname{Leb}_2\{B(u,r)\cap\OmegaL\}\geq a_Ur^2
\qquad
\forall u\in\OmegaL,\quad 0<r\leq r_U^0
\]
for some $a_U,r_U^0>0$. The same calculation, now using the lower bound for
$K_U$, gives $c_{U,h}(u)\geq k_Ua_U\rho_{K,U}^2$ uniformly in $u$ for all sufficiently
small $h$. The two upper bounds follow from nonnegativity and the unit-integral
conditions, since each normalizer is the integral of its kernel over a subset
of the corresponding ambient space. This proves
\Cref{eq:u-kernel-lb,eq:z-kernel-lb}; the normalization identities
then follow directly from the definitions.

For the norm bounds, consider either normalized kernel. Let $\mathcal X$ denote
its domain, let $m$ be the dimension of $\mathcal X$, let $K$ be the unscaled
kernel, and let $c_h(x)$ be the normalizing factor. The bounds just proved give
$c_h(x)\geq\underline c>0$ uniformly in $x\in\mathcal X$. Hence, directly from
the definition of $\kappa_h$,
\[
\|\kappa_h(x,\cdot)\|_{L^\infty(\mathcal X)}
\leq
\frac{h^{-m}}{c_h(x)}\|K\|_{L^\infty(\R^m)}
\leq
\underline c^{-1}\|K\|_{L^\infty(\R^m)}h^{-m}.
\]
For the squared $L^2$ norm, we similarly obtain
\[
\begin{aligned}
\int_{\mathcal X}\kappa_h(x,x')^2\dd x'
&=
\frac{h^{-2m}}{c_h(x)^2}
\int_{\mathcal X}
K\!\left(\frac{x'-x}{h}\right)^2\dd x'\\
&=
\frac{h^{-m}}{c_h(x)^2}
\int_{(\mathcal X-x)/h}K(y)^2\dd y\\
&\leq
\underline c^{-2}
\|K\|_{L^2(\R^m)}^2h^{-m},
\end{aligned}
\]
where the second equality uses the change of variables
$y=(x'-x)/h$, so that $\dd x'=h^m\dd y$. Taking
$(\mathcal X,m,K)=(\calZ,d,K_Z)$ and
$(\mathcal X,m,K)=(\OmegaL,2,K_U)$ proves
\eqref{eq:kernel-norm-bounds}.

It remains to establish the Lipschitz bounds. Put
$q_h(x,x'):=h^{-m}K((x'-x)/h)$. Lipschitz continuity of $K$ gives
\[
|q_h(x,x'')-q_h(x',x'')|
\leq \Lip(K)h^{-m-1}\norm{x-x'}.
\]
The difference is supported in the union of two balls of radius $h$.
Integrating the preceding bound over that union therefore gives
$|c_h(x)-c_h(x')|\leq Ch^{-1}\norm{x-x'}$. Since
$c_h\geq\underline c$,
\[
\left|\frac1{c_h(x)}-\frac1{c_h(x')}\right|
\leq Ch^{-1}\norm{x-x'}.
\]
Combining these bounds with $0\leq q_h\leq\|K\|_\infty h^{-m}$ yields
\[
|\kappa_h(x,x'')-\kappa_h(x',x'')|
\leq Ch^{-m-1}\norm{x-x'}.
\]
Taking $m=d$ and $m=2$ proves \eqref{eq:kappa-z-lipschitz} and
\eqref{eq:kappa-u-lipschitz}, respectively.

Finally, under \Cref{ass:design}, the design-density upper bound and the
covariate $L^2$ bound give
\[
\int_\calZ\kappa_{Z,h_Z}(z,z')^2p_Z(z')\dd z'
\leq C_Z\int_\calZ\kappa_{Z,h_Z}(z,z')^2\dd z'
\leq Ch_Z^{-d},
\]
which proves \eqref{eq:covariate-kernel-second-moment}.
\end{proof}

\subsection{Auxiliary lemmas}\label{app:lemmas}

\setcounter{lemma}{0}
\renewcommand{\thelemma}{A\arabic{lemma}}
\renewcommand{\theHlemma}{A\arabic{lemma}}

\begin{lemma}[Covariate denominator and local counts]\label{lem:denominator-concentration}
Fix $\delta\in(0,1)$ and assume \Cref{ass:geometry-Z,ass:design}. Suppose that
$K_Z$ satisfies the kernel conditions imposed on it in \Cref{ass:kernels}, and that
$h_Z\to0$ and
$nh_Z^d/\log n\to\infty$. Then, for all sufficiently large $n$, there is an
event
$\mathcal E_Z=\mathcal E_Z(h_Z,\delta)$ with probability at least $1-\delta$ such that
\[
\sup_{z\in\calZ}|\widehat D_n(z)-D_h(z)|\leq C r_{Z,n}(h_Z,\delta)
\qquad\text{and}\qquad
\sup_{z\in\calZ}\frac1n\sum_{i=1}^n\ind{\norm{Z_i-z}\leq 2h_Z}\leq C h_Z^d.
\]
Moreover, $\inf_{z\in\calZ}D_h(z)\geq c_Z$. Therefore, if $C r_{Z,n}(h_Z,\delta)\leq c_Z/2$, then on $\mathcal E_Z$,
\[
\inf_{z\in\calZ}\widehat D_n(z)\geq \frac{c_Z}{2}.
\]
\end{lemma}

\begin{proof}
We first establish the deterministic lower bound. By \Cref{ass:design},
$p_Z(z')\geq c_Z$ on $\calZ$, while $\kappa_{Z,h_Z}$ is nonnegative and
integrates to one in its second argument by \eqref{eq:kernel-normalization}.
Hence, for every
$z\in\calZ$,
\begin{align*}
D_h(z)
&=\int_\calZ\kappa_{Z,h_Z}(z,z')p_Z(z')\dd z'\\
&\geq c_Z\int_\calZ\kappa_{Z,h_Z}(z,z')\dd z'
=c_Z.
\end{align*}

We next apply \Cref{prop:uniform-empirical-process}.  For
$x\in\calZ$, define
\[
f_z(x):=\kappa_{Z,h_Z}(z,x),
\qquad
g_z(x):=f_z(x)-Pf_z,
\qquad z\in\calZ.
\]
The index set $\calZ$ is compact and hence separable, and
$z\mapsto g_z(x)$ is continuous for every $x$ by the Lipschitz property of
the normalized kernel.  The pointwise and second-moment bounds in
\Cref{obs:kernel-bounds} give
\begin{equation}\label{eq:denominator-class-bounds}
\sup_{z\in\calZ}\norm{g_z}_\infty
\leq B_Z:=Ch_Z^{-d},
\qquad
\sup_{z\in\calZ}Pg_z^2
\leq\sigma_Z^2:=Ch_Z^{-d}.
\end{equation}
Indeed, $Pf_z=D_h(z)\leq C_Z$ by the design upper bound and kernel
normalization, and
$Pg_z^2\leq Pf_z^2\leq Ch_Z^{-d}$ by
\eqref{eq:covariate-kernel-second-moment}.

It remains to verify the entropy condition.  The first-argument
Lipschitz bound \eqref{eq:kappa-z-lipschitz} and
$|P(f_z-f_{z'})|\leq P|f_z-f_{z'}|$ imply, uniformly in $x$,
\[
|g_z(x)-g_{z'}(x)|
\leq Ch_Z^{-d-1}\norm{z-z'}
\leq CB_Z\frac{\norm{z-z'}}{h_Z}.
\]
Consequently, a Euclidean net of $\calZ$ with mesh proportional to
$\eta h_Z/B_Z$ induces an $L^2(Q)$-cover of
$\mathcal G_Z:=\{g_z:z\in\calZ\}$ at radius $\eta$, for every probability
measure $Q$.  Since $\calZ$ is bounded, its cardinality satisfies
\begin{equation}\label{eq:denominator-entropy}
N(\mathcal G_Z,L^2(Q),\eta)
\leq
C\left(\frac{B_Z}{\eta h_Z}\right)^d
\leq
\left(\frac{A_ZB_Z}{\eta}\right)^d,
\qquad
A_Z:=C h_Z^{-1},
\end{equation}
after increasing $C$ so that $A_Z\geq e$.

Apply \Cref{prop:uniform-empirical-process} with
$v_0=d$, $B=B_Z$, and $\sigma=\sigma_Z$; for sufficiently small $h_Z$,
$\sigma_Z\leq B_Z$ after increasing the envelope constant. Since
\[
\log\!\left(\frac{A_ZB_Z}{\sigma_Z}\right)
\leq C\{1+\log(1/h_Z)\},
\]
the logarithmic factor in \eqref{eq:uniform-empirical-process-bound} is
bounded, for all sufficiently large $n$, by
$C\log\{Cn/(\delta h_Z^d)\}$.  Recalling that
$(P_n-P)f_z=\widehat D_n(z)-D_h(z)$, we conclude that, on an event of
probability at least $1-\delta$,
\begin{equation}\label{eq:denominator-empirical-process-conclusion}
\sup_{z\in\calZ}|\widehat D_n(z)-D_h(z)|
\leq C r_{Z,n}(h_Z,\delta).
\end{equation}

It remains to prove the local-count bound, which follows from the denominator
bound on the same event. Put $\rho_*:=\rho_{K,Z}\wedge1$. By the kernel
lower bound in \Cref{ass:kernels}, $K_Z(y)\geq k_Z$ whenever
$\norm{y}\leq\rho_*$. Moreover, $c_{Z,h_Z}(x)\leq1$ because $K_Z$ is
nonnegative and integrates to one. Therefore, for every $x\in\calZ$,
\begin{align}
\widehat D_n(x)
&=\frac1n\sum_{i=1}^n\kappa_{Z,h_Z}(x,Z_i)\nonumber\\
&\geq \frac{k_Zh_Z^{-d}}{n}
\sum_{i=1}^n\ind{\norm{Z_i-x}\leq\rho_*h_Z}.
\label{eq:local-count-from-denominator}
\end{align}
On the denominator-concentration event already constructed,
\[
\sup_{x\in\calZ}\widehat D_n(x)
\leq \sup_xD_h(x)+\sup_x|\widehat D_n(x)-D_h(x)|
\leq C_Z+C r_{Z,n}(h_Z,\delta)
\leq C
\]
for all sufficiently large $n$. Hence \eqref{eq:local-count-from-denominator}
gives
\begin{equation}\label{eq:small-ball-local-count}
\sup_{x\in\calZ}\frac1n\sum_{i=1}^n
\ind{\norm{Z_i-x}\leq\rho_*h_Z}
\leq C h_Z^d.
\end{equation}
For each $z\in\calZ$, rescale $B(z,2h_Z)\cap\calZ$ by $h_Z^{-1}$
and apply the elementary Euclidean covering bound recorded at the beginning of
the appendix. This gives a cover by at most $C=C(d,\rho_*)$ balls of
radius $\rho_*h_Z$ whose centres belong to $\calZ$. Summing \eqref{eq:small-ball-local-count} over such a cover yields
\[
\sup_{z\in\calZ}\frac1n\sum_{i=1}^n
\ind{\norm{Z_i-z}\leq2h_Z}
\leq C h_Z^d.
\]
Thus the denominator event itself may be taken as $\mathcal E_Z$. Finally, if
$C r_{Z,n}(h_Z,\delta)\leq c_Z/2$, then on $\mathcal E_Z$,
\[
\widehat D_n(z)
\geq D_h(z)-|\widehat D_n(z)-D_h(z)|
\geq c_Z-Cr_{Z,n}(h_Z,\delta)
\geq c_Z/2
\]
uniformly in $z$, as claimed.
\end{proof}

\begin{lemma}[Local effective sample size and Lipschitz weights]\label{lem:local-effective-sample}
On the event $\mathcal E_Z$ of \Cref{lem:denominator-concentration}, assume $C r_{Z,n}(h_Z,\delta)\leq c_Z/2$, and define
\[
a_i(z):=\frac{\kappa_{Z,h_Z}(z,Z_i)}{\sum_{j=1}^n\kappa_{Z,h_Z}(z,Z_j)}.
\]
Then
\[
\sup_{z\in\calZ}\sum_{i=1}^n a_i(z)^2\leq \frac{C}{n h_Z^d},
\qquad\text{equivalently}\qquad
\inf_{z\in\calZ}\Bigl(\sum_{i=1}^n a_i(z)^2\Bigr)^{-1}\geq cnh_Z^d.
\]
Moreover, whenever $\norm{z-z'}\leq h_Z$,
\[
\sum_{i=1}^n |a_i(z)-a_i(z')|\leq C\frac{\norm{z-z'}}{h_Z}.
\]
\end{lemma}

\begin{proof}
Fix $z\in\calZ$ and abbreviate
\[
k_i(z):=\kappa_{Z,h_Z}(z,Z_i),
\qquad
S(z):=\sum_{j=1}^n k_j(z)=n\widehat D_n(z).
\]
On $\mathcal E_Z$, \Cref{lem:denominator-concentration} gives
$S(z)\geq nc_Z/2$ uniformly in $z$. By \eqref{eq:kernel-norm-bounds},
$k_i(z)\leq C h_Z^{-d}$, and hence
\[
0\leq a_i(z)=\frac{k_i(z)}{S(z)}
\leq\frac{C}{nh_Z^d}.
\]
The $a_i(z)$ are nonnegative and sum to one. It follows term by term that
$a_i(z)^2\leq(\max_j a_j(z))a_i(z)$, and hence
\[
\sum_{i=1}^n a_i(z)^2
\leq \max_j a_j(z)\sum_{i=1}^na_i(z)
\leq\frac{C}{nh_Z^d}.
\]
Taking reciprocals proves the equivalent effective-sample-size statement.

We now compare the weight vectors at $z$ and $z'$. Directly adding and
subtracting $k_i(z')/S(z)$ gives
\begin{align*}
|a_i(z)-a_i(z')|
&\leq \frac{|k_i(z)-k_i(z')|}{S(z)}
 +k_i(z')\left|\frac1{S(z)}-\frac1{S(z')}\right|.
\end{align*}
After summing in $i$ and using
$|S(z)-S(z')|\leq\sum_i|k_i(z)-k_i(z')|$, we obtain
\begin{equation}\label{eq:weights-l1-intermediate}
\sum_{i=1}^n|a_i(z)-a_i(z')|
\leq
\frac{2}{S(z)\wedge S(z')}
\sum_{i=1}^n|k_i(z)-k_i(z')|.
\end{equation}

Suppose $\norm{z-z'}\leq h_Z$. Because $K_Z$ is supported in the unit ball,
$k_i(z)$ or $k_i(z')$ can be nonzero only if $Z_i$ is within $h_Z$ of $z$ or
$z'$. This union is contained in the ball of radius $2h_Z$ about $z$. The
local-count part of $\mathcal E_Z$ therefore shows that at most $Cnh_Z^d$
indices contribute to the sum in \eqref{eq:weights-l1-intermediate}. For every
such index, \eqref{eq:kappa-z-lipschitz} gives
\[
|k_i(z)-k_i(z')|
\leq C h_Z^{-d-1}\norm{z-z'}.
\]
Therefore,
\[
\sum_{i=1}^n|k_i(z)-k_i(z')|
\leq Cnh_Z^d h_Z^{-d-1}\norm{z-z'}
=Cn\frac{\norm{z-z'}}{h_Z}.
\]
Substitution into \eqref{eq:weights-l1-intermediate}, followed by
$S(z)\wedge S(z')\geq nc_Z/2$, proves
\[
\sum_{i=1}^n|a_i(z)-a_i(z')|
\leq C\frac{\norm{z-z'}}{h_Z}.
\]
\end{proof}

\begin{lemma}[Bias of the expected double-kernel ratio]
\label{lem:double-kernel-bias}
Let $(Z,\nu)$ be a random pair, where $\nu$ is a finite nonnegative
random measure on $\OmegaL$. Suppose that its conditional mean measure
\[
\Lambda_z^\nu(A)
:=
\E\left[\nu(A)\mid Z=z\right],
\qquad
A\in\calB(\OmegaL),
\]
admits a bounded uniformly continuous density
$\lambda_\nu:\calZ\times\OmegaL\to[0,\infty)$.

Define
\[
Y^\nu_{h_U}(u)
:=
\int_{\OmegaL}
\kappa_{U,h_U}(u,v)\dd\nu(v),
\]
\[
N_h^\nu(z,u)
:=
\E\left[
\kappa_{Z,h_Z}(z,Z)Y^\nu_{h_U}(u)
\right],
\]
and
\[
D_h(z)
:=
\E\left[
\kappa_{Z,h_Z}(z,Z)
\right].
\]
Under
\Cref{ass:kernels,ass:geometry-Z,ass:design},
\[
\sup_{z\in\calZ,u\in\OmegaL}
\left|
\frac{N_h^\nu(z,u)}{D_h(z)}
-
\lambda_\nu(z,u)
\right|
\leq
\omega_{\lambda_\nu}(h_Z+h_U).
\]
Moreover,
\[
\inf_{z\in\calZ}D_h(z)\geq c_Z,
\qquad
\sup_{z,u}
\left|
\frac{N_h^\nu(z,u)}{D_h(z)}
\right|
\leq
\norm{\lambda_\nu}_\infty,
\]
and
\[
\sup_{z,u}
\left|
N_h^\nu(z,u)
\right|
\leq
C_Z\norm{\lambda_\nu}_\infty.
\]
\end{lemma}

\begin{proof}
By the design lower bound and the normalization of the covariate
kernel,
\[
D_h(z)
=
\int_\calZ
\kappa_{Z,h_Z}(z,z')p_Z(z')\dd z'
\geq
c_Z
\int_\calZ
\kappa_{Z,h_Z}(z,z')\dd z'
=
c_Z.
\]
Similarly,
\[
D_h(z)\leq C_Z.
\]

Conditionally on $Z=z'$, the definition of the conditional mean
measure gives
\[
\E\left[
\left.
Y^\nu_{h_U}(u)
\right|Z=z'
\right]
=
\int_{\OmegaL}
\kappa_{U,h_U}(u,v)
\lambda_\nu(z',v)\dd v.
\]
Therefore, by conditioning on $Z$ and applying Tonelli's theorem,
\[
N_h^\nu(z,u)
=
\int_\calZ\int_{\OmegaL}
\kappa_{Z,h_Z}(z,z')p_Z(z')
\kappa_{U,h_U}(u,v)
\lambda_\nu(z',v)
\dd v\dd z'.
\]

Define
\[
Q_{z,u}(\dd z',\dd v)
:=
\frac{
\kappa_{Z,h_Z}(z,z')p_Z(z')
\kappa_{U,h_U}(u,v)
}{
D_h(z)
}
\dd z'\dd v.
\]
The normalization identities for the two kernels imply
\[
\begin{aligned}
\int_\calZ\int_{\OmegaL}
Q_{z,u}(\dd z',\dd v)
&=
\frac{1}{D_h(z)}
\left[
\int_\calZ
\kappa_{Z,h_Z}(z,z')p_Z(z')\dd z'
\right]
\left[
\int_{\OmegaL}
\kappa_{U,h_U}(u,v)\dd v
\right]\\
&=1.
\end{aligned}
\]
Thus $Q_{z,u}$ is a probability measure. Moreover, by compact support
of the kernels,
\[
\operatorname{supp}(Q_{z,u})
\subseteq
\left\{
(z',v):
\norm{z'-z}\leq h_Z,\
\norm{v-u}\leq h_U
\right\}.
\]
We therefore have the exact local-average representation
\[
\frac{N_h^\nu(z,u)}{D_h(z)}
=
\int_\calZ\int_{\OmegaL}
\lambda_\nu(z',v)
Q_{z,u}(\dd z',\dd v).
\]
It follows that
\[
\begin{aligned}
\left|
\frac{N_h^\nu(z,u)}{D_h(z)}
-
\lambda_\nu(z,u)
\right|
&\leq
\int_\calZ\int_{\OmegaL}
\left|
\lambda_\nu(z',v)-\lambda_\nu(z,u)
\right|
Q_{z,u}(\dd z',\dd v)\\
&\leq
\omega_{\lambda_\nu}(h_Z+h_U).
\end{aligned}
\]
The same probability-average representation gives
\[
\left|
\frac{N_h^\nu(z,u)}{D_h(z)}
\right|
\leq
\norm{\lambda_\nu}_\infty.
\]
Finally, since $D_h(z)\leq C_Z$,
\[
\left|
N_h^\nu(z,u)
\right|
\leq
C_Z\norm{\lambda_\nu}_\infty.
\]
\end{proof}

\begin{lemma}[Bounds for the kernelized weighted response]
\label{lem:kernelized-response-bounds}
Let
\[
Y^w_{h_U}(u)
:=
\int_{\OmegaL}\kappa_{U,h_U}(u,v)\dd\mu^w(v).
\]
Under \Cref{ass:kernels,ass:first-moment} and
\Cref{ass:envelope}\textup{(E1)}, for all sufficiently large $n$,
uniformly over $z\in\calZ$ and $u\in\OmegaL$,
\begin{gather*}
0
\leq
Y^w_{h_U}(u)
\leq
CM_wh_U^{-2}
\qquad\text{almost surely},
\\
0
\leq
\E\bigl[Y^w_{h_U}(u)\mid Z=z\bigr]
\leq
\norm{\lambda_w}_\infty,
\qquad
\E\bigl[Y^w_{h_U}(u)^2\mid Z=z\bigr]
\leq
Ch_U^{-2}.
\end{gather*}
If \Cref{ass:geometry-Z,ass:design} also hold, then, uniformly over
$z\in\calZ$ and $u\in\OmegaL$,
\[
\left|
\kappa_{Z,h_Z}(z,Z)Y^w_{h_U}(u)
\right|
\leq
CM_wh_Z^{-d}h_U^{-2}
\qquad\text{almost surely},
\]
and
\[
\E\left[
\left(
\kappa_{Z,h_Z}(z,Z)Y^w_{h_U}(u)
\right)^2
\right]
\leq
Ch_Z^{-d}h_U^{-2}.
\]
\end{lemma}

\begin{proof}
The persistence-kernel bounds collected in
\Cref{obs:kernel-bounds} give
\[
\kappa_{U,h_U}(u,v)\geq0,
\qquad
\sup_{u,v\in\OmegaL}\kappa_{U,h_U}(u,v)
\leq
Ch_U^{-2}.
\]
Since $\mu^w$ is a nonnegative measure and
$\mu^w(\OmegaL)\leq M_w$ almost surely under
\Cref{ass:envelope}\textup{(E1)}, we have, almost surely,
\[
\begin{aligned}
0
\leq
Y^w_{h_U}(u)
&=
\int_{\OmegaL}
\kappa_{U,h_U}(u,v)\dd\mu^w(v)\\
&\leq
Ch_U^{-2}\mu^w(\OmegaL)
\leq
CM_wh_U^{-2}.
\end{aligned}
\]

By the definition of the conditional mean measure,
\[
\E\bigl[Y^w_{h_U}(u)\mid Z=z\bigr]
=
\int_{\OmegaL}
\kappa_{U,h_U}(u,v)\lambda_w(z,v)\dd v.
\]
The kernel is nonnegative and integrates to one by
\eqref{eq:kernel-normalization}. Therefore,
\[
0
\leq
\E\bigl[Y^w_{h_U}(u)\mid Z=z\bigr]
\leq
\norm{\lambda_w}_\infty.
\]
For the conditional second moment, the Cauchy--Schwarz inequality with
respect to the finite measure $\mu^w$ gives
\[
\begin{aligned}
Y^w_{h_U}(u)^2
&=
\left[
\int_{\OmegaL}\kappa_{U,h_U}(u,v)\dd\mu^w(v)
\right]^2\\
&\leq
\mu^w(\OmegaL)
\int_{\OmegaL}\kappa_{U,h_U}(u,v)^2\dd\mu^w(v)\\
&\leq
M_w
\int_{\OmegaL}\kappa_{U,h_U}(u,v)^2\dd\mu^w(v)
\qquad\text{almost surely}.
\end{aligned}
\]
Taking conditional expectations, using the defining property of the
conditional mean measure, and then applying
\eqref{eq:kernel-norm-bounds}, we obtain
\begin{align}
\E\bigl[Y^w_{h_U}(u)^2\mid Z=z\bigr]
&\leq
M_w
\int_{\OmegaL}
\kappa_{U,h_U}(u,v)^2\lambda_w(z,v)\dd v\nonumber\\
&\leq
M_w\norm{\lambda_w}_\infty
\int_{\OmegaL}\kappa_{U,h_U}(u,v)^2\dd v
\leq Ch_U^{-2}.
\label{eq:kernel-response-second-moment}
\end{align}

Now suppose that \Cref{ass:geometry-Z,ass:design} also hold. The
covariate-kernel pointwise bound in
\eqref{eq:kernel-norm-bounds}, combined with the almost-sure response
bound above, gives
\[
\left|
\kappa_{Z,h_Z}(z,Z)Y^w_{h_U}(u)
\right|
\leq
CM_wh_Z^{-d}h_U^{-2}
\qquad\text{almost surely}.
\]

Finally, the tower property and
\eqref{eq:kernel-response-second-moment} yield
\begin{align*}
&\E\left[
\left(
\kappa_{Z,h_Z}(z,Z)Y^w_{h_U}(u)
\right)^2
\right]
\\
&\qquad=
\E\left[
\kappa_{Z,h_Z}(z,Z)^2
\E\bigl[Y^w_{h_U}(u)^2\mid Z\bigr]
\right]
\\
&\qquad\leq
Ch_U^{-2}
\E\left[\kappa_{Z,h_Z}(z,Z)^2\right]
\\
&\qquad\leq
Ch_Z^{-d}h_U^{-2},
\end{align*}
where the last inequality is
\eqref{eq:covariate-kernel-second-moment}.
\end{proof}

\begin{lemma}[Uniform numerator concentration with a variable envelope]
\label{lem:numerator-concentration}
For each $n$, let
$(Z_1,\nu_{1,n}),\ldots,(Z_n,\nu_{n,n})$ be independent and identically
distributed pairs, where the $\nu_{i,n}$ are finite nonnegative random measures
on $\OmegaL$ and the covariates have a density satisfying $p_Z\leq C_Z$ on
$\calZ$. Define
\[
Y^{\nu}_{i,n,h_U}(u)
:=\int_{\OmegaL}\kappa_{U,h_U}(u,v)\dd\nu_{i,n}(v),
\qquad
\widehat N_n^{\nu}(z,u)
:=\frac1n\sum_{i=1}^n
\kappa_{Z,h_Z}(z,Z_i)Y^{\nu}_{i,n,h_U}(u),
\]
and let $N_h^{\nu}:=\E\widehat N_n^{\nu}$. Assume
\Cref{ass:kernels,ass:geometry-Z}, and suppose that there are deterministic
numbers $M_n$ with $1\leq M_n\leq n$ such that
\[
\nu_{i,n}(\OmegaL)\leq M_n\quad\text{almost surely},
\qquad
\sup_{z\in\calZ,u\in\OmegaL}
\E\!\left[(Y^{\nu}_{i,n,h_U}(u))^2\mid Z_i=z\right]
\leq C_Yh_U^{-2}.
\]
Then, for every $\delta\in(0,1)$ and all sufficiently large $n$, with
probability at least $1-\delta$,
\begin{equation}\label{eq:variable-envelope-numerator}
\sup_{z\in\calZ,u\in\OmegaL}
|\widehat N_n^{\nu}(z,u)-N_h^{\nu}(z,u)|
\leq C\left(
\sqrt{\frac{L_{n,\delta}}{nh_Z^dh_U^2}}
+\frac{M_nL_{n,\delta}}{nh_Z^dh_U^2}
\right),
\qquad
L_{n,\delta}:=\log\frac{Cn}{\delta h_Z^dh_U^2}.
\end{equation}
Here $C$ depends only on the kernel, domain, design-upper-bound, and
second-moment constants, and not on $M_n$.
\end{lemma}

\begin{proof}
At fixed $n$, let $P$ denote the common law of $(Z_i,\nu_{i,n})$.  For
$x=(\zeta,\nu)$, define
\[
f_{z,u}(x)
:=
\kappa_{Z,h_Z}(z,\zeta)
\int_{\OmegaL}\kappa_{U,h_U}(u,v)\dd\nu(v),
\qquad
g_{z,u}:=f_{z,u}-Pf_{z,u}.
\]
Then $(P_n-P)g_{z,u}=\widehat N_n^\nu(z,u)-N_h^\nu(z,u)$.
The index set $\calZ\times\OmegaL$ is compact and hence separable, and
$(z,u)\mapsto g_{z,u}(x)$ is continuous for every $x$, because the
normalized kernels are Lipschitz and $\nu$ has finite mass.

The pointwise kernel bounds and $\nu(\OmegaL)\leq M_n$ give
$0\leq f_{z,u}\leq B_n:=CM_nh_Z^{-d}h_U^{-2}$.  Since $f_{z,u}$ and
$Pf_{z,u}$ belong to the same interval,
$\norm{g_{z,u}}_\infty\leq B_n$.  Moreover, conditioning on $Z_i$ and
using the assumed response second moment, $p_Z\leq C_Z$, and
\eqref{eq:kernel-norm-bounds}, we obtain
\begin{align}
Pg_{z,u}^2
&\leq Pf_{z,u}^2\nonumber\\
&\leq C_Yh_U^{-2}
\int_\calZ\kappa_{Z,h_Z}(z,z')^2p_Z(z')\dd z'\nonumber\\
&\leq Ch_Z^{-d}h_U^{-2}
=:\sigma_n^2.
\label{eq:numerator-class-bounds}
\end{align}

For the entropy bound, the normalized-kernel Lipschitz estimates and
$\nu(\OmegaL)\leq M_n$ give, uniformly over the observation $x$,
\begin{align*}
|f_{z,u}(x)-f_{z',u'}(x)|
&\leq
CM_n\left(
h_Z^{-d-1}h_U^{-2}\norm{z-z'}
+h_Z^{-d}h_U^{-3}\norm{u-u'}
\right)\\
&\leq
CB_n\left(
\frac{\norm{z-z'}}{h_Z}
+\frac{\norm{u-u'}}{h_U}
\right).
\end{align*}
Centering changes only the constant, since
$|P(f_{z,u}-f_{z',u'})|\leq P|f_{z,u}-f_{z',u'}|$.  Therefore, nets of
$\calZ$ and $\OmegaL$ with respective mesh widths proportional to
$\eta h_Z/B_n$ and $\eta h_U/B_n$ induce an $L^2(Q)$-cover of
$\mathcal G_n:=\{g_{z,u}:(z,u)\in\calZ\times\OmegaL\}$ at radius
$\eta$, for every probability measure $Q$.  The Euclidean covering
bounds in dimensions $d$ and $2$ yield
\begin{align}
N(\mathcal G_n,L^2(Q),\eta)
&\leq
C\left(\frac{B_n}{\eta h_Z}\right)^d
 \left(\frac{B_n}{\eta h_U}\right)^2\nonumber\\
&\leq
\left(\frac{A_nB_n}{\eta}\right)^{d+2},
\qquad
A_n:=C(h_Z^dh_U^2)^{-1/(d+2)},
\label{eq:numerator-entropy}
\end{align}
where $C$ may be increased so that $A_n\geq e$.

We may now apply \Cref{prop:uniform-empirical-process} with
$v_0=d+2$, $B=B_n$, and $\sigma=\sigma_n$.  For sufficiently large $n$,
$\sigma_n\leq B_n$, and \eqref{eq:numerator-entropy} gives
\begin{align*}
\log\!\left(\frac{A_nB_n}{\sigma_n}\right)
&\leq C\{1+\log M_n+\log(1/h_Z)+\log(1/h_U)\},\\
H_{\mathcal G_n,\delta}
&\leq C\log\!\left(\frac{Cn}{\delta h_Z^dh_U^2}\right)
=CL_{n,\delta},
\end{align*}
where the last inequality uses $M_n\leq n$ and $h_Z,h_U\leq1$ eventually.
Substitution of $B_n$ and $\sigma_n^2$ into
\eqref{eq:uniform-empirical-process-bound} now gives
\eqref{eq:variable-envelope-numerator}.  The constants in the entropy,
envelope, and variance bounds depend only on the quantities listed in the
statement and not on $M_n$.
\end{proof}

\subsection[Proof of the sup-norm theorem and H\"older-rate corollary]{Proof of \Cref{thm:supnorm-rate,cor:holder-rate}}\label{app:supnorm}

\begin{proof}[Proof of \Cref{thm:supnorm-rate}]
Set
\[
\lambda_h(z,u):=\frac{N_h(z,u)}{D_h(z)}.
\]
We control separately the deterministic smoothing bias
$\lambda_h-\lambda_w$ and the stochastic error
$\widehat\lambda_{w,n}-\lambda_h$.

\medskip
\noindent
\emph{Deterministic bias.}
Applying \Cref{lem:double-kernel-bias} with $\nu=\mu^w$ and $\lambda_\nu=\lambda_w$, we have
\begin{equation}\label{eq:theorem-bias-bound}
\sup_{z,u}|\lambda_h(z,u)-\lambda_w(z,u)|
\leq \omega_{\lambda_w}(h_Z+h_U).
\end{equation}
The same lemma gives
$\norm{\lambda_h}_\infty\leq\norm{\lambda_w}_\infty$.

\medskip
\noindent
\emph{Stochastic ratio error.}
By \Cref{lem:kernelized-response-bounds}, the kernelized response satisfies
the conditional second-moment bound
\[
\sup_{z\in\calZ,u\in\OmegaL}
\E\left[Y^w_{i,h_U}(u)^2\mid Z_i=z\right]
\leq Ch_U^{-2},
\]
with the fixed constants $M_w$ and $\norm{\lambda_w}_\infty$ absorbed
into $C$.  Thus all moment conditions needed for
\Cref{lem:numerator-concentration} have already been verified.

The mass-envelope hypothesis of
\Cref{lem:numerator-concentration} holds with $\nu_{i,n}=\mu_i^w$ and  $M_n=M_w\vee1$.
Applying that lemma with failure probability $\delta/2$, and using that
$M_w$ is fixed, gives an event $\mathcal E_N$ with probability at least $1 - \delta/2$ such that, on $\mathcal E_N$,
\begin{equation}
\label{eq:numerator-concentration}
\sup_{z\in\calZ,u\in\OmegaL}
\left|
\widehat N_n(z,u)-N_h(z,u)
\right|
\leq
C r_n(h_Z,h_U,\delta).
\end{equation}
Indeed, the square-root term in
\eqref{eq:variable-envelope-numerator} is the corresponding term in
$r_n$, while the fixed factor $M_w\vee1$ in the linear term is absorbed
into $C$; replacing $\delta/2$ by $\delta$ inside the logarithm only
changes the constant.

Applying \Cref{lem:denominator-concentration} with failure probability
$\delta/2$ similarly gives an event $\mathcal E_D$  such that, on $\mathcal E_D$,
\begin{equation}
\label{eq:denominator-concentration}
\sup_{z\in\calZ}
\left|
\widehat D_n(z)-D_h(z)
\right|
\leq
C r_{Z,n}(h_Z,\delta).
\end{equation}
Let $\mathcal E = \mathcal E_N\cap\mathcal E_D$. By a union bound, $\Pp(\mathcal E)\geq1-\delta$ and, on $\mathcal E$ both \eqref{eq:numerator-concentration} and
\eqref{eq:denominator-concentration} hold simultaneously.
Since $h_U\leq1$ for all sufficiently large $n$, comparison of the two rate
functions gives
\begin{equation}\label{eq:rz-dominated}
r_{Z,n}(h_Z,\delta)\leq C r_n(h_Z,h_U,\delta).
\end{equation}
Moreover, \Cref{ass:kernels} implies
$nh_Z^d/\log n\to\infty$, and therefore
$r_{Z,n}(h_Z,\delta)\to0$ for every fixed $\delta\in(0,1)$.
Thus, for all sufficiently large $n$, the final conclusion of
\Cref{lem:denominator-concentration} gives
$\inf_z\widehat D_n(z)\geq c_Z/2$ on $\mathcal E$. In particular, the
estimator is given by the ratio
$\widehat\lambda_{w,n}=\widehat N_n/\widehat D_n$ throughout
$\calZ\times\OmegaL$ on this event.

The ratio error is controlled by one exact identity. Since
$N_h=\lambda_hD_h$,
\[
\widehat\lambda_{w,n}-\lambda_h
=
\frac{(\widehat N_n-N_h)-\lambda_h(\widehat D_n-D_h)}
     {\widehat D_n}.
\]
Using the bounds from \Cref{lem:double-kernel-bias},
\eqref{eq:numerator-concentration}, \eqref{eq:denominator-concentration}, and
\eqref{eq:rz-dominated}, together with
$\inf_z\widehat D_n(z)\geq c_Z/2$, we obtain on $\mathcal E$
\begin{align*}
\sup_{z,u}|\widehat\lambda_{w,n}(z,u)-\lambda_h(z,u)|
&\leq \frac{2}{c_Z}
\left(
\sup_{z,u}|\widehat N_n-N_h|
+\norm{\lambda_h}_\infty\sup_z|\widehat D_n-D_h|
\right)\\
&\leq C r_n(h_Z,h_U,\delta).
\end{align*}
Combining this bound with \eqref{eq:theorem-bias-bound} proves
\eqref{eq:supnorm-rate}.

It remains to prove the convergence statements. Under \Cref{ass:kernels},
$h_Z+h_U\to0$, so uniform continuity of $\lambda_w$ gives
$\omega_{\lambda_w}(h_Z+h_U)\to0$. The same assumption and
\eqref{eq:bandwidth-log-consequence} imply
$r_n(h_Z,h_U,\delta)\to0$ for every fixed $\delta\in(0,1)$.
Thus, given $\varepsilon,\eta>0$, apply the finite-sample bound with
$\delta=\eta$; its right-hand side is smaller than $\varepsilon$ for all
sufficiently large $n$. Hence
\[
\norm{\widehat\lambda_{w,n}-\lambda_w}_\infty
\overset{P}{\longrightarrow}0.
\]
Finally,
\[
\mathcal L_{\mathrm{IU}}(\widehat\lambda_{w,n},\lambda_w)
\leq
\norm{\widehat\lambda_{w,n}-\lambda_w}_\infty,
\]
so convergence under the integrated uniform loss follows as well.
\end{proof}

\begin{proof}[Proof of \Cref{cor:holder-rate}]
For every pair $(z,u),(z',u')$, the assumed anisotropic H\"older condition gives
\[
|\lambda_w(z,u)-\lambda_w(z',u')|
\leq L_Z\norm{z-z'}^{s_Z}+L_U\norm{u-u'}^{s_U}.
\]
On the support of the smoothing weights in \Cref{lem:double-kernel-bias},
$\norm{z-z'}\leq h_Z$ and $\norm{u-u'}\leq h_U$. Therefore the bias part of
that lemma has the sharper anisotropic bound
\begin{equation}\label{eq:anisotropic-bias}
\sup_{z,u}|\lambda_h(z,u)-\lambda_w(z,u)|
\leq L_Zh_Z^{s_Z}+L_Uh_U^{s_U}.
\end{equation}

Fix a confidence level $\delta\in(0,1)$. By
\eqref{eq:bandwidth-log-consequence},
\[
\log\!\left(\frac{Cn}{\delta h_Z^dh_U^2}\right)=O(\log n).
\]
Combining this fact, \eqref{eq:anisotropic-bias}, and
\Cref{thm:supnorm-rate} gives
\begin{align}
\sup_{z,u}|\widehat\lambda_{w,n}(z,u)-\lambda_w(z,u)|
=O_P\left(
h_Z^{s_Z}+h_U^{s_U}
+\sqrt{\frac{\log n}{nh_Z^dh_U^2}}
+\frac{\log n}{nh_Z^dh_U^2}
\right).\label{eq:holder-prebalance}
\end{align}

Let
\[
a_n:=\frac{\log n}{n}
\quad\text{and}\quad
\varrho_n:=a_n^{1/(2+d/s_Z+2/s_U)}.
\]
Choose the bandwidths so that both bias terms have order $\varrho_n$:
\[
h_Z^{s_Z}\asymp\varrho_n,
\qquad
h_U^{s_U}\asymp\varrho_n.
\]
Equivalently,
\begin{align*}
h_Z
&\asymp\varrho_n^{1/s_Z}
=a_n^{1/(2s_Z+d+2s_Z/s_U)},\\
h_U
&\asymp\varrho_n^{1/s_U}
=a_n^{1/(2s_U+2+ds_U/s_Z)},
\end{align*}
which is exactly \eqref{eq:holder-bandwidths}. To verify the stochastic balance,
observe that
\[
h_Z^dh_U^2
\asymp\varrho_n^{d/s_Z+2/s_U}.
\]
Therefore,
\begin{align*}
\sqrt{\frac{a_n}{h_Z^dh_U^2}}
&\asymp a_n^{1/2}
\varrho_n^{-d/(2s_Z)-1/s_U}\\
&=\varrho_n^{1+d/(2s_Z)+1/s_U}
\varrho_n^{-d/(2s_Z)-1/s_U}
=\varrho_n,
\end{align*}
where the equality $a_n^{1/2}=\varrho_n^{1+d/(2s_Z)+1/s_U}$ follows directly
from the definition of $\varrho_n$. Squaring this relation gives
\[
\frac{a_n}{h_Z^dh_U^2}\asymp\varrho_n^2=o(\varrho_n).
\]
Thus the final term in \eqref{eq:holder-prebalance}, namely
$\log n/(nh_Z^dh_U^2)$, is $o(\varrho_n)$. Substitution into
\eqref{eq:holder-prebalance} proves the anisotropic rate.

These bandwidths satisfy \Cref{ass:kernels}: both tend to zero,
their logarithms are $O(\log n)$, and
\[
\frac{nh_Z^dh_U^2}{\log n}
=\frac{h_Z^dh_U^2}{a_n}
\asymp\varrho_n^{-2}\longrightarrow\infty.
\]
Finally, if $s_Z=s_U=s$, then
\[
\varrho_n
=\left(\frac{\log n}{n}\right)^{1/(2+(d+2)/s)}
=\left(\frac{\log n}{n}\right)^{s/(2s+d+2)},
\]
which is the stated isotropic rate.
\end{proof}

\subsection[Proof of the minimax lower bound]{Proof of \Cref{thm:lower-bound}}\label{app:lower-bound}

\begin{proof}[Proof of \Cref{thm:lower-bound}]
Write
\[
d_\star:=d+2.
\]
For each sufficiently small scale $h>0$, we construct laws
\[
P_0,P_1,\ldots,P_{M_h}
\]
in $\mathcal P_{\mathrm{bm}}(s,L_H,M_w)$, with corresponding conditional weighted intensities
\[
\lambda_{w,0},\lambda_{w,1},\ldots,\lambda_{w,M_h}.
\]
In view of the Kullback version of the multiple-testing theorem
\citep[Theorem~2.5]{tsybakov2009introduction}, it is enough to establish the
following two properties for a suitable radius $\varepsilon_h$:
\begin{enumerate}
\item the targets are pairwise separated,
\[
\norm{\lambda_{w,j}-\lambda_{w,k}}_\infty\geq2\varepsilon_h
\]
for every $0\leq j<k\leq M_h$;
\item the experiments are sufficiently difficult to distinguish,
\[
P_j^{\otimes n}\ll P_0^{\otimes n}
\]
and
\[
\frac{1}{M_h}\sum_{j=1}^{M_h}\operatorname{KL}\left(P_j^{\otimes n},P_0^{\otimes n}\right)\leq\alpha\log M_h
\]
for some fixed $\alpha\in(0,1/8)$.
\end{enumerate}
We shall obtain $\varepsilon_h\asymp h^s$ and $M_h\asymp h^{-d_\star}$.

\emph{Step 1: an interior region and a normalized local perturbation.}

Since $w$ is continuous and vanishes on $\DeltaL$, an interior point can be
chosen sufficiently close to the diagonal that its weight is smaller than
$M_w$. Strict positivity on $\OmegaLo$ and continuity then give a sufficiently
small closed ball $C\subset\OmegaLo$ such that
\[
0<\inf_{u\in C}w(u)\leq\sup_{u\in C}w(u)<M_w.
\]
Choose closed balls $B'$ and $B$ satisfying
\[
B'\subset\operatorname{int}(B),\qquad B\subset\operatorname{int}(C).
\]
Let $f_0\in C_c^\infty(\operatorname{int}(C))$ be a probability density that is strictly positive on $B$, and define
\[
m_0:=\inf_{u\in B}f_0(u)>0,\qquad w_B:=\inf_{u\in B}w(u)>0.
\]

Assume for the moment that there exists a function $\Phi\in C_c^\infty(\mathbb R^d\times\mathbb R^2)$ satisfying
\[
\operatorname{supp}(\Phi)\subset B_{d_\star}(0,1),\qquad \Phi(0,0)=1,
\]
and
\[
\int_{\mathbb R^2}\Phi(t,v)\dd v=0
\]
for every $t\in\mathbb R^d$. The last condition ensures that a perturbation by $\Phi$ preserves the integral of the conditional location density for each fixed covariate value. The existence of such a function is verified after the proof.

Choose a constant $L_\star>0$ such that
\[
L_\star\leq1,\qquad L_\star\norm{\Phi}_\infty\leq\frac{m_0}{2}.
\]
The function $u\mapsto w(u)f_0(u)$ is Lipschitz on the compact persistence window and is therefore $s$-H\"older. Let $H_0<\infty$ denote one of its $s$-H\"older constants.

There is also a constant $H_1<\infty$, independent of $h\in(0,1]$ and of the centre $x_0=(z_0,u_0)$, such that
\[
(z,u)\longmapsto h^sw(u)\Phi\left(\frac{z-z_0}{h},\frac{u-u_0}{h}\right)
\]
is $s$-H\"older with constant at most $H_1$. Indeed, if $x=(z,u)$ and $y=(z',u')$, smoothness and compact support of $\Phi$ give
\begin{align*}
&h^s\left|w(u)\Phi\left(\frac{x-x_0}{h}\right)-w(u')\Phi\left(\frac{y-x_0}{h}\right)\right|\\
&\leq\norm{w}_\infty[\Phi]_{C^s}\norm{x-y}^s+h^s\operatorname{Lip}(w)\norm{\Phi}_\infty\norm{u-u'}\\
&\leq H_1\left(\norm{z-z'}+\norm{u-u'}\right)^s,
\end{align*}
where the last inequality uses $s\leq1$ and compactness of the domain.

Finally, choose a fixed number $\vartheta\in(0,1)$ such that
\[
\vartheta\leq\frac12,\qquad \vartheta\left(H_0+L_\star H_1\right)\leq L_H.
\]

\emph{Step 2: construction of the conditional location densities.}

For a sufficiently small $h>0$, choose points
\[
x_j=(z_j,u_j)\in[2h,1-2h]^d\times B',\qquad j=1,\ldots,M_h,
\]
whose pairwise Euclidean distances are at least $3h$. They may be chosen so that
\[
M_h\geq c_\star h^{-d_\star}
\]
for a constant $c_\star>0$ independent of $h$. For sufficiently small $h$, the sets $B_{d_\star}(x_j,h)$ are pairwise disjoint and contained in $[0,1]^d\times B$.

Define
\[
f_j(z,u):=f_0(u)+L_\star h^s\Phi\left(\frac{z-z_j}{h},\frac{u-u_j}{h}\right),\qquad j=1,\ldots,M_h,
\]
and retain $f_0(z,u):=f_0(u)$ for the baseline alternative.

For every fixed $z$, the fibrewise zero-integral property of $\Phi$ gives
\begin{align*}
\int_{\OmegaL}f_j(z,u)\dd u
&=1+L_\star h^s\int_{\OmegaL}\Phi\left(\frac{z-z_j}{h},\frac{u-u_j}{h}\right)\dd u\\
&=1+L_\star h^{s+2}\int_{\mathbb R^2}\Phi\left(\frac{z-z_j}{h},v\right)\dd v\\
&=1.
\end{align*}
Moreover, on the support of the perturbation, $u\in B$ and hence $f_0(u)\geq m_0$. Since $h^s\leq1$,
\[
\left|L_\star h^s\Phi\left(\frac{z-z_j}{h},\frac{u-u_j}{h}\right)\right|\leq\frac{m_0}{2}.
\]
Thus every $f_j(z,\cdot)$ is a nonnegative probability density. All these densities are supported in $C$.

\emph{Step 3: realization by mixed binomial point processes.}

For $j=0,\ldots,M_h$, define $P_j$ as follows. First draw
\[
Z\sim\operatorname{Unif}([0,1]^d).
\]
Conditionally on $Z=z$, draw
\[
N\sim\operatorname{Bernoulli}(\vartheta).
\]
If $N=0$, set $D=\emptyset$. If $N=1$, draw a point $U$ with density $f_j(z,\cdot)$ and set
\[
D=\{U\},
\]
with multiplicity one. Thus, conditionally on $Z=z$, $D$ is a mixed binomial point process with a Bernoulli law on its number of points.

Since $U\in C$ whenever $N=1$,
\[
\mu_D^w(\OmegaL)=Nw(U)\leq\sup_{u\in C}w(u)<M_w
\]
almost surely. The conditional weighted mean measure satisfies, for every Borel set $A\subseteq\OmegaL$,
\begin{align*}
\mathbb E_j\left[\mu_D^w(A)\mid Z=z\right]
&=\vartheta\mathbb E_j\left[w(U)\mathbf 1_{\{U\in A\}}\mid Z=z,N=1\right]\\
&=\vartheta\int_Aw(u)f_j(z,u)\dd u.
\end{align*}
Its density is therefore
\[
\lambda_{w,j}(z,u)=\vartheta w(u)f_j(z,u).
\]
By the definitions of $H_0$ and $H_1$,
\[
[\lambda_{w,j}]_{C^s}\leq\vartheta\left(H_0+L_\star H_1\right)\leq L_H.
\]
Consequently,
\[
P_j\in\mathcal P_{\mathrm{bm}}(s,L_H,M_w)
\]
for every $j=0,\ldots,M_h$.

\emph{Step 4: pairwise distance between the target intensities.}

The scaled perturbations have pairwise disjoint supports. For two nonbaseline
alternatives, evaluation at $x_j=(z_j,u_j)$ makes the $j$th perturbation equal
to one and the other vanish. For the baseline and the $k$th alternative,
evaluate instead at $x_k$. Hence, for every $0\leq j<k\leq M_h$,
\[
\norm{\lambda_{w,j}-\lambda_{w,k}}_\infty\geq\vartheta L_\star w_Bh^s.
\]
Define
\[
\varepsilon_h:=\frac{\vartheta L_\star w_B}{2}h^s.
\]
Then
\[
\norm{\lambda_{w,j}-\lambda_{w,k}}_\infty\geq2\varepsilon_h
\]
for every $0\leq j<k\leq M_h$. This verifies the pairwise-separation
condition of the multiple-testing theorem.

\emph{Step 5: Kullback--Leibler divergence.}

Let $Q_{j,z}$ denote the conditional law of $D$ given $Z=z$ under $P_j$.
The probability of the empty diagram is $1-\vartheta$ under every alternative.
On the singleton part, the conditional location density is $f_j(z,\cdot)$.
The perturbation is supported in $B$, where $f_0\geq m_0$, and outside $B$
we have $f_j=f_0$. Thus $Q_{j,z}\ll Q_{0,z}$ and, using the convention that
the common zero part contributes zero,
\[
\operatorname{KL}(Q_{j,z},Q_{0,z})
=\vartheta\int_{\OmegaL}f_j(z,u)
\log\left(\frac{f_j(z,u)}{f_0(u)}\right)\dd u.
\]

For all $x,y>0$,
\[
x\log\left(\frac{x}{y}\right)-x+y\leq\frac{(x-y)^2}{y}.
\]
Since $f_j-f_0$ is supported in $B$ and integrates to zero, the divergence
may be restricted to $B$ and the integral of $-f_j+f_0$ may be added there.
Consequently,
\begin{align*}
\operatorname{KL}(Q_{j,z},Q_{0,z})
&\leq\vartheta\int_B
\frac{\left(f_j(z,u)-f_0(u)\right)^2}{f_0(u)}\dd u\\
&\leq\frac{\vartheta L_\star^2h^{2s}}{m_0}\int_B
\Phi\left(\frac{z-z_j}{h},\frac{u-u_j}{h}\right)^2\dd u.
\end{align*}

The distribution of $Z$ is the same under all alternatives. The chain rule for relative entropy therefore gives
\begin{align*}
\operatorname{KL}(P_j,P_0)
&=\int_{[0,1]^d}\operatorname{KL}(Q_{j,z},Q_{0,z})\dd z\\
&\leq\frac{\vartheta L_\star^2h^{2s}}{m_0}\int_{[0,1]^d}\int_{\OmegaL}\Phi\left(\frac{z-z_j}{h},\frac{u-u_j}{h}\right)^2\dd u\dd z\\
&=\frac{\vartheta L_\star^2}{m_0}h^{2s+d_\star}\int_{\mathbb R^{d_\star}}\Phi(y)^2\dd y.
\end{align*}
Thus, for a constant $C_\Phi<\infty$ independent of $h$ and $j$,
\[
\operatorname{KL}(P_j,P_0)\leq C_\Phi\vartheta L_\star^2h^{2s+d_\star}.
\]
Relative entropy is additive for independent products, so
\[
\operatorname{KL}\left(P_j^{\otimes n},P_0^{\otimes n}\right)\leq C_\Phi\vartheta L_\star^2nh^{2s+d_\star}.
\]

\emph{Step 6: choice of the resolution and application of Tsybakov's theorem.}

Set
\[
h_n^\star:=\left(\frac{\log n}{n}\right)^{1/(2s+d_\star)},\qquad h:=\kappa h_n^\star,
\]
where $\kappa\in(0,1]$ is a fixed constant to be chosen.

Since $M_h\geq c_\star h^{-d_\star}$, there is a constant $c_1>0$ such that
\[
\log M_h\geq c_1\log n
\]
for all sufficiently large $n$. On the other hand,
\[
nh^{2s+d_\star}=\kappa^{2s+d_\star}\log n.
\]
Consequently,
\[
\operatorname{KL}\left(P_j^{\otimes n},P_0^{\otimes n}\right)\leq C_\Phi\vartheta L_\star^2\kappa^{2s+d_\star}\log n.
\]

Fix $\alpha=1/16$. Choose $\kappa>0$ sufficiently small that
\[
C_\Phi\vartheta L_\star^2\kappa^{2s+d_\star}\leq\alpha c_1.
\]
Then, for all sufficiently large $n$,
\[
\frac{1}{M_h}\sum_{j=1}^{M_h}\operatorname{KL}\left(P_j^{\otimes n},P_0^{\otimes n}\right)\leq\alpha\log M_h.
\]

We may therefore apply the Kullback version of the multiple-testing theorem
\citep[Theorem~2.5]{tsybakov2009introduction} with
\[
Q_j=P_j^{\otimes n},\qquad \theta_j=\lambda_{w,j},\qquad d_\Theta(f,g)=\norm{f-g}_\infty,\qquad r=\varepsilon_h.
\]
The pairwise-distance condition was established in Step 4, absolute continuity was established in Step 5, and the average Kullback--Leibler condition was established above. Hence
\begin{align*}
&\inf_{\widehat\lambda}\sup_{P\in\mathcal P_{\mathrm{bm}}(s,L_H,M_w)}P^{\otimes n}\left(\norm{\widehat\lambda-\lambda_{w,P}}_\infty\geq\varepsilon_h\right)\\
&\geq\frac{\sqrt{M_h}}{1+\sqrt{M_h}}\left(1-\frac18-\sqrt{\frac{1}{8\log M_h}}\right).
\end{align*}
Since $M_h\to\infty$, the right-hand side is bounded below by a constant $c_0>0$ for all sufficiently large $n$.

Finally,
\[
\varepsilon_h=\frac{\vartheta L_\star w_B\kappa^s}{2}\left(\frac{\log n}{n}\right)^{s/(2s+d+2)}.
\]
The result follows with
\[
c:=\frac{\vartheta L_\star w_B\kappa^s}{2}.
\]
\end{proof}

\paragraph{Construction of the perturbation $\Phi$.}

Choose $\eta\in C_c^\infty(\mathbb R^d)$ satisfying
\[
\operatorname{supp}(\eta)\subset B_d(0,1/4),\qquad \eta(0)=1.
\]
Choose a nonnegative function $\psi_+\in C_c^\infty(\mathbb R^2)$ satisfying
\[
\operatorname{supp}(\psi_+)\subset B_2(0,1/8),\qquad \psi_+(0)=1.
\]
Choose $v_0\in\mathbb R^2$ with $\norm{v_0}=1/2$ and a nonzero nonnegative function $\psi_-\in C_c^\infty(\mathbb R^2)$ satisfying
\[
\operatorname{supp}(\psi_-)\subset B_2(v_0,1/8).
\]
The supports of $\psi_+$ and $\psi_-$ are disjoint. Define
\[
c_\psi:=\frac{\int_{\mathbb R^2}\psi_+(v)\dd v}{\int_{\mathbb R^2}\psi_-(v)\dd v},\qquad \psi:=\psi_+-c_\psi\psi_-.
\]
Then
\[
\psi(0)=1,\qquad \int_{\mathbb R^2}\psi(v)\dd v=0.
\]
Finally, set
\[
\Phi(t,v):=\eta(t)\psi(v).
\]
It follows that
\[
\Phi(0,0)=1
\]
and, for every $t\in\mathbb R^d$,
\[
\int_{\mathbb R^2}\Phi(t,v)\dd v=\eta(t)\int_{\mathbb R^2}\psi(v)\dd v=0.
\]
The support choices also ensure that $\operatorname{supp}(\Phi)$ is contained in the unit ball of $\mathbb R^{d+2}$. Thus $\Phi$ has all the properties assumed in the proof.

\subsection{Uniform upper bound over the minimax class}
\label{app:uniform-upper}

To see that the constants in \Cref{thm:supnorm-rate} can be chosen uniformly
over the class, let $V_L:=\operatorname{Leb}_2(\OmegaL)$. For every $P$ in
the class, every $z\in[0,1]^d$, and every $u\in\OmegaL$, nonnegativity,
the mass bound, and the H\"older condition give
\[
V_L\lambda_{w,P}(z,u)
\leq
\int_{\OmegaL}\lambda_{w,P}(z,v)\dd v
+L_HV_L\operatorname{diam}(\OmegaL)^s
\leq M_w+L_HV_L\operatorname{diam}(\OmegaL)^s.
\]
Here the conditional mass bound first holds for almost every $z$; the
H\"older continuity of $\lambda_{w,P}$ extends it to every $z$ in the support
of the uniform design.
All remaining model constants in that theorem are fixed over the class:
the design is uniform on the unit cube, the weighted-mass envelope is $M_w$,
and the joint bias is at most $L_H(h_Z+h_U)^s$. The displayed bandwidths
satisfy \Cref{ass:kernels}, and substitution in \eqref{eq:supnorm-rate}
gives \eqref{eq:bounded-mass-uniform-upper}.

\subsection[Details and proofs for partial optimal transport]
{Details and proofs for \Cref{sec:pot-rate}}\label{app:pot}

For completeness, $\pi\in\adm(\nu,\xi)$ means that $\pi$ is a finite
positive measure on $\OmegaL\times\OmegaL$ such that, for every Borel set
$A\subseteq\OmegaLo$,
\[
\pi(A\times\OmegaL)=\nu(A),
\qquad
\pi(\OmegaL\times A)=\xi(A),
\]
with no marginal constraint on $\DeltaL$.

\begin{remark}[Choice of ground norm]\label{rem:ground-norm}
The ground norm in $\OT_q$ is Euclidean, following the partial-transport
convention of \citet{divol2021estimation,wu2024estimation}; persistence-diagram
Wasserstein and bottleneck distances are also frequently defined with the max
norm \citep{cohen2007stability,skraba2020wasserstein}. Let $\OT_q^{(2)}$ and
$\OT_q^{(\infty)}$ denote the corresponding partial-transport distances.
Since their admissible transports agree and
\[
\norm{x}_\infty\leq\norm{x}\leq\sqrt{2}\norm{x}_\infty,
\qquad x\in\R^2,
\]
we have
\[
\OT_q^{(\infty)}(\nu,\xi)
\leq \OT_q^{(2)}(\nu,\xi)
\leq \sqrt{2}\OT_q^{(\infty)}(\nu,\xi).
\]
Moreover,
\[
d_{\DeltaL,2}(u)=\frac{\pers(u)}{\sqrt{2}},
\qquad
d_{\DeltaL,\infty}(u)=\frac{\pers(u)}{2}.
\]
Thus changing the ground norm alters only fixed constants, numerical distance
values, and possibly optimal plans; it leaves the topology, assumptions, rate
exponents, and bandwidth exponents unchanged. We retain the Euclidean norm to
match the cited partial-transport framework and the geometry used in the
smoothing argument. This comparison does not extend the results to the
bottleneck case $q=\infty$.
\end{remark}

\begin{lemma}[A sufficient condition for the transport modulus]\label{lem:ot-modulus}
Suppose the conditional mean measures admit densities $\lambda_w(z,\cdot)$
and that, for some $s_Z\in(0,1]$ and $L_1<\infty$,
\[
\int_{\OmegaL}|\lambda_w(z,u)-\lambda_w(z',u)|\dd u
\leq L_1\norm{z-z'}^{s_Z}
\qquad(z,z'\in\calZ).
\]
Then $\omega_{\OT,q}(r)\leq C_{q,L}L_1r^{s_Z}$ for all $r>0$. In
particular, the H\"older condition of \Cref{cor:pot-holder-rate} holds with
$\beta=s_Z$. The $L^1$ condition holds whenever $\lambda_w$ satisfies the
pointwise H\"older condition of \Cref{cor:holder-rate}.
\end{lemma}

\begin{proof}[Proof of \Cref{lem:ot-modulus}]
The coupling constructed in the proof of Wu, Kim and Rinaldo~\cite[Theorem~3.1]{wu2024estimation} gives, for finite measures on $\OmegaL$ with densities $f$ and $g$,
\[
\OT_q^q(f\dd u,g\dd u)\leq\int_{\OmegaL}d_{\DeltaL,2}(u)^q|f(u)-g(u)|\dd u.
\]
Apply this inequality with $f(u)=\lambda_w(z,u)$ and $g(u)=\lambda_w(z',u)$. Since $d_{\DeltaL,2}(u)\leq L/\sqrt{2}$,
\[
\OT_q^q(\Lambda_z^w,\Lambda_{z'}^w)\leq\left(\frac{L}{\sqrt{2}}\right)^q\int_{\OmegaL}|\lambda_w(z,u)-\lambda_w(z',u)|\dd u\leq\left(\frac{L}{\sqrt{2}}\right)^qL_1\norm{z-z'}^{s_Z}.
\]
Taking the supremum over $\norm{z-z'}\leq r$ proves the claim. Under the pointwise H\"older condition in \Cref{cor:holder-rate},
\[
|\lambda_w(z,u)-\lambda_w(z',u)|\leq L_Z\norm{z-z'}^{s_Z},
\]
and integration over $\OmegaL$ gives $L_1=L_Z\operatorname{Vol}(\OmegaL)$.
\end{proof}

\paragraph{Boundary normalization.}
The kernel $\kappa_{U,h}$ in \eqref{eq:normalized-kernels} is normalized over
the source variable $v$ for each evaluation point $u$. This reproduces
constant functions at the boundary, as required by the sup-norm argument, but
need not preserve mass when a source point $v$ is spread over target locations
$u$. The resulting defect is precisely $\mathfrak b_{U,q}(h)$ defined in
\Cref{sec:pot-rate}. A source-normalized smoothing kernel would instead have
$r_{U,h}\equiv1$ and zero defect.

The defect can arise only within order $h$ of the boundary of $\OmegaL$. Its
diagonal part is harmless under bounded weighted mass because
\[
\int_{\{d_{\DeltaL,2}(v)\leq Ch\}}
d_{\DeltaL,2}(v)^q\dd\Lambda_z^w(v)
\leq (Ch)^q\Lambda_z^w(\OmegaL).
\]
Thus only the artificial outer boundary
$\partial_{\mathrm{out}}\OmegaL:=\partial\OmegaL\setminus\DeltaL$ requires
control. A sufficient condition for
$\mathfrak b_{U,q}(h_{U,n})\to0$ along every sequence $h_{U,n}\downarrow0$
is
\[
\lim_{h\downarrow0}\sup_{z\in\calZ}
\int_{\{v\in\OmegaL:\dist(v,\partial_{\mathrm{out}}\OmegaL)\leq h\}}
d_{\DeltaL,2}(v)^q\dd\Lambda_z^w(v)=0.
\]

Recall the kernel-weight notation of \Cref{lem:local-effective-sample}:
$a_i(z)=\kappa_{Z,h_Z}(z,Z_i)/\sum_j\kappa_{Z,h_Z}(z,Z_j)$ on the event where the denominator is positive.

\begin{lemma}[Scalar local Nadaraya--Watson average]
\label{lem:scalar-local-average}
Let
\[
(Z_i,B_i),\qquad i=1,\ldots,n,
\]
be independent and identically distributed pairs, where the $Z_i$ have the
design distribution considered above and the scalar responses satisfy
\[
0\le B_i\le B
\qquad\text{almost surely},
\]
for some constant $B<\infty$. Suppose moreover that
\[
\sup_{z\in\calZ}\E[B_i\mid Z_i=z]\le b_*
\]
for some $b_*\ge0$. Under the assumptions of
\Cref{lem:denominator-concentration}, for every fixed
$\delta\in(0,1)$ and all sufficiently large $n$, with probability at least
$1-\delta$,
\[
\sup_{z\in\calZ}
\sum_{i=1}^n a_i(z)B_i
\le
C\bigl(
b_*+r_{Z,n}(h_Z,\delta)
\bigr),
\]
where $C$ may depend on $B$.
\end{lemma}

\begin{proof}
Define
\[
\widehat N_B(z)=\frac1n\sum_{i=1}^n\kappa_{Z,h_Z}(z,Z_i)B_i,
\qquad
N_B(z)=\E\widehat N_B(z).
\]
Conditioning on $Z_i$ and using nonnegativity gives
\[
N_B(z)\leq b_*D_h(z).
\]
Apply \Cref{prop:uniform-empirical-process} to the centered class generated by
\[
f_z(\zeta,b)=\kappa_{Z,h_Z}(z,\zeta)b,
\qquad z\in\calZ.
\]
The kernel bounds give an envelope $CBh_Z^{-d}$ and variance $CB^2h_Z^{-d}$. The first-argument Lipschitz bound for $\kappa_{Z,h_Z}$ gives, for every probability measure $Q$,
\[
N\left(\{f_z-Pf_z:z\in\calZ\},L^2(Q),\eta\right)\leq\left(\frac{CBh_Z^{-d-1}}{\eta}\right)^d.
\]
Therefore, with probability at least $1-\delta/2$,
\[
\sup_{z\in\calZ}|\widehat N_B(z)-N_B(z)|\leq C_Br_{Z,n}(h_Z,\delta).
\]
On the denominator event of Lemma A1, also taken with failure probability $\delta/2$,
\[
\inf_z\widehat D_n(z)\geq\frac{c_Z}{2},
\qquad
\sup_z\frac{D_h(z)}{\widehat D_n(z)}\leq C.
\]
Since $\sum_i a_i(z)B_i=\widehat N_B(z)/\widehat D_n(z)$,
\[
\sup_{z\in\calZ}\sum_{i=1}^na_i(z)B_i\leq C\left(b_*+r_{Z,n}(h_Z,\delta)\right).
\]
A union bound completes the proof.
\end{proof}

\begin{lemma}[Smoothing in partial optimal transport]
\label{lem:smoothing-pot}
Assume that $K_U$ satisfies the kernel conditions imposed in
\Cref{ass:kernels}. Let $\nu$ be a finite measure on $\OmegaL$ and define
\[
T_{U,h}\nu(A)
:=
\int_{\OmegaL}\int_A
\kappa_{U,h}(u,v)\dd u\dd\nu(v).
\]
Then, for all sufficiently small $h$,
\[
\OT_q^q(T_{U,h}\nu,\nu)
\leq
C\left(
h^q\nu(\OmegaL)
+
\int_{\OmegaL}
|r_{U,h}(v)-1|
d_{\DeltaL,2}(v)^q
\dd\nu(v)
\right),
\]
where $C$ depends only on $q$ and the kernel constants.
\end{lemma}

\begin{proof}
For $v\in\OmegaL$, define the finite measure
\[
K_{h,v}(\dd u):=\kappa_{U,h}(u,v)\dd u
\]
and write
\[
r(v):=K_{h,v}(\OmegaL)=r_{U,h}(v).
\]
Let $P_{\DeltaL}(u)$ denote the Euclidean projection of $u$ onto
$\DeltaL$, and abbreviate
\[
d(v):=d_{\DeltaL,2}(v).
\]

By the kernel support assumption,
\[
\operatorname{supp}(K_{h,v})
\subseteq
\{u\in\OmegaL:\norm{u-v}\leq h\}.
\]
Although \eqref{eq:kernel-normalization} normalizes the kernel in its second
argument, the pointwise kernel bound and the support condition imply
\begin{align*}
0\leq r(v)
&=
\int_{\OmegaL}\kappa_{U,h}(u,v)\dd u\\
&\leq
Ch^{-2}
\operatorname{Leb}
\{u\in\OmegaL:\norm{u-v}\leq h\}\\
&\leq C.
\end{align*}
Thus $r(v)$ is bounded uniformly in $v$ and $h$.

We now construct, for each $v$, a partial transport between the unit mass
$\delta_v$ and the kernel measure $K_{h,v}$.

If $0\leq r(v)\leq1$, set
\[
\pi_v
:=
\delta_v\otimes K_{h,v}
+
(1-r(v))
\delta_{(v,P_{\DeltaL}(v))}.
\]
The first term transports the mass $r(v)$ from $v$ to the kernel bump,
while the second deletes the remaining mass $1-r(v)$ at the diagonal.

If $r(v)>1$, set
\begin{align*}
\pi_v(\dd x,\dd y)
&:=
\delta_v(\dd x)\frac{K_{h,v}(\dd y)}{r(v)}\\
&\quad+
\left(1-\frac1{r(v)}\right)
\delta_{P_{\DeltaL}(y)}(\dd x)K_{h,v}(\dd y).
\end{align*}
The first term transports the one available unit at $v$ to the normalized
kernel bump, while the second creates the excess mass $r(v)-1$ from the
diagonal.

Indeed, for every Borel set $A\subseteq\OmegaLo$,
\[
\pi_v(A\times\OmegaL)=\delta_v(A),
\qquad
\pi_v(\OmegaL\times A)=K_{h,v}(A).
\]
More explicitly, if $r(v)\leq1$, the first marginal is exactly
$\delta_v$, while the second full marginal is
\[
K_{h,v}+(1-r(v))\delta_{P_{\DeltaL}(v)};
\]
the additional mass is therefore supported on $\DeltaL$. If $r(v)>1$,
the second marginal is exactly $K_{h,v}$, while the first full marginal is
\[
\delta_v+
\left(1-\frac1{r(v)}\right)
(P_{\DeltaL})_{\#}K_{h,v},
\]
whose additional mass, of total size $r(v)-1$, is likewise supported on
$\DeltaL$. Hence the required marginals agree with $\delta_v$ and
$K_{h,v}$ on $\OmegaLo$, which is precisely the admissibility condition
for partial transport.

Therefore, defining
\[
\pi:=\int_{\OmegaL}\pi_v\,\dd\nu(v),
\]
we obtain, for every Borel set $A\subseteq\OmegaLo$,
\[
\pi(A\times\OmegaL)
=
\int_{\OmegaL}\pi_v(A\times\OmegaL)\,\dd\nu(v)
=
\int_{\OmegaL}\delta_v(A)\,\dd\nu(v)
=
\nu(A),
\]
and
\begin{align*}
\pi(\OmegaL\times A)
&=\int_{\OmegaL}\pi_v(\OmegaL\times A)\,\dd\nu(v)\\
&=\int_{\OmegaL}K_{h,v}(A)\,\dd\nu(v)\\
&=\int_{\OmegaL}\int_A
\kappa_{U,h}(u,v)\,\dd u\,\dd\nu(v)\\
&=T_{U,h}\nu(A).
\end{align*}
Thus the first marginal of $\pi$, restricted to $\OmegaLo$, is $\nu$,
while its second marginal, restricted to $\OmegaLo$, is $T_{U,h}\nu$.
Any additional marginal mass introduced by the plans $\pi_v$ is supported
on $\DeltaL$, where the marginal constraints of partial optimal transport
are not imposed. Hence $\pi$ is an admissible partial transport from
$\nu$ to $T_{U,h}\nu$.

It remains to bound its cost. Suppose first that $0\leq r(v)\leq1$.
Since $K_{h,v}$ is supported in the ball of radius $h$ around $v$,
\begin{align*}
\int_{\OmegaL\times\OmegaL}
\norm{x-y}^q\dd\pi_v(x,y)
&=
\int_{\OmegaL}\norm{u-v}^qK_{h,v}(\dd u)
+
(1-r(v))d(v)^q\\
&\leq
r(v)h^q
+
(1-r(v))d(v)^q\\
&\leq
h^q+|r(v)-1|d(v)^q.
\end{align*}

Suppose now that $r(v)>1$. The cost of the directly transported part is
\[
\frac1{r(v)}
\int_{\OmegaL}\norm{u-v}^qK_{h,v}(\dd u)
\leq h^q.
\]
The remaining mass is created from the diagonal and has cost
\[
\left(1-\frac1{r(v)}\right)
\int_{\OmegaL}d(u)^qK_{h,v}(\dd u).
\]
The distance to the closed set $\DeltaL$ is $1$-Lipschitz. Hence, on the
support of $K_{h,v}$,
\[
d(u)\leq d(v)+\norm{u-v}\leq d(v)+h.
\]
Using
\[
(x+y)^q\leq2^{q-1}(x^q+y^q),
\]
we obtain
\begin{align*}
\left(1-\frac1{r(v)}\right)
\int_{\OmegaL}d(u)^qK_{h,v}(\dd u)
&\leq
C\left(1-\frac1{r(v)}\right)
r(v)\bigl(d(v)^q+h^q\bigr)\\
&=
C(r(v)-1)\bigl(d(v)^q+h^q\bigr).
\end{align*}
Since $r(v)$ is uniformly bounded,
\[
(r(v)-1)h^q\leq Ch^q.
\]
Thus, also when $r(v)>1$,
\[
\int_{\OmegaL\times\OmegaL}
\norm{x-y}^q\dd\pi_v(x,y)
\leq
C\left(
h^q+|r(v)-1|d(v)^q
\right).
\]

Integrating this pointwise estimate with respect to $\nu(\dd v)$ gives
\begin{align*}
\int_{\OmegaL\times\OmegaL}
\norm{x-y}^q\dd\pi(x,y)
&\leq
C\int_{\OmegaL}
\left(
h^q+|r(v)-1|d(v)^q
\right)
\dd\nu(v)\\
&=
C\left(
h^q\nu(\OmegaL)
+
\int_{\OmegaL}
|r_{U,h}(v)-1|
d_{\DeltaL,2}(v)^q
\dd\nu(v)
\right).
\end{align*}
The transpose of $\pi$ is an admissible transport from $T_{U,h}\nu$ to
$\nu$ and has the same cost. Taking the infimum over admissible transports
proves the claimed inequality.
\end{proof}

The next lemma adapts the multiscale construction in the supplementary
Lemma~4 of \citet{divol2021estimation} to arbitrary finite weighted measures
on the bounded persistence window.

\begin{lemma}[Multiscale partial-transport bound]\label{lem:multiscale}
Let $\nu,\xi$ be finite positive measures on $\OmegaL$, and let
$Q\supseteq\OmegaL$ be a fixed square of side $L$.  For each $k\geq0$, let
$\mathcal D_k$ be the standard nested dyadic partition of $Q$ into $4^k$
measurable squares of side $L2^{-k}$, with boundary points assigned according
to a fixed convention, and define
\[
\mathcal P_k
:=\{D\cap\OmegaL:D\in\mathcal D_k,\ D\cap\OmegaL\neq\varnothing\}.
\]
Then, for every $q\in[1,\infty)$ and every integer $k^*\geq0$,
\[
\OT_q^q(\nu,\xi)
\leq
C_{q,L}\left(
\sum_{k=0}^{k^*}2^{-kq}\sum_{c\in\mathcal P_k}|\nu(c)-\xi(c)|
\;+
2^{-k^*q}\bigl(\nu(\OmegaL)\wedge\xi(\OmegaL)\bigr)
\right).
\]
\end{lemma}

\begin{proof}
This is a direct specialization of Supplementary Lemma~4 of
\citet{divol2021estimation}, whose multiresolution argument is in turn based
on Lemma~6 of \citet{singh2018minimax}. Extend \(\nu\) and \(\xi\) by zero
from \(\OmegaL\) to \(Q\), use \(Q\) as the single bounded block in that
lemma, and take its nested partitions to be
\(\{\mathcal D_k\}_{k=0}^{k^*}\). Intersecting the cells with \(\OmegaL\)
does not change any cell mass, and every resulting cell has diameter at
most \(\sqrt{2}L2^{-k}\).

With \(S_k=\sum_{c\in\mathcal P_k}|\nu(c)-\xi(c)|\), the cited lemma
therefore gives
\[
\OT_q^q(\nu,\xi)
\leq C_{q,L}\left\{
2^{-k^*q}\bigl(\nu(\OmegaL)\wedge\xi(\OmegaL)\bigr)
+|\nu(\OmegaL)-\xi(\OmegaL)|
+\sum_{k=1}^{k^*}2^{-kq}S_k
\right\}.
\]
The result applies to arbitrary finite positive measures; its construction
uses only restrictions, rescaling, and couplings of submeasures, not
integer-valued atoms. Since
\(S_0=|\nu(\OmegaL)-\xi(\OmegaL)|\), the last display is the asserted bound,
after enlarging \(C_{q,L}\) to absorb the harmless dyadic index shift in the
cited formulation.
\end{proof}

\paragraph{Relation to unconditional empirical-measure rates.}
For comparison, specialize \citet[Theorem~1]{divol2021estimation} to zeroth
persistence moment and transport exponent $q$. If $N$ i.i.d.\ persistence
measures are supported in their fixed window $A_{L_0}$ and each has total mass
at most $M$, their equally weighted empirical mean $\bar\mu_N$ satisfies
\[
\E\OT_q^q(\bar\mu_N,\E\mu)
\leq CML_0^q\{N^{-1/2}+a_q(N)N^{-q}\},
\qquad
a_q(N)=
\begin{cases}
1,&q>1,\\
\log N,&q=1.
\end{cases}
\]
Our local estimator requires a conditional high-probability extension to
independent measures whose conditional laws may differ and whose weights depend
on the covariates. The next two lemmas establish this extension: for
$N_{\mathrm{eff}}(a(z))=(\sum_i a_i(z)^2)^{-1}\gtrsim nh_Z^d$, the mixture
$\sum_i a_i(z)\mu_i^w$ concentrates around its conditional mean in
$\OT_q^q$ at rate
$\chi_q(N_{\mathrm{eff}})N_{\mathrm{eff}}^{-1/2}$, uniformly in $z$.

\begin{lemma}[Conditional weighted empirical-measure bound]
\label{lem:weighted-epd}
Let $1\leq q<\infty$ and define
\[
\chi_q(N):=
\begin{cases}
1, & q>1,\\
\log(2+N), & q=1.
\end{cases}
\]
Let $\mathcal G$ be a sigma-field, and let
$\nu_1,\ldots,\nu_n$ be random finite measures on $\OmegaL$ that are
conditionally independent given $\mathcal G$. Assume that, almost surely,
\[
\nu_i(\OmegaL)\leq M,
\qquad i=1,\ldots,n.
\]
For each $i$, let $\Lambda_i$ denote the conditional mean measure of $\nu_i$
given $\mathcal G$, namely the $\mathcal G$-measurable random measure
satisfying
\[
\Lambda_i(A)
=
\E\bigl[\nu_i(A)\mid\mathcal G\bigr]
\]
for every Borel set $A\subseteq\OmegaL$.

Let $a_1,\ldots,a_n$ be nonnegative $\mathcal G$-measurable random
variables such that
$\sum_{i=1}^n a_i=1
$
almost surely, and define
\[
\nu_a:=\sum_{i=1}^n a_i\nu_i,
\qquad
\Lambda_a:=\sum_{i=1}^n a_i\Lambda_i,
\qquad
N_{\mathrm{eff}}(a):=
\left(\sum_{i=1}^n a_i^2\right)^{-1}.
\]
Then there exists a constant $C<\infty$, depending only on $q$, $M$, and
$L$, such that, for every $\epsilon\in(0,1)$, almost surely,
\[
\Pp\left(
\left.
\OT_q^q(\nu_a,\Lambda_a)
\leq
C\left\{
\frac{\chi_q(N_{\mathrm{eff}}(a))}
{\sqrt{N_{\mathrm{eff}}(a)}}
+
\sqrt{
\frac{\log(1/\epsilon)}
{N_{\mathrm{eff}}(a)}
}
\right\}
\,\right|\,\mathcal G
\right)
\geq
1-\epsilon.
\]
\end{lemma}

\begin{proof}
Throughout the proof we condition on $\mathcal G$, so that the weights
$a_i$ and the measures $\Lambda_i$ are fixed, whereas
$\nu_1,\ldots,\nu_n$ remain independent. We suppress the conditioning from
the notation.

For
\[
S_k:=
\sum_{c\in\mathcal P_k}
|\nu_a(c)-\Lambda_a(c)|,
\]
\Cref{lem:multiscale}, applied with an arbitrary terminal level $K$, gives
\[
\OT_q^q(\nu_a,\Lambda_a)
\leq
C\left(
\sum_{k=0}^K2^{-kq}S_k+2^{-Kq}M
\right),
\]
because both $\nu_a$ and $\Lambda_a$ have total mass at most $M$.
Since $S_k\leq2M$, letting $K\to\infty$ yields
\begin{equation}
\label{eq:weighted-mixture-infinite}
\OT_q^q(\nu_a,\Lambda_a)
\leq
C G,
\qquad
G:=\sum_{k=0}^\infty2^{-kq}S_k.
\end{equation}

We first bound $\E G$. For every $c\in\mathcal P_k$,
\[
\nu_a(c)-\Lambda_a(c)
=
\sum_{i=1}^n
a_i\{\nu_i(c)-\Lambda_i(c)\}
\]
is a sum of independent centered variables. Therefore,
\[
\Var\{\nu_a(c)\}
=
\sum_{i=1}^na_i^2\Var\{\nu_i(c)\}
\leq
\sum_{i=1}^na_i^2\E\{\nu_i(c)^2\}.
\]
Using $\E|X|\leq\sqrt{\E X^2}$ for centered $X$, followed by
Cauchy--Schwarz over the at most $4^k$ cells of $\mathcal P_k$, gives
\begin{align*}
\E S_k
&\leq
\sum_{c\in\mathcal P_k}
\sqrt{\Var\{\nu_a(c)\}}\\
&\leq
2^k
\left(
\sum_{c\in\mathcal P_k}
\Var\{\nu_a(c)\}
\right)^{1/2}\\
&\leq
2^k
\left(
\sum_{i=1}^na_i^2
\E\left[
\sum_{c\in\mathcal P_k}\nu_i(c)^2
\right]
\right)^{1/2}\\
&\leq
2^kM
\left(\sum_{i=1}^na_i^2\right)^{1/2}
=
2^kMN^{-1/2},
\end{align*}
where
\[
\sum_{c\in\mathcal P_k}\nu_i(c)^2
\leq
\left(
\sum_{c\in\mathcal P_k}\nu_i(c)
\right)^2
=
\nu_i(\OmegaL)^2
\leq M^2.
\]
Together with the deterministic bound $S_k\leq2M$, this gives
\[
\E S_k
\leq
M\left(2^kN^{-1/2}\wedge2\right).
\]

Let
$k_N:=\left\lfloor\frac12\log_2N\right\rfloor.
$
For $k\leq k_N$, we have $2^kN^{-1/2}\leq1$, and therefore use \( \E S_k\leq M2^kN^{-1/2}. \) For $k>k_N$, we use the uniform bound \( \E S_k\leq2M. \) Consequently,
\begin{align*}
\E G
&=
\sum_{k=0}^{k_N}2^{-kq}\E S_k
+
\sum_{k>k_N}2^{-kq}\E S_k\\
&\leq
\sum_{k=0}^{k_N}2^{-kq}
\left(2^kMN^{-1/2}\right)
+
\sum_{k>k_N}2^{-kq}(2M)\\
&=
MN^{-1/2}
\sum_{k=0}^{k_N}2^{-k(q-1)}
+
2M\sum_{k>k_N}2^{-kq}.
\end{align*}
Absorbing the fixed factor $2$ into the constant gives
\[
\E G
\leq
CMN^{-1/2}
\sum_{k=0}^{k_N}2^{-k(q-1)}
+
CM\sum_{k>k_N}2^{-kq}.
\]

If $q>1$, the first sum is uniformly bounded; if $q=1$, it contains
$k_N+1\leq C\log(2+N)$ equal terms. Moreover,
\[
\sum_{k>k_N}2^{-kq}
\leq
C2^{-qk_N}
\leq
CN^{-q/2}
\leq
CN^{-1/2}.
\]
Hence, in both cases,
\begin{equation}
\label{eq:weighted-mixture-expectation}
\E G
\leq
CM\frac{\chi_q(N)}{\sqrt N}.
\end{equation}

It remains to control deviations of $G$. Replace $\nu_i$ by another measure
$\nu_i'$ satisfying $\nu_i'(\OmegaL)\leq M$, and let $S_k'$ and $G'$ denote
the resulting analogously defined quantities. At every level,
\begin{align*}
|S_k-S_k'|
&\leq
a_i\sum_{c\in\mathcal P_k}
|\nu_i(c)-\nu_i'(c)|\\
&\leq
a_i\{\nu_i(\OmegaL)+\nu_i'(\OmegaL)\}
\leq
2Ma_i.
\end{align*}
Consequently,
\[
|G-G'|
\leq
2Ma_i\sum_{k=0}^\infty2^{-kq}
\leq
C_qMa_i.
\]
We now apply McDiarmid's inequality explicitly. Replacing $\nu_i$ by
$\nu_i'$ changes $G$ by at most
\[
c_i:=C_qMa_i.
\]
Consequently,
\[
\sum_{i=1}^n c_i^2
=
C_q^2M^2\sum_{i=1}^n a_i^2
=
\frac{C_q^2M^2}{N}.
\]
McDiarmid's inequality therefore gives, for every $t>0$,
\[
\Pp\left(
G-\E G\geq t\mid\mathcal G
\right)
\leq
\exp\left(
-\frac{2Nt^2}{C_q^2M^2}
\right).
\]
Taking
\[
t=
\frac{C_qM}{\sqrt{2}}
\sqrt{\frac{\log(1/\epsilon)}{N}}
\]
shows that, with conditional probability at least $1-\epsilon$,
\[
G
\leq
\E G+
\frac{C_qM}{\sqrt{2}}
\sqrt{\frac{\log(1/\epsilon)}{N}}.
\]
Using the previously established estimate
\[
\E G
\leq
CM\frac{\chi_q(N)}{\sqrt N},
\]
we obtain, with the same conditional probability,
\[
G
\leq
CM\left\{
\frac{\chi_q(N)}{\sqrt N}
+
\sqrt{\frac{\log(1/\epsilon)}{N}}
\right\}.
\]
Finally, since
\[
\OT_q^q(\nu_a,\Lambda_a)\leq CG,
\]
we conclude, after enlarging the constant, that
\[
\OT_q^q(\nu_a,\Lambda_a)
\leq
C\left\{
\frac{\chi_q(N)}{\sqrt N}
+
\sqrt{\frac{\log(1/\epsilon)}{N}}
\right\}
\]
with conditional probability at least $1-\epsilon$.
\end{proof}

The extra logarithm for $q=1$ comes from the multiscale sum: its dyadic terms
form a geometric series for $q>1$, whereas for $q=1$ there are order $\log N$
equal contributions. We do not attempt to remove this factor.

\begin{lemma}[Uniform local measure bound]
\label{lem:uniform-local-epd}
Under the assumptions of \Cref{lem:denominator-concentration} and
\Cref{ass:envelope}\textup{(E1)}, for every fixed $\delta\in(0,1)$ and all
sufficiently large $n$, on the event
$\mathcal E_Z=\mathcal E_Z(h_Z,\delta)$ of
\Cref{lem:denominator-concentration}, with conditional probability at least
$1-\delta$ given the covariates,
\begin{align*}
&\sup_{z\in\calZ}
\OT_q^q\Biggl(
\sum_{i=1}^n a_i(z)\mu_i^w,\\[-0.2em]
&\hspace{7em}\sum_{i=1}^n a_i(z)\Lambda_{Z_i}^w
\Biggr)\\
&\qquad\leq
C\left\{
\frac{\chi_q(nh_Z^d)}{\sqrt{nh_Z^d}}
+
\sqrt{
\frac{\log\!\left(Cn/(\delta h_Z^d)\right)}
{nh_Z^d}
}
\right\}.
\end{align*}
\end{lemma}

\begin{proof}
Let
\[
\mathcal G:=\sigma(Z_1,\ldots,Z_n)
\]
and fix a realization of the covariates belonging to $\mathcal E_Z$.
Conditional on $\mathcal G$, the random measures
$\mu_1^w,\ldots,\mu_n^w$ are independent. Moreover, for every Borel set
$A\subseteq\OmegaL$,
\[
\E\bigl[\mu_i^w(A)\mid\mathcal G\bigr]
=
\E\bigl[\mu_i^w(A)\mid Z_i\bigr]
=
\Lambda_{Z_i}^w(A).
\]
Thus $\Lambda_{Z_i}^w$ is the conditional mean measure of $\mu_i^w$ given
$\mathcal G$. By \Cref{ass:envelope}\textup{(E1)},
\[
\mu_i^w(\OmegaL)\leq M_w
\]
almost surely, and hence
\[
\Lambda_{Z_i}^w(\OmegaL)
=
\E\bigl[\mu_i^w(\OmegaL)\mid\mathcal G\bigr]
\leq M_w.
\]

For each $z\in\calZ$, define
\[
\nu_z:=\sum_{i=1}^n a_i(z)\mu_i^w,
\qquad
\Lambda_z^\circ:=
\sum_{i=1}^n a_i(z)\Lambda_{Z_i}^w.
\]

We first control the variation of these mixtures with respect to $z$.
For $z,z'\in\calZ$, define
\[
\gamma_{z,z'}
:=
\sum_{i=1}^n
\min\{a_i(z),a_i(z')\}\mu_i^w.
\]
Since $\min\{a_i(z),a_i(z')\}\leq a_i(z)$ and
$\min\{a_i(z),a_i(z')\}\leq a_i(z')$ for every $i$, we have
\[
\gamma_{z,z'}\leq\nu_z
\qquad\text{and}\qquad
\gamma_{z,z'}\leq\nu_{z'}.
\]
Hence $\gamma_{z,z'}$ is a common submeasure of $\nu_z$ and $\nu_{z'}$.
The part of the two measures represented by $\gamma_{z,z'}$ can be coupled
through
\[
(\operatorname{id},\operatorname{id})_{\#}\gamma_{z,z'},
\]
which has zero transport cost. It remains only to handle the residual
measures
\[
\nu_z-\gamma_{z,z'}
=
\sum_{i=1}^n
\bigl(a_i(z)-a_i(z')\bigr)_+\mu_i^w
\]
and
\[
\nu_{z'}-\gamma_{z,z'}
=
\sum_{i=1}^n
\bigl(a_i(z')-a_i(z)\bigr)_+\mu_i^w.
\]
 The remaining mass in $\nu_z$ can be sent to the diagonal, and
the remaining mass in $\nu_{z'}$ can be created from the diagonal. Since
\[
d_{\max}:=
\sup_{u\in\OmegaL}d_{\DeltaL,2}(u)<\infty
\]
and $\mu_i^w(\OmegaL)\leq M_w$, this gives an admissible partial transport
whose cost is at most
$d_{\max}^qM_w
\sum_{i=1}^n|a_i(z)-a_i(z')|.
$

The same construction, using
$\Lambda_{Z_i}^w(\OmegaL)\leq M_w$, applies to
$\Lambda_z^\circ$ and $\Lambda_{z'}^\circ$. Therefore,
\[
\OT_q^q(\nu_z,\nu_{z'})
+
\OT_q^q(\Lambda_z^\circ,\Lambda_{z'}^\circ)
\leq
C\sum_{i=1}^n|a_i(z)-a_i(z')|.
\]
On $\mathcal E_Z$, \Cref{lem:local-effective-sample} consequently gives,
whenever $\norm{z-z'}\leq h_Z$,
\begin{equation}
\label{eq:local-mixture-continuity}
\OT_q^q(\nu_z,\nu_{z'})
+
\OT_q^q(\Lambda_z^\circ,\Lambda_{z'}^\circ)
\leq
C\frac{\norm{z-z'}}{h_Z}.
\end{equation}

Set
$\eta:=h_Zn^{-2}
$
and let $\mathcal N_\eta\subseteq\calZ$ be an $\eta$-net. Since $\calZ$ is
a bounded subset of $\mathbb R^d$,
\[
|\mathcal N_\eta|
\leq C\eta^{-d}
=
Cn^{2d}h_Z^{-d}.
\]
For every $z\in\calZ$, choose $z'\in\mathcal N_\eta$ such that
$\norm{z-z'}\leq\eta.
$
Since $\eta\leq h_Z$, \eqref{eq:local-mixture-continuity} applies to the
pair $(z,z')$. Moreover, the triangle inequality for $\OT_q$ gives
\[
\OT_q(\nu_z,\Lambda_z^\circ)
\leq
\OT_q(\nu_z,\nu_{z'})
+
\OT_q(\nu_{z'},\Lambda_{z'}^\circ)
+
\OT_q(\Lambda_{z'}^\circ,\Lambda_z^\circ).
\]
Therefore, using
\[
(x+y+r)^q\leq3^{q-1}(x^q+y^q+r^q)
\]
and \eqref{eq:local-mixture-continuity}, we obtain
\begin{align*}
\OT_q^q(\nu_z,\Lambda_z^\circ)
&\leq
3^{q-1}\Bigl\{
\OT_q^q(\nu_z,\nu_{z'})
+
\OT_q^q(\nu_{z'},\Lambda_{z'}^\circ)
+
\OT_q^q(\Lambda_{z'}^\circ,\Lambda_z^\circ)
\Bigr\}\\
&\leq
C\OT_q^q(\nu_{z'},\Lambda_{z'}^\circ)
+
C\frac{\norm{z-z'}}{h_Z}\\
&\leq
C\OT_q^q(\nu_{z'},\Lambda_{z'}^\circ)
+
C\frac{\eta}{h_Z}.
\end{align*}

Since $\eta/h_Z=n^{-2}$,
\begin{equation}
\label{eq:local-mixture-net-extension}
\sup_{z\in\calZ}
\OT_q^q(\nu_z,\Lambda_z^\circ)
\leq
C\max_{z'\in\mathcal N_\eta}
\OT_q^q(\nu_{z'},\Lambda_{z'}^\circ)
+
Cn^{-2}.
\end{equation}

It remains to control the finitely many net points. On $\mathcal E_Z$,
\Cref{lem:local-effective-sample} gives
\[
N_{\mathrm{eff}}\{a(t)\}
=
\left(\sum_{i=1}^n a_i(t)^2\right)^{-1}
\geq cnh_Z^d
\qquad
\text{for every }t\in\mathcal N_\eta.
\]
At a fixed $t\in\mathcal N_\eta$, apply
\Cref{lem:weighted-epd} conditionally on $\mathcal G$ with
\[
\nu_i=\mu_i^w,
\qquad
\Lambda_i=\Lambda_{Z_i}^w,
\qquad
\epsilon=\frac{\delta}{|\mathcal N_\eta|}.
\]
It follows that, except on a conditional event of probability at most
$\delta/|\mathcal N_\eta|$,
\[
\OT_q^q(\nu_t,\Lambda_t^\circ)
\leq
C\left\{
\frac{
\chi_q\!\left(N_{\mathrm{eff}}\{a(t)\}\right)
}{
\sqrt{N_{\mathrm{eff}}\{a(t)\}}
}
+
\sqrt{
\frac{
\log\!\left(|\mathcal N_\eta|/\delta\right)
}{
N_{\mathrm{eff}}\{a(t)\}
}
}
\right\}.
\]
A union bound makes this conclusion simultaneous over all
$t\in\mathcal N_\eta$ with conditional probability at least $1-\delta$.

For all sufficiently large $nh_Z^d$, the function
\[
N\longmapsto\frac{\chi_q(N)}{\sqrt N}
\]
is decreasing. Using
$N_{\mathrm{eff}}\{a(t)\}\geq cnh_Z^d$ and changing the constant therefore
gives
\begin{equation}
\label{eq:local-mixture-net-bound}
\max_{z'\in\mathcal N_\eta}
\OT_q^q(\nu_{z'},\Lambda_{z'}^\circ)
\leq
C\left\{
\frac{\chi_q(nh_Z^d)}{\sqrt{nh_Z^d}}
+
\sqrt{
\frac{\log\!\left(|\mathcal N_\eta|/\delta\right)}
{nh_Z^d}
}
\right\}
\end{equation}
with conditional probability at least $1-\delta$.

Finally, the bound on the cardinality of the net implies
\[
\begin{aligned}
\log\!\left(\frac{|\mathcal N_\eta|}{\delta}\right)
&\leq
\log(C/\delta)+2d\log n+d\log(1/h_Z)\\
&\leq
C\log\!\left(\frac{Cn}{\delta h_Z^d}\right),
\end{aligned}
\]
where the last inequality uses that $d$ is fixed and $h_Z\leq1$ for all
sufficiently large $n$. Consequently,
\[
\max_{z'\in\mathcal N_\eta}
\OT_q^q(\nu_{z'},\Lambda_{z'}^\circ)
\leq
C\left\{
\frac{\chi_q(nh_Z^d)}{\sqrt{nh_Z^d}}
+
\sqrt{
\frac{\log\!\left(Cn/(\delta h_Z^d)\right)}
{nh_Z^d}
}
\right\}.
\]
For all sufficiently large $n$, the extension error $n^{-2}$ in
\eqref{eq:local-mixture-net-extension} is bounded by
\[
\frac{\chi_q(nh_Z^d)}{\sqrt{nh_Z^d}}
+
\sqrt{
\frac{\log\!\left(Cn/(\delta h_Z^d)\right)}
{nh_Z^d}
}.
\]
Combining the preceding bound with
\eqref{eq:local-mixture-net-extension}, and enlarging the constant $C$,
proves the result.
\end{proof}

\begin{proof}[Proof of \Cref{thm:pot-rate}]
Work on the event on which the covariate denominator is positive uniformly in $z$; \Cref{lem:denominator-concentration} will give this event the required probability below.  With $a_i(z)$ as above, set
\[
\overline\Lambda_{z,n}^w:=\sum_{i=1}^n a_i(z)\mu_i^w,
\qquad
\overline\Lambda_z^w:=\sum_{i=1}^n a_i(z)\Lambda_{Z_i}^w.
\]
The double-kernel estimator, interpreted as a measure, is
\[
\widehat\Lambda_{z,n}^{w,h_U}=T_{U,h_U}\overline\Lambda_{z,n}^w,
\]
with $T_{U,h}$ as in \Cref{lem:smoothing-pot}.  To separate the three sources of error, insert the two intermediate measures
$\overline\Lambda_{z,n}^w$ and $\overline\Lambda_z^w$.  The triangle inequality for $\OT_q$ gives
\begin{align*}
\OT_q(\widehat\Lambda_{z,n}^{w,h_U},\Lambda_z^w)
&\le
\OT_q(T_{U,h_U}\overline\Lambda_{z,n}^w,
       \overline\Lambda_{z,n}^w)\\
&\quad+
\OT_q(\overline\Lambda_{z,n}^w,\overline\Lambda_z^w)
+\OT_q(\overline\Lambda_z^w,\Lambda_z^w).
\end{align*}
For partial transport, this triangle inequality follows by gluing two admissible transports along their common off-diagonal marginal; diagonal mass remains unconstrained and can be retained in the glued plan.  Raising the preceding display to the $q$th power and using
$(x+y+r)^q\le3^{q-1}(x^q+y^q+r^q)$ yields
\[
\OT_q^q
\left(
\widehat\Lambda_{z,n}^{w,h_U},
\Lambda_z^w
\right)
\le
C_q(A_n(z)+B_n(z)+C_n(z)),
\]
where
\[
A_n(z):=\OT_q^q\left(T_{U,h_U}\overline\Lambda_{z,n}^w,\overline\Lambda_{z,n}^w\right),
\quad
B_n(z):=\OT_q^q\left(\overline\Lambda_{z,n}^w,\overline\Lambda_z^w\right),
\quad
C_n(z):=\OT_q^q\left(\overline\Lambda_z^w,\Lambda_z^w\right).
\]

We first control the covariate-bias term $C_n(z)$.  The weights are nonnegative and sum to one.  In addition, kernel support implies
\[
a_i(z)>0\quad\Longrightarrow\quad\norm{Z_i-z}\le h_Z.
\]
We recall explicitly the mixture-convexity calculation.  For each $i$, choose an admissible transport $\pi_i$ from
$\Lambda_{Z_i}^w$ to $\Lambda_z^w$ whose cost is within $\varepsilon$ of the infimum.  Then
\[
\pi:=\sum_{i=1}^na_i(z)\pi_i
\]
is admissible.  Indeed, for every Borel set $A\subseteq\OmegaLo$, its two constrained marginals are
\begin{align*}
\pi(A\times\OmegaL)
&=\sum_i a_i(z)\Lambda_{Z_i}^w(A)
=\overline\Lambda_z^w(A),\\
\pi(\OmegaL\times A)
&=\sum_i a_i(z)\Lambda_z^w(A)
=\Lambda_z^w(A).
\end{align*}
Its cost is the corresponding weighted sum of the costs of the $\pi_i$.  Letting $\varepsilon\downarrow0$ gives
\[
C_n(z)
\le
\sum_{i=1}^n a_i(z)\OT_q^q(\Lambda_{Z_i}^w,\Lambda_z^w)
\le
\omega_{\OT,q}(h_Z),
\]
uniformly in $z$.  This term is deterministic once the covariates and the weights are fixed.

We next control the persistence-smoothing term $A_n(z)$.  The mixture
$\overline\Lambda_{z,n}^w$ has mass at most $M_w$, because every
$\mu_i^w$ does and the weights sum to one.  \Cref{lem:smoothing-pot} therefore gives
\[
A_n(z)
\le
C h_U^q
+
C\sum_{i=1}^n a_i(z)\int_{\OmegaL}|r_{U,h_U}(v)-1|d_{\DeltaL,2}(v)^q\dd\mu_i^w(v).
\]
Define
\[
B_i:=\int_{\OmegaL}|r_{U,h_U}(v)-1|d_{\DeltaL,2}(v)^q\dd\mu_i^w(v).
\]
The proof of \Cref{lem:smoothing-pot} showed that
$r_{U,h_U}$ is uniformly bounded.  Also,
$d_{\DeltaL,2}(v)\le d_{\max}$ on the compact window and
$\mu_i^w(\OmegaL)\le M_w$.  Hence $B_i\le C M_wd_{\max}^q$ almost surely.  By the defining property of the conditional mean measure,
\begin{align*}
\E[B_i\mid Z_i=z']
&=\int_{\OmegaL}|r_{U,h_U}(v)-1|
d_{\DeltaL,2}(v)^q\dd\Lambda_{z'}^w(v)\\
&\le\mathfrak b_{U,q}(h_U)
\qquad\text{for every }z'\in\calZ.
\end{align*}
Apply \Cref{lem:scalar-local-average} with
$b_*(h_U)=\mathfrak b_{U,q}(h_U)$ and failure probability $\delta/3$.  It follows that
\[
\sup_{z\in\calZ}A_n(z)
\le
C(h_U^q+\mathfrak b_{U,q}(h_U)+r_{Z,n}(h_Z,\delta))
\]
with probability at least $1-\delta/3$, after changing the constant to absorb the factor $3$ inside the logarithm.

It remains to control $B_n(z)$, the fluctuation of the local empirical measure about its conditional mean.  Apply \Cref{lem:denominator-concentration} with failure probability $\delta/3$ to obtain $\mathcal E_Z$.  Conditional on the covariates and on this event, \Cref{lem:uniform-local-epd}, again with failure probability $\delta/3$, gives
\[
\sup_{z\in\calZ}B_n(z)
\le
C\left(\frac{\chi_q(nh_Z^d)}{\sqrt{nh_Z^d}}+r_{Z,n}(h_Z,\delta)\right)
\]
after adjusting constants inside logarithms.  The smoothing event fails with probability at most $\delta/3$, the denominator event with probability at most $\delta/3$, and, conditional on the covariates, the empirical-measure event with probability at most $\delta/3$.  The tower property and a union bound therefore show that all three estimates hold simultaneously with unconditional probability at least $1-\delta$.  Taking the supremum in the three-term decomposition and substituting the three bounds proves \eqref{eq:pot-rate}.

We finish by checking every term needed for consistency.  By assumption,
\[
h_U^q\to0,
\qquad
\mathfrak b_{U,q}(h_U)\to0,
\qquad
\omega_{\OT,q}(h_Z)\to0.
\]
Let $N_Z:=nh_Z^d$. The condition $N_Z/\log n\to\infty$, together with
\eqref{eq:covariate-log-consequence}, implies
\[
r_{Z,n}(h_Z,\delta)
=O\left(\sqrt{\frac{\log n}{N_Z}}+\frac{\log n}{N_Z}\right)
\longrightarrow0
\]
for fixed $\delta$.  If $q>1$, then
$\chi_q(N_Z)N_Z^{-1/2}=N_Z^{-1/2}\to0$.  If $q=1$, the same conclusion follows from the elementary limit
$\log(2+x)/\sqrt{x}\to0$ as $x\to\infty$.  Thus the right-hand side of \eqref{eq:pot-rate} converges to zero.  Given any probability tolerance, apply the finite-sample result with that value of $\delta$; this proves convergence in probability of the supremum of $\OT_q^q$, and taking the $q$th root proves the stated $\OT_q$ consistency.
\end{proof}

\begin{proof}[Proof of \Cref{cor:pot-holder-rate}]
The H\"older assumption gives the explicit covariate-bias bound
\[
\omega_{\OT,q}(h_Z)
=\sup_{\norm{z-z'}\le h_Z}
\OT_q^q(\Lambda_z^w,\Lambda_{z'}^w)
\le L_{\OT}h_Z^\beta.
\]
The boundary assumption gives
$\mathfrak b_{U,q}(h_U)\le Ch_U^q$.  Fix
$\delta\in(0,1)$. \Cref{eq:covariate-log-consequence} implies
\[
r_{Z,n}(h_Z,\delta)
=O\left(
\sqrt{\frac{\log n}{nh_Z^d}}
+\frac{\log n}{nh_Z^d}
\right).
\]
Since the corollary assumes $q>1$, we have $\chi_q\equiv1$, and
\[
\frac1{\sqrt{nh_Z^d}}
\le
\sqrt{\frac{\log n}{nh_Z^d}}
\]
for all sufficiently large $n$.  Substitution in \Cref{thm:pot-rate} therefore yields
\[
\sup_{z\in\calZ}
\OT_q^q(\widehat\Lambda_{z,n}^{w,h_U},\Lambda_z^w)
=O_P\left(
h_U^q+h_Z^\beta
+\sqrt{\frac{\log n}{nh_Z^d}}
+\frac{\log n}{nh_Z^d}
\right),
\]
To verify the bandwidth balance, put
\[
t_n:=\frac{\log n}{n},
\qquad
h_Z\asymp t_n^{1/(2\beta+d)}.
\]
Then
\[
h_Z^\beta\asymp t_n^{\beta/(2\beta+d)}.
\]
For the square-root stochastic term,
\begin{align*}
\sqrt{\frac{\log n}{nh_Z^d}}
&=\sqrt{\frac{t_n}{h_Z^d}}\\
&\asymp
\left(t_nt_n^{-d/(2\beta+d)}\right)^{1/2}\\
&=t_n^{\beta/(2\beta+d)}.
\end{align*}
The linear stochastic term is smaller:
\[
\frac{\log n}{nh_Z^d}
=\frac{t_n}{h_Z^d}
\asymp t_n^{2\beta/(2\beta+d)}
=o(t_n^{\beta/(2\beta+d)}).
\]
Finally, any $h_U$ satisfying
$h_U^q\lesssim t_n^{\beta/(2\beta+d)}$ makes both persistence-smoothing contributions no larger than the balanced covariate rate.  These bandwidths satisfy
$nh_Z^d/\log n\asymp t_n^{-2\beta/(2\beta+d)}\to\infty$, as required by the theorem.  Therefore,
\[
\sup_{z\in\calZ}
\OT_q^q(\widehat\Lambda_{z,n}^{w,h_U},\Lambda_z^w)
=O_P(t_n^{\beta/(2\beta+d)}).
\]
Because $x\mapsto x^{1/q}$ is increasing on $[0,\infty)$, taking $q$th roots gives
\[
\sup_{z\in\calZ}
\OT_q(\widehat\Lambda_{z,n}^{w,h_U},\Lambda_z^w)
=O_P(t_n^{\beta/(q(2\beta+d))}),
\]
which is the final assertion.
\end{proof}

\subsection[Proof of the normalized-density theorem]
{Proof of \Cref{thm:normalized-rate}}
\label{app:normalized}

\begin{proof}[Proof of \Cref{thm:normalized-rate}]
Write
\[
\widehat{\widetilde N}_n(z,u)
:=
\frac1n\sum_{i=1}^n
\kappa_{Z,h_Z}(z,Z_i)\widetilde Y_{i,h_U}(u),
\qquad
\widetilde N_h:=\E\widehat{\widetilde N}_n,
\qquad
\widetilde\lambda_h:=\frac{\widetilde N_h}{D_h}.
\]
The normalized response has unit total mass. Consequently,
Cauchy--Schwarz with respect to \(\widetilde\mu_i^w\) gives
\[
\widetilde Y_{i,h_U}(u)^2
\leq
\int_{\OmegaL}\kappa_{U,h_U}(u,v)^2
\dd\widetilde\mu_i^w(v).
\]
Taking conditional expectations and using the density of the normalized
conditional mean measure and \eqref{eq:kernel-norm-bounds}, we obtain
\begin{equation}\label{eq:normalized-response-second-moment}
\sup_{z,u}
\E[\widetilde Y_{i,h_U}(u)^2\mid Z_i=z]
\leq
\norm{\widetilde\lambda_w}_\infty
\sup_u\int_{\OmegaL}\kappa_{U,h_U}(u,v)^2\dd v
\leq Ch_U^{-2}.
\end{equation}
This is the only response-specific check not already contained in the proof
of \Cref{thm:supnorm-rate}. Indeed,
\Cref{lem:double-kernel-bias}, applied to
\(\nu=\widetilde\mu^w\), gives
\begin{equation}\label{eq:normalized-deterministic-bias}
\norm{\widetilde\lambda_h-\widetilde\lambda_w}_\infty
\leq\omega_{\widetilde\lambda_w}(h_Z+h_U),
\qquad
\inf_zD_h(z)\geq c_Z,
\qquad
\norm{\widetilde\lambda_h}_\infty
\leq\norm{\widetilde\lambda_w}_\infty.
\end{equation}
The mass-envelope condition in \Cref{lem:numerator-concentration} holds
with \(M_n=1\), and its second-moment condition is
\eqref{eq:normalized-response-second-moment}. Applying that lemma and
\Cref{lem:denominator-concentration}, each with failure probability
\(\delta/2\), yields an event \(\mathcal E\) of probability at least
\(1-\delta\) on which
\[
\norm{\widehat{\widetilde N}_n-\widetilde N_h}_\infty
\leq Cr_n(h_Z,h_U,\delta),
\qquad
\norm{\widehat D_n-D_h}_\infty
\leq Cr_{Z,n}(h_Z,\delta).
\]
As in the proof of \Cref{thm:supnorm-rate}, \(h_U\leq1\) eventually implies
\(r_{Z,n}\leq Cr_n\), and \Cref{ass:kernels} implies
\(r_{Z,n}\to0\). Hence, for all sufficiently large \(n\),
\(\inf_z\widehat D_n(z)\geq c_Z/2\) on \(\mathcal E\). The exact ratio
identity
\[
\widehat{\widetilde\lambda}_{w,n}-\widetilde\lambda_h
=
\frac{
(\widehat{\widetilde N}_n-\widetilde N_h)
-\widetilde\lambda_h(\widehat D_n-D_h)
}{
\widehat D_n
}
\]
and \eqref{eq:normalized-deterministic-bias} therefore give
\[
\norm{
\widehat{\widetilde\lambda}_{w,n}
-\widetilde\lambda_w
}_\infty
\leq
C\left\{
\omega_{\widetilde\lambda_w}(h_Z+h_U)
+r_n(h_Z,h_U,\delta)
\right\}
\]
on \(\mathcal E\), proving the finite-sample assertion. Uniform continuity
and \Cref{ass:kernels} make the two terms on the right tend to zero; applying
the bound with an arbitrary fixed failure probability proves convergence in
probability.
\end{proof}

\subsection{Sufficient conditions and examples for truncation}
\label{app:truncation-details}

The two truncation requirements concern different issues. First, the
deterministic levels $M_n$ must be large enough that the probability of
truncating any of the $n$ observed diagrams vanishes, but small enough that
the variable-envelope term in the numerator concentration bound remains
negligible. A convenient sufficient condition follows from a moment bound.
If
\[
\E\left[M(D)^r\right]<\infty
\]
for some $r>0$ and
\[
n^{1/r}
=
o\left(
\frac{nh_Z^dh_U^2}{\log n}
\right),
\]
choose a nondecreasing sequence $\ell_n\to\infty$ satisfying
\[
\ell_n
=
o\left(
\frac{nh_Z^dh_U^2}
{n^{1/r}\log n}
\right)
\]
and set $M_n=n^{1/r}\ell_n$. Then Markov's inequality gives
\[
n\Pp\{M(D)>M_n\}
\leq
\frac{\E[M(D)^r]}{\ell_n^r}
\longrightarrow0,
\]
while
\[
\frac{M_n\log n}{nh_Z^dh_U^2}
\longrightarrow0.
\]
For exponentially light tails, one may instead take $M_n$ of logarithmic
order; for example, if $\E\exp\{cM(D)\}<\infty$, then $M_n=C\log n$ is
admissible for sufficiently large $C$ whenever
$nh_Z^dh_U^2/(\log n)^2\to\infty$.

The spatial truncation error is controlled separately by the conditional
mass-weighted mean measure
\[
\Gamma_z^w(A)
:=
\E\left[
M(D)\mu_D^w(A)\mid Z=z
\right].
\]
Suppose that $\Gamma_z^w$ has a density $\gamma_w(z,u)$ satisfying
\[
\sup_{z\in\calZ,u\in\OmegaL}\gamma_w(z,u)<\infty.
\]
This condition also verifies the second-moment assumption in
\Cref{prop:truncation}. Indeed, Cauchy--Schwarz with respect to $\mu_D^w$
and the definition of $\Gamma_z^w$ give
\[
\E\left[(Y^w_{h_U}(u))^2\mid Z=z\right]
\leq
\int_{\OmegaL}\kappa_{U,h_U}(u,v)^2\gamma_w(z,v)\dd v
\leq Ch_U^{-2}.
\]
Then, for every $t>0$,
\[
\E\left[
\mu_D^w(A)\ind{M(D)>t}
\mid Z=z
\right]
\leq
\frac{1}{t}\Gamma_z^w(A).
\]
Therefore, the conditional tail mean measure has a density
$\lambda^{\mathrm{tail}}_{w,t}$ satisfying
\[
0
\leq
\lambda^{\mathrm{tail}}_{w,t}(z,u)
\leq
\frac{\gamma_w(z,u)}{t},
\]
and therefore
\[
\left\|
\lambda_{w,n}^{\mathrm{tr}}-\lambda_w
\right\|_\infty
\leq
\frac{\|\gamma_w\|_\infty}{M_n}
\longrightarrow0.
\]
This condition is not specific to Poisson processes and does not require the
diagram to contain finitely many points. For a finite diagram with cardinality
$N$, it holds, for example, if the conditional one-point marginals given
$(Z,N)$ are uniformly bounded and
$\sup_z\E[N^2\mid Z=z]<\infty$. It is also preserved under conditionally
independent finite superpositions. Thus one may superpose a Poisson noise
component having infinitely many points accumulating near the diagonal with a
finite repulsive component, provided the total weighted mass is finite and the
corresponding component mass-weighted mean densities are uniformly bounded.
In the Poisson case, the latter property follows directly from the Mecke
formula when the weighted intensity is bounded, as detailed in
\Cref{prop:poisson-truncation}. In particular, the Poisson designs of
\Cref{sec:sim-forward} satisfy this condition.

\subsection[Proof of the truncation proposition]{Proof of \Cref{prop:truncation}}\label{app:truncation}

\begin{proof}[Proof of \Cref{prop:truncation}]
Define
\[
\mu_{i,n}^{w,\mathrm{tr}}
:=
\mu_i^w\ind{M(D_i)\leq M_n},
\qquad
\mathcal A_n:=\bigcap_{i=1}^n\{M(D_i)\leq M_n\}.
\]
Let $\widehat\lambda_{w,n}^{\mathrm{tr}}$ be the estimator obtained from
$\mu_{i,n}^{w,\mathrm{tr}}$ in place of $\mu_i^w$.
A union bound gives
\(\Pp(\mathcal A_n^c)\leq n\Pp\{M(D)>M_n\}\to0\), and on
\(\mathcal A_n\) the original and truncated estimators coincide.

Let \(\widehat N_n^{\mathrm{tr}}\) be the numerator formed from the truncated
measures, let \(N_{h,n}^{\mathrm{tr}}:=\E\widehat
N_n^{\mathrm{tr}}\), and set
\[
\lambda_{h,n}^{\mathrm{tr}}
:=
\frac{N_{h,n}^{\mathrm{tr}}}{D_h},
\qquad
b_n:=
\norm{\lambda_{w,n}^{\mathrm{tr}}-\lambda_w}_\infty.
\]
The local-average calculation in the proof of
\Cref{lem:double-kernel-bias} gives
\[
\norm{\lambda_{h,n}^{\mathrm{tr}}
-\lambda_{w,n}^{\mathrm{tr}}}_\infty
\leq
\omega_{\lambda_{w,n}^{\mathrm{tr}}}(h_Z+h_U).
\]
Adding and subtracting \(\lambda_w\) at the two arguments in this modulus
shows that
\[
\omega_{\lambda_{w,n}^{\mathrm{tr}}}(r)
\leq\omega_{\lambda_w}(r)+2b_n.
\]
The same local-average representation yields
\[
\norm{\lambda_{h,n}^{\mathrm{tr}}}_\infty
\leq\norm{\lambda_{w,n}^{\mathrm{tr}}}_\infty
\leq\norm{\lambda_w}_\infty+b_n,
\]
which is uniformly bounded for all sufficiently large \(n\).

The truncated response satisfies
\[
Y_{i,h_U,n}^{\mathrm{tr}}(u)
=
\ind{M(D_i)\leq M_n}Y^w_{i,h_U}(u),
\]
so its conditional second moment is at most \(C_Yh_U^{-2}\), while
\(\mu_{i,n}^{w,\mathrm{tr}}(\OmegaL)\leq M_n\). Since
\(h_Z^dh_U^2\leq1\) eventually, the assumed variable-envelope condition
implies \(M_n\log n/n\to0\); together with \(M_n\uparrow\infty\), this gives
\(1\leq M_n\leq n\) eventually. We may therefore apply
\Cref{lem:numerator-concentration}. With failure probability \(\delta/2\),
\begin{equation}\label{eq:truncated-uniform-numerator}
\norm{\widehat N_n^{\mathrm{tr}}-N_{h,n}^{\mathrm{tr}}}_\infty
\leq
C\left\{
\sqrt{\frac{L_{n,\delta}}{nh_Z^dh_U^2}}
+\frac{M_nL_{n,\delta}}{nh_Z^dh_U^2}
\right\},
\qquad
L_{n,\delta}:=\log\frac{Cn}{\delta h_Z^dh_U^2}.
\end{equation}
With another failure probability \(\delta/2\),
\Cref{lem:denominator-concentration} gives
\[
\norm{\widehat D_n-D_h}_\infty
\leq Cr_{Z,n}(h_Z,\delta),
\qquad
\inf_z\widehat D_n(z)\geq\frac{c_Z}{2}
\]
for all sufficiently large \(n\). On the intersection of these two events,
the exact ratio identity used in the proof of
\Cref{thm:supnorm-rate} gives
\[
\norm{\widehat\lambda_{w,n}^{\mathrm{tr}}
-\lambda_{h,n}^{\mathrm{tr}}}_\infty
\leq
\frac{2}{c_Z}\left\{
\norm{\widehat N_n^{\mathrm{tr}}-N_{h,n}^{\mathrm{tr}}}_\infty
+\norm{\lambda_{h,n}^{\mathrm{tr}}}_\infty
\norm{\widehat D_n-D_h}_\infty
\right\}.
\]
Because \(M_n\geq1\) eventually and \(h_U\leq1\),
the denominator rate is absorbed by the two terms in
\eqref{eq:truncated-uniform-numerator}. Combining the stochastic ratio
bound with the two deterministic bias bounds yields
\begin{equation}\label{eq:truncation-proof-bound}
\norm{\widehat\lambda_{w,n}^{\mathrm{tr}}-\lambda_w}_\infty
\leq C\left\{
\omega_{\lambda_w}(h_Z+h_U)+b_n
+\sqrt{\frac{L_{n,\delta}}{nh_Z^dh_U^2}}
+\frac{M_nL_{n,\delta}}{nh_Z^dh_U^2}
\right\}.
\end{equation}
This event has probability at least \(1-\delta\). On \(\mathcal A_n\) the
same bound holds for the original estimator; the total failure probability
is at most \(\delta+n\Pp\{M(D)>M_n\}\). Finally,
\(L_{n,\delta}=O(\log n)\) by
\eqref{eq:bandwidth-log-consequence}, and every term on the right-hand side
of \eqref{eq:truncation-proof-bound} tends to zero by the assumptions.
The claimed convergence in probability follows.
\end{proof}

\subsection[Proof of the conditionally Poisson proposition]
{Proof of \Cref{prop:poisson-truncation}}
\label{app:poisson-truncation}

\begin{proof}[Proof of \Cref{prop:poisson-truncation}]
Write $W:=\norm{w}_\infty$ and
$\Lambda_*:=\norm{\lambda_w}_\infty$.
For every fixed $t>0$, the
Poisson exponential formula and
$(e^{tr}-1)/r\leq te^{tW}$ for $0<r\leq W$ give, uniformly in $z$,
\[
\E[e^{tM(D)}\mid Z=z]
=\exp\!\left\{\int_{\OmegaLo}
\frac{e^{tw(u)}-1}{w(u)}\lambda_w(z,u)\dd u\right\}
\leq
\exp\!\left\{te^{tW}\Lambda_*\operatorname{Vol}(\OmegaL)\right\}
=:C_t.
\]
For a sigma-finite intensity, the identity follows first on an increasing
finite-intensity exhaustion of $\OmegaLo$ and then by monotone convergence.
Consequently,
\[
\sup_z\Pp\{M(D)>T\mid Z=z\}\leq C_te^{-tT}.
\]
The Mecke equation \citep[Theorem~4.1]{last2017lectures} gives both the tail
bias density
\[
\lambda_{w,T}^{\mathrm{tail}}(z,u)
=\lambda_w(z,u)\Pp\{M(D)+w(u)>T\mid Z=z\},
\qquad
\norm{\lambda_{w,T}^{\mathrm{tail}}}_\infty
\leq \Lambda_*C_te^{tW}e^{-tT},
\]
and, for $Y_{h_U}^w(u):=\int\kappa_{U,h_U}(u,v)w(v)\dd D(v)$,
\begin{align*}
\E[(Y_{h_U}^w(u))^2\mid Z=z]
&=\left(\int\kappa_{U,h_U}(u,v)\lambda_w(z,v)\dd v\right)^2\\
&\quad+\int\kappa_{U,h_U}(u,v)^2w(v)\lambda_w(z,v)\dd v
\leq \Lambda_*^2+CW\Lambda_*h_U^{-2}
\leq Ch_U^{-2}.
\end{align*}
Thus $M_n=A\log n$ with $tA>2$ verifies the probability and tail-bias
requirements of \Cref{prop:truncation}; its variable-envelope requirement is
satisfied whenever $nh_Z^dh_U^2/(\log n)^2\to\infty$.
The conclusion follows from \Cref{prop:truncation}.
\end{proof}

\subsection[Proof of the trimming proposition]{Proof of \Cref{prop:trimmed}}\label{app:trimmed}

\begin{proof}[Proof of \Cref{prop:trimmed}]
Fix \(\varepsilon>0\) and take
\(A_\varepsilon,m_{\varepsilon,n},\delta_{\varepsilon,n}\) from the
local-mass condition. In the proofs of
\Cref{lem:denominator-concentration,lem:numerator-concentration}, the design
lower bound is used only to turn a population denominator bound into an
empirical one; the concentration calculations themselves require only
\(p_Z\leq C_Z\). Hence there is an event \(\mathcal E_n\), with
\(\Pp(\mathcal E_n)\geq1-2\delta_{\varepsilon,n}\), on which
\begin{align}
\norm{\widehat D_n-D_h}_\infty
&\leq Cr_{Z,n}(h_Z,\delta_{\varepsilon,n})=:S_{D,n},
\label{eq:trimmed-den-conc}\\
\norm{\widehat N_n-N_h}_\infty
&\leq Cr_n(h_Z,h_U,\delta_{\varepsilon,n})=:S_{N,n}.
\label{eq:trimmed-num-conc}
\end{align}
Since \(h_U\leq1\) eventually, \(S_{D,n}\leq CS_{N,n}\), and the
local-mass assumption implies
\[
\frac{S_{D,n}}{m_{\varepsilon,n}}\to0,
\qquad
\frac{S_{N,n}}{m_{\varepsilon,n}}\to0.
\]

For \(z\in A_\varepsilon\), define
\(\lambda_h(z,u):=N_h(z,u)/D_h(z)\). The local-average calculation in the
proof of \Cref{lem:double-kernel-bias} uses no global design lower bound once
\(D_h(z)>0\), and therefore gives
\begin{equation}\label{eq:trimmed-local-bias}
\sup_{z\in A_\varepsilon,u\in\OmegaL}
|\lambda_h(z,u)-\lambda_w(z,u)|
\leq b_n,
\qquad
b_n:=\omega_{\lambda_w}(h_Z+h_U)\to0,
\end{equation}
together with \(0\leq\lambda_h\leq\|\lambda_w\|_\infty\).

On \(\mathcal E_n\), for all sufficiently large \(n\) and every
\(z\in A_\varepsilon\),
\[
\widehat D_n(z)
\geq D_h(z)-S_{D,n}
\geq\frac12m_{\varepsilon,n}
\geq\rho_n.
\]
Thus the denominator floor is inactive on \(A_\varepsilon\). Using
\(N_h=\lambda_hD_h\), the same exact ratio identity as in
\Cref{thm:supnorm-rate} yields
\begin{align}
\sup_{z\in A_\varepsilon,u\in\OmegaL}
|\widehat\lambda^{\rho_n}_{w,n}(z,u)-\lambda_h(z,u)|
&\leq
\frac{2}{m_{\varepsilon,n}}
\left\{S_{N,n}
+\|\lambda_w\|_\infty S_{D,n}\right\}\nonumber\\
&\leq
C\frac{r_n(h_Z,h_U,\delta_{\varepsilon,n})}
{m_{\varepsilon,n}}.
\label{eq:trimmed-ratio}
\end{align}
Together with \eqref{eq:trimmed-local-bias}, this bounds the un-clipped
error on \(A_\varepsilon\) by
\[
\zeta_{\varepsilon,n}
:=
C\frac{r_n(h_Z,h_U,\delta_{\varepsilon,n})}
{m_{\varepsilon,n}}+b_n
\longrightarrow0.
\]

The clipping map \(x\mapsto x\wedge B\) is one-Lipschitz on
\([0,\infty)\), and \(0\leq\lambda_w<B\). Clipping therefore does not
increase the error on \(A_\varepsilon\); on its complement, both the clipped
estimator and the target belong to \([0,B]\). Consequently, on
\(\mathcal E_n\),
\begin{equation}\label{eq:trimmed-integrated-bound}
\int_\calZ\sup_u
|\widehat\lambda^{\rho_n,B}_{w,n}(z,u)-\lambda_w(z,u)|
\dd P_Z(z)
\leq
\zeta_{\varepsilon,n}+B\varepsilon.
\end{equation}
Given \(\eta>0\), first choose
\(\varepsilon<\eta/(2B)\), and then \(n\) so large that
\(\zeta_{\varepsilon,n}<\eta/2\). The probability that the integrated loss
exceeds \(\eta\) is then at most
\(2\delta_{\varepsilon,n}\to0\), which proves the claim.
\end{proof}

The local-mass condition holds, for example, when $p_Z$ vanishes only on a
$P_Z$-null boundary region at a polynomial rate, with $A_\varepsilon$ a
slightly shrunken version of $\calZ$. Equation
\eqref{eq:trimmed-integrated-bound} also shows that a quantitative choice of
$\varepsilon=\varepsilon_n$ yields a convergence rate.

\clearpage
\section{Bandwidth selection by leave-one-out risk estimation}\label{sec:cv}

The bandwidth schedules of \Cref{cor:holder-rate} determine the rate
exponents but leave their multiplicative constants unspecified, and the
smoothness indices are rarely known in practice.
We discuss now a practical approach to select $(h_Z,h_U)$ via leave-one-out cross-validation.

For candidate bandwidths $(h_Z,h_U)$ and each $i$, let
$\widehat\lambda^{(-i)}$ be the estimator obtained by applying
\eqref{eq:nw-estimator} to the $n-1$ observations with index different from
$i$. Define its leave-one-out risk by
\[
R_{n-1}(h_Z,h_U)
:=
\E\int_{\OmegaL}
\left(
\widehat\lambda^{(-1)}(Z_1,u)-\lambda_w(Z_1,u)
\right)^2
\dd u.
\]
Equivalently, this is the integrated squared error of an estimator trained
on $n-1$ observations, averaged over both the training sample and an
independent covariate drawn from $P_Z$.
Observe that the results above assess the estimator under supremum-norm or partial optimal transport losses. The leave-one-out criterion below instead targets a different quantity,
namely the $P_Z$-averaged integrated squared risk.
The two losses are related only in one direction. For every realization of the estimator,
\[
    \int_{\mathcal Z}\int_{\OmegaL} \left( \widehat\lambda(z,u)-\lambda_w(z,u) \right)^2 \dd u\,\dd P_Z(z) \leq \operatorname{Leb}(\OmegaL) \left\|\widehat\lambda-\lambda_w\right\|_\infty^2.
\]
Thus the uniform theory implies control of the integrated squared error,
but bandwidths minimizing the latter need not minimize the worst-case
error.

We use integrated squared error for bandwidth selection because it admits an exact leave-one-out identity. After expanding the square, the term $\int_{\OmegaL}\lambda_w(Z,u)^2\dd u$ is independent of the bandwidths, while the unknown cross term can be estimated without bias from a held-out weighted diagram through the defining conditional mean-measure identity. Unfortunately, no analogous unbiased decomposition is available for the supremum-norm loss.

Expanding the square gives
\begin{align}
R_{n-1}(h_Z,h_U)
&=
\E\int_{\OmegaL}
\widehat\lambda^{(-1)}(Z_1,u)^2
\dd u
-2\E\int_{\OmegaL}
\widehat\lambda^{(-1)}(Z_1,u)\lambda_w(Z_1,u)
\dd u
+
C_0,
\label{eq:loo-risk-expansion}
\end{align}
where $C_0$ does not depend on the bandwidths.

The cross term in \eqref{eq:loo-risk-expansion} can be estimated from the
held-out diagram. Indeed, conditional on the $n-1$ training observations and
on $Z_i$, the function
$u\mapsto\widehat\lambda^{(-i)}(Z_i,u)$ is fixed, while the conditional mean
measure of $\mu_i^w$ has density $\lambda_w(Z_i,\cdot)$. Therefore,
\begin{equation}\label{eq:cv-conditional-identity}
\E\left[
\int_{\OmegaL}
\widehat\lambda^{(-i)}(Z_i,u)
\dd\mu_i^w(u)
\Bigm|
\left\{(Z_j,D_j):j\neq i\right\},Z_i
\right]
=
\int_{\OmegaL}
\widehat\lambda^{(-i)}(Z_i,u)\lambda_w(Z_i,u)
\dd u.
\end{equation}
This suggests the criterion
\begin{equation}\label{eq:cv-criterion}
\mathrm{CV}_n(h_Z,h_U)
:=
\frac1n\sum_{i=1}^n
\left\{
\int_{\OmegaL}
\widehat\lambda^{(-i)}(Z_i,u)^2
\dd u
-
2\int_{\OmegaL}
\widehat\lambda^{(-i)}(Z_i,u)
\dd\mu_i^w(u)
\right\}.
\end{equation}
Both terms are computable. The first is an integral over the persistence
window, evaluated numerically if necessary, while
\[
\int_{\OmegaL}
\widehat\lambda^{(-i)}(Z_i,u)
\dd\mu_i^w(u)
=
\sum_{x\in D_i}
a_xw(x)\widehat\lambda^{(-i)}(Z_i,x)
\]
is a finite sum over the points of the held-out diagram.

Related leave-one-out criteria for density
estimation are studied in
\citet{rudemo1982empirical,bowman1984alternative}, and related criteria for
spatial-intensity estimation appear in \citet{cronie2018non}.

\begin{proposition}[Unbiased leave-one-out risk identity]
\label{prop:cv-risk}
Assume that
\[
\E\int_{\OmegaL}\lambda_w(Z_1,u)^2\dd u<\infty,
\qquad
\E\int_{\OmegaL}
\widehat\lambda^{(-1)}(Z_1,u)^2
\dd u<\infty.
\]
Then, for every fixed candidate pair $(h_Z,h_U)$,
\[
\E\,\mathrm{CV}_n(h_Z,h_U)
=
R_{n-1}(h_Z,h_U)-C_0.
\]
\end{proposition}

We minimize \eqref{eq:cv-criterion} over a grid of bandwidth pairs, allowing
$h_Z$ and $h_U$ to be tuned separately. Although \Cref{prop:cv-risk} is stated for leave-one-out cross-validation, the same conditional calculation applies to $K$-fold cross-validation: one replaces $\widehat\lambda^{(-i)}$ by the estimator trained without the fold containing observation $i$, and obtains the corresponding identity for the average fold-specific hold-out risk. This is the computational version used in \Cref{sec:sim-cv}.

The proposition establishes unbiasedness of the criterion for the hold-out
risk, up to the bandwidth-independent constant $C_0$, which implies that the
minimizers of $\E\,\mathrm{CV}_n(h_Z,h_U)$ coincide with the minimizers of
the leave-one-out risk $R_{n-1}(h_Z,h_U)$. It does not by itself
give an oracle inequality or an adaptive convergence rate for the selected
bandwidths.

\subsection[Proof of the cross-validation proposition]{Proof of \Cref{prop:cv-risk}}\label{app:cv}

\begin{proof}[Proof of \Cref{prop:cv-risk}]
Fix an index $i$, and let $\mathcal T_i$ denote all observations used to
construct $\widehat\lambda^{(-i)}$, together with any fold assignment or
auxiliary randomization used by the fitting rule. In the leave-one-out setting $\mathcal T_i = \{1 \le j \le n, j \ne i\}$.

By construction, $(Z_i,D_i)$ is independent of the training information
$\mathcal T_i$.  More precisely, conditional on $Z_i$ and $\mathcal T_i$,
the conditional law of $D_i$ is still the law of $D$ given $Z=Z_i$, while
$\widehat\lambda^{(-i)}$ is fixed.
By the definition of conditional intensity,
\begin{equation}\label{eq:cv-conditional-campbell}
 \E\left[\int_{\OmegaL}\widehat\lambda^{(-i)}(Z_i,u)\dd\mu_i^w(u)
 \Bigm|Z_i,\mathcal T_i\right]
 =\int_{\OmegaL}\widehat\lambda^{(-i)}(Z_i,u)\lambda_w(Z_i,u)\dd u.
\end{equation}
Applying the Cauchy--Schwarz inequality (first on $L^2(\OmegaL)$ and then for expectations)
\begin{align*}
 \E\int_{\OmegaL}|\widehat\lambda^{(-i)}(Z_i,u)\lambda_w(Z_i,u)|\dd u
 &\leq\E\left[
 \left(\int\widehat\lambda^{(-i)}(Z_i,u)^2\dd u\right)^{1/2}
 \left(\int \lambda_w(Z_i,u)^2\dd u\right)^{1/2}\right]\\
 &\leq
 \left(\E\int\widehat\lambda^{(-i)}(Z_i,u)^2\dd u\right)^{1/2}
 \left(\E\int\lambda_w(Z_i,u)^2\dd u\right)^{1/2}
 <\infty,
\end{align*}
which shows that \eqref{eq:cv-conditional-campbell} is integrable and justifies interchanging expectations and integrals.

The expected value of the contribution of the $i$th observation to the cross-validation score $\mathrm{CV}_n(h_Z,h_U)$ is therefore
\begin{align*}
 &\E\left[\int\widehat\lambda^{(-i)}(Z_i,u)^2\dd u
 -2\int\widehat\lambda^{(-i)}(Z_i,u)\dd\mu_i^w(u)\right]\\
 &\quad=
 \E\int\left(\widehat\lambda^{(-i)}(Z_i,u)^2
 -2\widehat\lambda^{(-i)}(Z_i,u)\lambda_w(Z_i,u)\right)\dd u.
\end{align*}
Use the elementary identity $a^2-2ab=(a-b)^2-b^2$ pointwise, and then
integrate, to get
\begin{equation}\label{eq:cv-one-index-risk}
 \E\left[\int(\widehat\lambda^{(-i)})^2
 -2\int\widehat\lambda^{(-i)}\dd\mu_i^w\right]
 =\E\int(\widehat\lambda^{(-i)}(Z_i,u)-\lambda_w(Z_i,u))^2\dd u
 -\E\int\lambda_w(Z_i,u)^2\dd u.
\end{equation}

Averaging \eqref{eq:cv-one-index-risk} over $i$ gives
\begin{align}
 \E\mathrm{CV}_n(h_Z,h_U)
 &=\frac1n\sum_{i=1}^n
 \E\int_{\OmegaL}
 (\widehat\lambda^{(-i)}(Z_i,u)-\lambda_w(Z_i,u))^2\dd u
 \nonumber\\
 &\quad-\E\int_{\OmegaL}\lambda_w(Z,u)^2\dd u.
 \label{eq:cv-average-risk}
\end{align}
This is the general average-risk identity displayed in the proposition.
For the leave-one-out rule used here, the fitting map is
permutation-equivariant: relabeling the observations only relabels the held-out
fit.  The same is true for equal-sized folds when the fold construction and
fitting rule are permutation-equivariant.  Since the observations are
identically distributed, the $n$ risks in the sum in
\eqref{eq:cv-average-risk} are then equal.  In that case their average equals
the risk for index one, giving the single-risk reduction stated after the
display in the proposition.

Finally, the last term in \eqref{eq:cv-average-risk} involves only the true
conditional intensity and the design distribution.  In particular, it does
not depend on $(h_Z,h_U)$.  Subtracting the same constant from every candidate
risk does not change its minimizers, which proves the bandwidth-selection
claim.
\end{proof}

\clearpage
\section{Simulation studies}\label{sec:simulations}\label{app:simulation-results}

We report two studies. The first validates the estimator \eqref{eq:nw-estimator} against an analytically known conditional intensity, across data-generating processes, covariate dimensions, sample sizes, and bandwidth schedules derived from \Cref{cor:holder-rate}. The second examines bandwidth selection by the unbiased-risk criterion of \Cref{sec:cv}; in this design it selects the oracle candidate from the chosen grid. The experiments were run with fixed random seeds in a plain NumPy implementation of the estimators; code reproducing every table is provided in the supplementary material.

\subsection{Forward estimation with an exact conditional intensity}\label{sec:sim-forward}

\paragraph{Design rationale.}
A recurring pitfall in evaluating intensity estimators on simulated diagrams is that the ``ground truth'' is itself produced by smoothing a large Monte Carlo sample, and therefore depends on the smoothing bandwidth: the empirical loss then compares the estimator with a moving oracle, and the resulting curves need not decrease in $n$ even for a consistent estimator. The present design avoids this artefact by generating diagrams directly from a conditional Poisson model whose weighted intensity is available in closed form, so that all reported errors are computed against a fixed analytical function.

\paragraph{Data-generating processes.}
For each observation, a covariate $Z_i\sim\mathrm{Uniform}([0,1]^{d})$ is drawn, with $d\in\{1,2,4\}$. Conditionally on $Z_i=z$, the diagram is a Poisson point process on the birth--persistence square $[0,1]^2$: the number of points is $N_i\mid Z_i=z\sim\mathrm{Poisson}(m(z))$ and, given $N_i$, the points $(B_{ij},P_{ij})$ are i.i.d.\ from the density
\[
f_z(b,p)=1+a(z)\cos(2\pi b)\cos(2\pi p)+c(z)\sin(2\pi b)\sin(2\pi p),
\qquad (b,p)\in[0,1]^2,
\]
whose trigonometric terms integrate to zero, so $f_z$ integrates to one; the coefficients are bounded so that $f_z$ remains positive. The covariate enters through the smooth scalar index $\rho(z)=z_1$ for $d=1$ and
\[
\rho(z)=0.65\frac{\sum_{j=1}^{d}jz_j}{\sum_{j=1}^{d}j}
+0.35\frac1d\sum_{j=1}^{d}\sin^2(\pi z_j)
\qquad\text{for }d>1.
\]
Three data-generating processes are considered: \emph{location}, in which the total mass is constant and the covariate changes the shape of $f_z$; \emph{mass}, in which the shape is constant and the covariate changes the total mass; and \emph{mixed}, in which both vary. Their exact definitions are collected in \Cref{tab:dgp}. The estimator uses the persistence weight $w(b,p)=p$, which vanishes linearly on the diagonal $\{p=0\}$ in accordance with \Cref{ass:weight}, so the known target is the weighted conditional intensity
\[
\lambda_w(z,b,p)=m(z)pf_z(b,p).
\]
Because $N_i$ is Poisson, the weighted mass $M(D)$ is unbounded and \Cref{ass:envelope}\textup{(E1)} fails; the design is instead covered by the moment alternative of \Cref{prop:truncation}, since $M(D)\leq N$ has moments of every order and the points have a common bounded density.

\paragraph{Estimator and bandwidths.}
The estimator is \eqref{eq:nw-estimator} with Epanechnikov product kernels in both the covariate and the diagram coordinates. Bandwidths follow the schedule \eqref{eq:holder-bandwidths} of \Cref{cor:holder-rate} with covariate dimension $d$, diagram dimension two, and smoothness parameters $s_Z=s_U=1$, for which both exponents simplify to $1/(d+4)$:
\[
h_Z=h_{\mathrm{birth}}=h_{\mathrm{persistence}}=c_{\mathrm{bw}}\left(\frac{\log n}{n}\right)^{1/(d+4)},
\qquad
c_{\mathrm{bw}}\in\{0.10,0.25,0.50,0.75\}.
\]
The multiplier $c_{\mathrm{bw}}$ scans a range from severe undersmoothing to the vicinity of the rate-optimal constant. The full design crosses the three data-generating processes with $d\in\{1,2,4\}$, sample sizes $n\in\{10^2,10^3,10^4,10^5\}$, and the four multipliers, with $100$ Monte Carlo replicates per cell.

\paragraph{Evaluation.}
For $d=1$, we use nine equally spaced covariate values in $[0.10,0.90]$. For $d>1$, we use deterministic Halton points mapped to $[0.10,0.90]^{d}$. Diagram-space errors are computed on an interior birth--persistence grid on $[0.10,0.90]^2$. The evaluation-point normalization $\int\kappa_{U,h}(u,v)\dd v=1$ corrects the loss of kernel mass at the boundary, but a local-constant estimator can still have first-order smoothing bias there for a nonconstant target; including the outer edge would therefore make the reported maximum primarily a boundary diagnostic. Writing $Z_{\mathrm{ev}}$ for the evaluation covariates, $U_{\mathrm{ev}}$ for the interior grid with cell area $A$, and $e_z(u):=\widehat\lambda_{w,n}(z,u)-\lambda_w(z,u)$, the reported losses are
\begin{align*}
\text{integrated sup loss}
&=\frac{1}{|Z_{\mathrm{ev}}|}\sum_{z\in Z_{\mathrm{ev}}}\max_{u\in U_{\mathrm{ev}}}|e_z(u)|,
\qquad
\text{relative sup loss}
=\frac{\text{int.\ sup loss}}{\frac{1}{|Z_{\mathrm{ev}}|}\sum_{z}\max_{u}|\lambda_w|},
\\[2pt]
\text{L1 loss}
&=\frac{1}{|Z_{\mathrm{ev}}|}\sum_{z\in Z_{\mathrm{ev}}}A\sum_{u\in U_{\mathrm{ev}}}|e_z(u)|,
\qquad\quad
\text{ISE}
=\frac{1}{|Z_{\mathrm{ev}}|}\sum_{z\in Z_{\mathrm{ev}}}A\sum_{u\in U_{\mathrm{ev}}}e_z(u)^2 .
\end{align*}

\paragraph{Results.}
The endpoint summaries at the multiplier minimizing mean integrated sup loss
are reported in
\Cref{tab:empirical-verifications}; the complete losses at $n=10^5$ and the
empirical rate exponents obtained by regressing log loss on $\log n$ within
each cell are reported in \Cref{tab:forward-final,tab:forward-rates}.

\paragraph{Findings.}
The main diagnostic is positive: for every data-generating process and covariate dimension, the empirical sup-loss exponents at multipliers $0.50$ and $0.75$ are positive, in the direction predicted by \Cref{thm:supnorm-rate,cor:holder-rate}, and the relative sup loss at $n=10^5$ is between $3\%$ and $7\%$ for $d\in\{1,2\}$. Three qualitative patterns deserve comment. First, errors grow with $d$ at fixed $n$, as the effective dimension $d+2$ of the sup-norm problem dictates. Second, the smallest multiplier $0.10$ at $d=4$ yields \emph{negative} empirical rates: at these sample sizes such bandwidths leave many evaluation points with nearly empty covariate neighbourhoods, so the variance term of \eqref{eq:supnorm-rate} dominates and the loss has not yet entered its asymptotic regime---a finite-sample manifestation of the bandwidth condition $nh_Z^dh_U^2/\log n\to\infty$ in \Cref{ass:kernels}. Third, the empirical exponents themselves are rough, being estimated from only four sample sizes, and we deliberately refrain from a quantitative comparison with the exponents of \Cref{cor:holder-rate,cor:pot-holder-rate}; the present study validates the qualitative behaviour against a fixed analytical target and shows how a bandwidth-dependent reference can obscure convergence by making the target move with the bandwidth.

\subsection{Data-driven bandwidth selection}\label{sec:sim-cv}

We next apply the unbiased-risk criterion \eqref{eq:cv-criterion} within the exact-intensity design, where the integrated squared error of every candidate bandwidth can be computed against the analytical target. The configuration is the \emph{mixed} process of \Cref{tab:dgp} with $d=1$, $n=1{,}000$, and $12$ Monte Carlo replicates. Candidate bandwidths are $h_Z=h_U=c_{\mathrm{bw}}(\log n/n)^{1/5}$ with multipliers $c_{\mathrm{bw}}\in\{0.10,0.25,0.50,0.75,1.00,1.50\}$ (base value $0.370$ at $n=1{,}000$). The criterion is evaluated in a five-fold scheme: for each fold and each candidate, the held-out term of \eqref{eq:cv-criterion} uses a $26\times26$ midpoint quadrature of the full window for the quadratic integral and the exact finite sum over the held-out diagram points for the linear term. The comparison metric is the interior integrated squared error of the corresponding full-data fit, computed as in \Cref{sec:sim-forward}.

The risk profile over multipliers is sharply convex, with the oracle at $c_{\mathrm{bw}}=0.50$ in every replicate. The criterion selected this oracle multiplier in $12$ of $12$ replicates, so the cross-validated and oracle fits both have mean integrated squared error $0.124$ on this grid. The candidate grid is coarse, and the comparison evaluates the selected bandwidth on the full-sample fit while the criterion targets hold-out risk. The experiment nevertheless shows that the criterion avoids both the badly undersmoothed multipliers, whose risk is one to two orders of magnitude larger, and the oversmoothed ones. The complete risk profile is reported in \Cref{tab:cv}.

\FloatBarrier
\subsection{Additional numerical results}

This subsection collects the exact data-generating processes and the full numerical
summaries for the two simulation studies.

\begin{table}[ht!]
\centering
\caption{Conditional Poisson data-generating processes for the forward study.}
\label{tab:dgp}
\begin{tabular}{lccc}
\toprule
DGP & $m(z)$ & $a(z)$ & $c(z)$\\
\midrule
location & $9$ & $0.30(2\rho(z)-1)$ & $0.20\sin(2\pi\rho(z))$\\
mass & $7+5\rho(z)$ & $0.22$ & $-0.16$\\
mixed & $8+3\sin^2(\pi\rho(z))$ & $0.28\sin(2\pi\rho(z))$ & $0.22\cos(2\pi\rho(z))$\\
\bottomrule
\end{tabular}
\end{table}

\begin{table}[ht!]
\centering
\caption{Forward study: losses at $n=10^5$ with the Epanechnikov kernel, at the multiplier minimizing mean integrated sup loss per configuration. Averages over $100$ replicates.}
\label{tab:forward-final}
\begin{tabular}{lcccccc}
\toprule
DGP & $d$ & Multiplier & Integrated sup & Relative sup & L1 & ISE\\
\midrule
location & 1 & 0.75 & 0.320 & 0.035 & 0.048 & 0.007\\
location & 2 & 0.50 & 0.591 & 0.069 & 0.075 & 0.017\\
location & 4 & 0.50 & 1.383 & 0.165 & 0.142 & 0.062\\
mass     & 1 & 0.75 & 0.333 & 0.033 & 0.049 & 0.006\\
mass     & 2 & 0.50 & 0.623 & 0.061 & 0.081 & 0.019\\
mass     & 4 & 0.50 & 1.724 & 0.170 & 0.155 & 0.074\\
mixed    & 1 & 0.75 & 0.411 & 0.041 & 0.063 & 0.011\\
mixed    & 2 & 0.50 & 0.694 & 0.063 & 0.090 & 0.024\\
mixed    & 4 & 0.50 & 1.689 & 0.151 & 0.178 & 0.099\\
\bottomrule
\end{tabular}
\end{table}
\FloatBarrier

\Cref{tab:forward-rates} reports the empirical rate exponents summarized in
\Cref{sec:sim-forward}. They are obtained from only four sample sizes and serve as a
finite-sample diagnostic of the direction and dimension dependence predicted by the
theory.

\begin{table}[p]
\centering
\caption{Forward study: empirical rate exponents defined as minus the slope from regressing log loss on $\log n$ over $n\in\{10^2,10^3,10^4,10^5\}$, Epanechnikov kernel. Positive values indicate decreasing loss.}
\label{tab:forward-rates}
\small
\begin{tabular}{clcccc}
\toprule
$d$ & DGP & Multiplier & Sup rate & L1 rate & ISE rate\\
\midrule
1 & location & 0.10 & 0.272 & 0.243 & 0.495\\
1 & location & 0.25 & 0.210 & 0.245 & 0.491\\
1 & location & 0.50 & 0.331 & 0.300 & 0.618\\
1 & location & 0.75 & 0.385 & 0.364 & 0.755\\
1 & mass     & 0.10 & 0.273 & 0.243 & 0.494\\
1 & mass     & 0.25 & 0.211 & 0.244 & 0.489\\
1 & mass     & 0.50 & 0.362 & 0.295 & 0.615\\
1 & mass     & 0.75 & 0.406 & 0.364 & 0.770\\
1 & mixed    & 0.10 & 0.271 & 0.243 & 0.496\\
1 & mixed    & 0.25 & 0.212 & 0.244 & 0.490\\
1 & mixed    & 0.50 & 0.335 & 0.303 & 0.625\\
1 & mixed    & 0.75 & 0.365 & 0.343 & 0.717\\
\addlinespace
2 & location & 0.10 & 0.231 & 0.147 & 0.402\\
2 & location & 0.25 & 0.185 & 0.218 & 0.441\\
2 & location & 0.50 & 0.290 & 0.268 & 0.551\\
2 & location & 0.75 & 0.196 & 0.306 & 0.587\\
2 & mass     & 0.10 & 0.223 & 0.147 & 0.392\\
2 & mass     & 0.25 & 0.184 & 0.218 & 0.440\\
2 & mass     & 0.50 & 0.316 & 0.266 & 0.554\\
2 & mass     & 0.75 & 0.188 & 0.289 & 0.571\\
2 & mixed    & 0.10 & 0.230 & 0.154 & 0.403\\
2 & mixed    & 0.25 & 0.185 & 0.218 & 0.441\\
2 & mixed    & 0.50 & 0.291 & 0.273 & 0.562\\
2 & mixed    & 0.75 & 0.192 & 0.272 & 0.540\\
\addlinespace
4 & location & 0.10 & $-0.179$ & $-0.014$ & $-0.190$\\
4 & location & 0.25 & 0.094 & 0.136 & 0.267\\
4 & location & 0.50 & 0.178 & 0.222 & 0.442\\
4 & location & 0.75 & 0.086 & 0.210 & 0.348\\
4 & mass     & 0.10 & $-0.173$ & $-0.014$ & $-0.182$\\
4 & mass     & 0.25 & 0.100 & 0.139 & 0.276\\
4 & mass     & 0.50 & 0.177 & 0.216 & 0.434\\
4 & mass     & 0.75 & 0.081 & 0.187 & 0.331\\
4 & mixed    & 0.10 & $-0.172$ & $-0.017$ & $-0.185$\\
4 & mixed    & 0.25 & 0.098 & 0.138 & 0.273\\
4 & mixed    & 0.50 & 0.178 & 0.218 & 0.436\\
4 & mixed    & 0.75 & 0.089 & 0.175 & 0.324\\
\bottomrule
\end{tabular}
\end{table}

\begin{table}[!htbp]
\centering
\caption{Bandwidth selection in the exact-intensity design (mixed process, $d=1$, $n=1{,}000$, $12$ replicates): interior integrated squared error of the full-data fit at each fixed multiplier, at the multiplier selected by the unbiased-risk criterion \eqref{eq:cv-criterion}, and at the per-replicate oracle multiplier. Means (standard deviations).}
\label{tab:cv}
\begin{tabular}{lcc}
\toprule
Bandwidth multiplier & ISE & (sd)\\
\midrule
$0.10$ & 8.776 & 0.316\\
$0.25$ & 0.574 & 0.052\\
$0.50$ & 0.124 & 0.028\\
$0.75$ & 0.190 & 0.025\\
$1.00$ & 0.362 & 0.025\\
$1.50$ & 0.893 & 0.032\\
\midrule
Cross-validated & \textbf{0.124} & 0.028\\
Oracle & 0.124 & 0.028\\
\bottomrule
\end{tabular}
\end{table}

\FloatBarrier

\section{Additional details for the cerebral artery analysis}
\label{app:arteries-details}

This appendix gives the computational and inferential details behind \Cref{sec:arteries}.
The main text defines the scientific targets and reports the principal findings; here we
record data checks, tuning rules, the exact simultaneous statistic, secondary components,
robustness and backmapping diagnostics, and a fuller comparison with previous analyses.

\subsection{Data preparation and persistence construction}
\label{app:arteries-preparation}

The archive contains cross-sectional centreline reconstructions from 98 subjects, with no
repeated measurements of an individual. The deposited graphs contain approximately
\(10.95\) million vertices across \(509\) rooted components. Each component has
three-dimensional coordinates, edges, a parent relation, and exactly one supplied root;
components do not cross the four arterial-system labels.

For each subject, the preparation code checks that the undirected graph is a forest, that
the deposited parent links reproduce exactly the same edge set, that every connected
component has one and only one parentless vertex, and that every component belongs to a
single arterial-system label. The parentless vertex is used as the supplied root. The
polyline length of a system is the sum of Euclidean lengths of all of its centreline edges.

The same union--find lower-star algorithm is applied to \(f\) and \(-f\). The first pass
returns the finite sublevel (local-minimum) pairs, and the second returns the finite
superlevel (local-maximum) pairs. The minimum and maximum of each connected component
supply its global extended pair. The prepared table has \(369{,}136\) rows in total:
\(183{,}035\) finite local-minimum pairs, \(185{,}592\) finite local-maximum pairs, and
\(509\) component min--max pairs.

Persistences at most \(10^{-10}\) mm were treated as numerical zero. Repeating the
preparation with a zero floor produced the same row count; the smallest retained finite
persistence is approximately \(6.0\times10^{-9}\) mm. Every row stores zero-based
internal vertex indices and their one-based source/MATLAB counterparts. Local-maximum
rows additionally retain the merge edge, while global rows retain both extreme vertices.

\subsection{Subject-level measures}
\label{app:arteries-measures}

The raw and system-length-standardized measures are defined in
\Cref{sec:arteries-data}. In the implementation, a finite pair from system \(s\) has raw
weight \(p\) and standardized weight \(1000p/L_{is}\). The factor is attached before
systems are combined. Direct audits recover the raw total mass by summing persistences and
the standardized total mass by summing these four system-specific rates. All coordinate
columns remain in millimetres in both analyses.

\subsection{Bandwidth selection}
\label{app:arteries-bandwidths}

We estimate each conditional intensity with \eqref{eq:nw-estimator}, using an
Epanechnikov kernel in age and a boundary-corrected product Epanechnikov kernel on the
rectangular \((q,p)\) window. Diagram and age smoothing play different roles, so we tune
the two bandwidths sequentially for each finite field and each weighting scheme. The
diagram candidates are
\[
\mathcal H_U=\{6,7,8,9,10,11,12,13,14,15,16,17,18,20,22,24,27,30,34,38\}
\ \mathrm{mm},
\]
and the age candidates are \(\mathcal H_A=\{6,8,\ldots,40\}\) years.

\paragraph{Persistence-plane bandwidth.}
The first stage adapts the least-squares intensity criterion of
\citet{chazal2019density} to the signed old-minus-young contrast. Subjects at or below the
lower age quartile and at or above the upper age quartile form the two strata; the middle
subjects receive zero contrast weight during this stage. We generate \(50\) repeated
stratified five-fold splits at the subject level. For fold \(v\), system \(s\), and
candidate \(h_U\), let \(\widehat\Delta_{v,s,h_U}^{\mathrm{tr}}\) be the smoothed
old-minus-young contrast estimated from the training subjects, with each age group averaged
using weights that sum to one. Let \(\nu_{v,s}^{\mathrm{val}}\) be the corresponding
signed empirical measure from the validation subjects. The fold score is
\[
 \operatorname{CV}_{v}(h_U)
 =
 \sum_{s=1}^4
 \left\{
   \int_{\Omega_L}
     \bigl[\widehat\Delta_{v,s,h_U}^{\mathrm{tr}}(u)\bigr]^2\,\dd u
   -2\int_{\Omega_L}
     \widehat\Delta_{v,s,h_U}^{\mathrm{tr}}(u)
     \,\dd\nu_{v,s}^{\mathrm{val}}(u)
 \right\}.
\]
The validation-only squared norm does not depend on \(h_U\) and is omitted. The four
system scores are added after their separate calculation, so no cross-system products enter
the criterion. We average the five fold scores within each repetition and then average over
the \(50\) repetitions. The primary \(h_U\) minimizes this mean score. As a sensitivity
choice, we also record the largest candidate whose mean score is within one standard error
of the minimum.

\paragraph{Age bandwidth.}
After fixing \(h_U\) at its primary minimizer, the second stage uses all \(98\) subjects.
For a candidate \(h_A\), each subject is left out in turn and its four system-specific
persistence measures are predicted by an Epanechnikov-weighted average of the remaining
subjects, using age as the covariate. The observed response is a discrete weighted diagram
measure; the prediction is a smooth conditional-mean measure obtained after diagram
smoothing at the selected \(h_U\). Both are mapped to the same augmented Persistence
Sphere representation \citep{pegoraro2026persistence}. This representation is linear and
bi-continuous for the partial-optimal-transport topology, and its squared \(L^2\) distance
can be evaluated through inner products. It is substantially faster here than solving an
optimal-partial-transport problem for every subject and bandwidth candidate.

The loss for a held-out subject is the sum, over the four systems, of the squared sphere
\(L^2\) distance between its observed measure and its prediction. We select the candidate
with the smallest mean subject loss; candidates for which any leave-one-out age-kernel
denominator is zero are excluded. We also report
\[
 R^2_{\mathrm{sphere}}
 =1-
 \frac{\text{selected leave-one-out prediction loss}}
      {\text{leave-one-out mean-measure prediction loss}}.
\]
For raw local-minimum and local-maximum measures, these values are \(-0.092\) and
\(0.247\); for the length-standardized measures they are \(0.006\) and \(0.090\).
The held-out response is a discrete realization with substantial subject-to-subject
variation. The regression predicts its smooth conditional mean, so these modest values
indicate limited subject-level age predictability. The bandwidth stabilizes estimation of
that conditional mean, and the selected contrast is assessed by the subject bootstrap
described below.

\begin{table}[t]
\centering
\caption{Selected bandwidths and primary \(95\%\) simultaneous results. Ranges and
percentage changes describe the principal connected component after selection. The
standardized local-maximum set also has four smaller components, detailed in
\Cref{app:arteries-secondary}.}
\label{tab:arteries-primary}
\small
\resizebox{\textwidth}{!}{%
\begin{tabular}{llrrlccrr}
\toprule
Measure & Diagram & \(h_U\) & \(h_A\) & Joint \(95\%\) result & \(q\) (mm) & \(p\) (mm) & Count \(\Delta\) & Mass \(\Delta\)\\
\midrule
Raw & Local minimum & 22 & 18 & Negative principal & 41.3--97.7 & 0.0--35.8 & \(-25.8\%\) & \(-27.5\%\)\\
Raw & Local maximum & 13 & 18 & Negative principal & 17.7--96.9 & 0.0--64.5 & \(-24.3\%\) & \(-33.0\%\)\\
Length-standardized & Local minimum & 12 & 16 & None & -- & -- & -- & --\\
Length-standardized & Local maximum & 13 & 22 & Negative principal; four smaller & 71.2--85.9 & 0.0--14.7 & \(-13.2\%\) & \(-16.6\%\)\\
\bottomrule
\end{tabular}%
}
\end{table}

\subsection{Simultaneous bootstrap bands}
\label{app:arteries-inference}

For each fixed set of selected bandwidths, we draw \(4{,}999\) complete-subject bootstrap
samples. Sampling a subject retains all four labelled systems and every persistence pair
belonging to that subject. For each resample, the age-\(57\) minus age-\(31.25\) contrast is
recomputed on the common grid. The selected bandwidths are held fixed in this primary
bootstrap.

Inference is restricted to an age-independent grid set. A grid location enters the primary
set when (i) at least five subjects have a persistence pair with nonzero local kernel weight
at that location and (ii) the pooled intensity there is at least \(10^{-4}\) times the largest
pooled intensity in the same diagram. These conditions remove locations at which both the
estimated contrast and its standard error are effectively numerical zero. The rule is
constructed without using the age contrast.

Let \(\widehat s_g^m(u)\) denote the pointwise bootstrap standard error of the fitted
contrast. For bootstrap replicate \(\ell\), the simultaneous statistic is
\[
 T_\ell^m
 =
 \max_{g\in\{\mathrm{min},\mathrm{max}\}}
 \max_{u\in\mathcal I_g^m}
 \left|
 \frac{\widehat\Delta_{g,\ell}^{m,*}(u)-\widehat\Delta_g^m(u)}
      {\widehat s_g^m(u)}
 \right|,
\]
where \(\mathcal I_g^m\) is the age-independent inference grid. The empirical
\(1-\alpha\) quantile \(c_{1-\alpha}^m\) yields the band
\(\widehat\Delta_g^m(u)\pm c_{1-\alpha}^m\widehat s_g^m(u)\). The primary \(95\%\)
critical values are \(4.007\) for the raw measures and \(4.185\) for the
system-length-standardized measures.

The bootstrap keeps the selected bandwidths fixed. Its nominal statement concerns the
sampling variability of those smoothed contrasts. Neighbouring-bandwidth fits and
complete deletion retuning are reported separately in \Cref{app:arteries-robustness}.

\subsection{Primary numerical summaries and secondary components}
\label{app:arteries-secondary}

The standardized local-maximum mask has \(1{,}301\) pixels split across five connected
components. R1 has \(725\) pixels and contains \(18{,}861\) observed generators. R2 lies
at \(q=89.3\)--\(96.0\) mm and \(p=0\)--\(2.95\) mm; its features are extremely close
to the diagonal. R3 has nine pixels around \(q=70.35\)--\(72.03\) mm and
\(p=24.85\)--\(26.12\) mm but contains only one observed generator. R4 has \(488\)
pixels around \(q=32.86\)--\(44.23\) mm and \(p=48.86\)--\(62.77\) mm and is supported
by rare high-persistence features, observed in \(15\) of \(25\) younger-quartile subjects
and \(7\) of \(27\) older-quartile subjects. R5 is one grid pixel and contains no observed
generator at that exact location. We base the spatial interpretation on R1 because it is
densely populated and represented in every subject in both age tails. The complete five
components select \(22{,}918\) observed local maxima.

For the one-dimensional maximum-radius comparator, the selected radius bandwidth is
\(16\) mm for both weighting schemes. Numerical integration of the one-dimensional
surfaces recovers direct persistence mass with mean relative error about \(0.153\%\) and
maximum error below \(0.6\%\). No radius interval is selected at \(90\%\), \(95\%\), or
\(99\%\).

\subsection{Global minimum--maximum classes}
\label{app:arteries-global}

The third degree-zero extended-persistence class contains one pair per connected
component. Since Euclidean distance is measured from the supplied root, the global
minimum is zero and a component contributes the point \((0,R_{ic})\), where \(R_{ic}\)
is its largest Euclidean root distance and also the persistence of the global class. A
two-dimensional density is inappropriate because every such pair has \(q=0\). The
accompanying notebook therefore estimates a one-dimensional conditional intensity over
global persistence \(R_{ic}\), separately from the two finite local-extremum fields.

The deposited graphs contain \(509\) components. To distinguish the
largest-persistence tree in each labelled system from the remaining disconnected pieces,
the notebook analyses three populations: all \(509\) components; the
largest-persistence component in
each subject--system, giving \(98\times4=392\) components; and the remaining \(117\)
components. It also compares three measures. The unit-weighted measure
\(\sum_c\delta_{R_{ic}}\) estimates the distribution of component maximum radii without
confounding it with their magnitudes. The framework's persistence-weighted measure is
\(\sum_c R_{ic}\delta_{R_{ic}}\). Its system-length-standardized counterpart assigns
weight \(1000R_{ic}/L_{is}\) to a component in system \(s\). Persistence and age
bandwidths are selected separately for every population and measure, after which
\(4{,}999\) complete-subject bootstrap resamples provide simultaneous bands for the
age-\(57\) minus age-\(31.25\) contrast. We treat this analysis as its own inferential
family and compute its critical values separately from the two finite fields in
\Cref{app:arteries-inference}.

The principal result is already present in the unit-weighted analysis of the largest
component per subject--system. Its selected bandwidths are \(h_R=6\) mm and
\(h_A=24\) years. The \(95\%\) band is positive over \(77.3\)--\(91.5\) mm and negative
over \(97.9\)--\(104.6\) mm; both intervals remain selected at \(99\%\), with slightly
narrower boundaries. Because this population contains exactly four classes per subject,
the integrated unit-weighted intensity is fixed at four up to numerical integration. The
two intervals identify a change in the shape of the conditional distribution: the older
fitted surface has more global classes near \(80\)--\(90\) mm and fewer near \(100\) mm,
with four principal components per subject throughout the analysis.

Weighting the classes by \(R_{ic}\) selects the same intervals and changes the fitted
integrated global persistence from \(396.88\) to \(380.90\) mm, a \(4.0\%\) decrease.
The close agreement with the unit-weighted result is expected because the weight is the
same coordinate on which the one-dimensional intensity is evaluated, so both versions
describe the same redistribution. The length-standardized analysis retains both regions
and has a positive integrated change of \(13.2\%\). This total is sensitive to the
denominator: fitted total centreline length decreases by \(14.1\%\) between the two
evaluation ages. Its estimand is global persistence per unit system length; the
unnormalized fitted arterial extent decreases.

The all-component analysis recovers the same principal positive and negative regions, and
the \(117\) additional components have no selected interval at \(95\%\) under any of the
three measures. The largest trees carry the selected change, and the number of small
disconnected pieces does not contribute a selected interval. Exact backmapping gives
\(105\) largest-component maxima in the unit-weighted positive interval and \(118\) in
the negative interval. In root-centred coordinates, the two clouds overlap and follow the
outer arterial trajectory, with median root-centred \(Y\) coordinates of approximately
\(56.7\) and \(71.8\) mm. \Cref{fig:arteries-global} displays the conditional intensity,
the simultaneous contrast, and the root-relative spatial locations of these endpoints.

\begin{figure}[p]
\centering
\includegraphics[width=\textwidth]{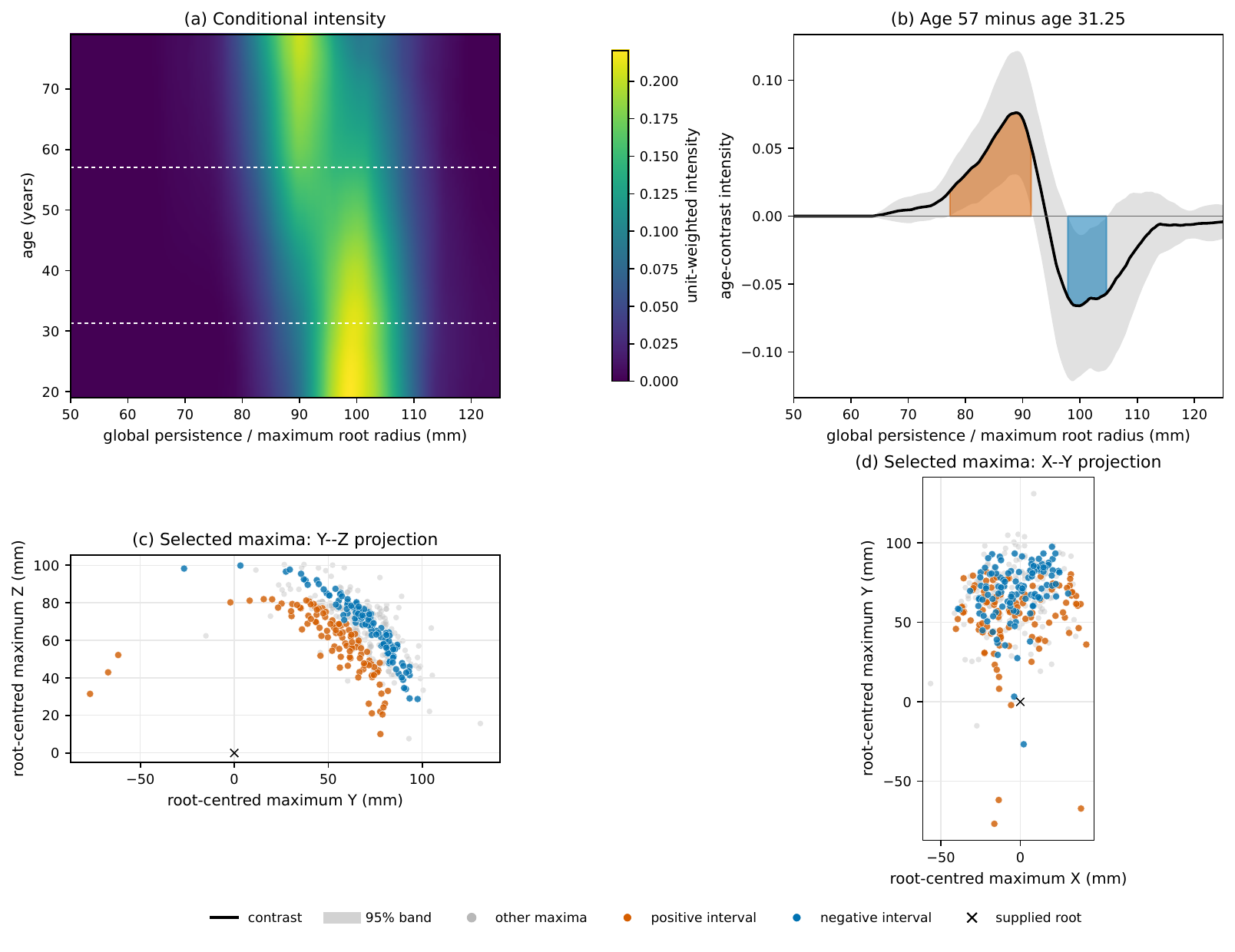}
\caption{Global minimum--maximum classes for the largest-persistence component in each
subject--system, using unit weights. Panels show the fitted conditional intensity, the
age-\(57\) minus age-\(31.25\) contrast and simultaneous \(95\%\) band, and root-relative
backmaps of the maxima in the selected positive and negative intervals.}
\label{fig:arteries-global}
\end{figure}

\begin{table}[ht!]
\centering
\caption{One-dimensional analysis of the global class for the largest-persistence
component in each subject--system. Intervals are selected by the separate \(95\%\)
simultaneous band for the age-\(57\) minus age-\(31.25\) contrast.}
\label{tab:arteries-global}
\small
\begin{tabular}{lrrll}
\toprule
Measure & \(h_R\) & \(h_A\) & Positive interval (mm) & Negative interval (mm)\\
\midrule
Unit-weighted count & 6 & 24 & 77.3--91.5 & 97.9--104.6\\
Persistence weighted & 6 & 24 & 77.3--91.5 & 97.9--104.6\\
Persistence per 1000 mm & 8 & 24 & 74.3--92.1 & 99.9--103.3\\
\bottomrule
\end{tabular}
\end{table}
\FloatBarrier

\subsection{Robustness analyses}
\label{app:arteries-robustness}

We varied the inference set by requiring at least ten rather than five contributing subjects
and by using pooled-intensity floors of \(10^{-4}\) and \(10^{-3}\). We also repeated the
simultaneous analysis at \(90\%\), \(95\%\), and \(99\%\). Across these choices, both raw
fields retain broad negative regions, the standardized local-minimum field remains empty,
and the standardized local-maximum contrast remains negative in the neighbourhood of R1.

The same neighbourhood is negative under adjacent persistence and age bandwidths and
when the displayed age anchors use lower and upper quintiles, quartiles, or tertiles. Its
broad location and sign are more stable than the ordering of individual boundary pixels.

The four labelled arterial systems were additionally fitted separately at the selected pooled
bandwidths. Summing their fitted intensity fields recovers the pooled field with relative
\(L^2\) discrepancy below \(10^{-7}\). Within each pooled primary region, all four
system-specific integrated contrasts have the same negative sign. This descriptive
decomposition shows that all four labelled systems contribute in the same direction. The
simultaneous calibration remains attached to the pooled analysis.

Finally, we deleted each subject in turn and reran both bandwidth-selection stages and the
final contrast fit. Age cutoffs and repeated folds were recomputed on the remaining \(97\)
subjects. The resulting contrast maps are highly correlated with the full-data maps, and
every grid location in the full-data R1 negative set remains negative in every deletion fit.

\subsection{Backmapping and spatial summaries}
\label{app:arteries-backmapping}

The backmapping audit covers all \(368{,}627\) finite generators. Root distance recomputed
from the stored centreline coordinates agrees with the appropriate persistence endpoint to
within \(1.5\times10^{-14}\) mm. No arterial-system mismatch or failure of the relevant
local-extremum condition was found.

For spatial plots, a generator at deposited coordinate \(x_{ij}\) is represented by
\(x_{ij}-r_{ic}\), where \(r_{ic}\) is the supplied root of its connected component.
This translation leaves the complete persistence calculation unchanged. It removes the
unregistered scan origin but does not correct rotation or scale, so all spatial statements
are root-relative rather than atlas-anatomical.

Using all finite local maxima as reference, the full radius distribution has median
approximately \(52.3\) mm; R1 has median \(78.75\) mm and interquartile range
\(76.19\)--\(81.93\) mm. Since the generator radius of a superlevel maximum is \(q+p\),
the actual shape of R1 predicts a relatively thin outer band after root centring. A
selected region spanning a more articulated set of \((q,p)\) values would not generally
have this geometry.

We retain one post-selection check for localization within that band. Each R1 generator
is compared with the five nearest non-selected local maxima in radius from the same
subject, arterial system, and connected component; the median fifth-control radius gap is
\(3.36\) mm and its \(95\)th percentile is \(7.92\) mm. First- and second-order
directional moments are averaged within subject and calibrated by \(99{,}999\)
subject-level random sign flips. Both comparisons give
\(p_{\mathrm{MC}}=10^{-5}\), and the conclusion is unchanged from one to twenty controls.
This check shows that the selected directions are not explained by radius alone; it is
not a second selection analysis or a test of an age-by-direction interaction.

The degree comparison was designed primarily to check whether selection was concentrated
at leaves. It was not: \(95.7\%\) of R1 generators and \(93.2\%\) of non-selected
local-maximum generators have degree two, with few terminal or branch vertices in either
group. A local radial maximum can occur on a smooth centreline segment whenever its
direction relative to the root changes from outward to inward.

All summaries here characterize a region selected by the simultaneous persistence-plane
analysis and are intended to guide subsequent anatomically registered or clinician-led
interpretation.

\subsection{Detailed comparison with previous analyses}
\label{app:arteries-literature}

The earliest CASILab ageing analyses used conventional vessel attributes within
anatomically predefined arterial systems \citep{bullitt2010effects}. A second strand
treats the observations as combinatorial or metric trees. Object-oriented formulations
\citep{wang2007object,aydin2009principal,aydin2012new,alfaro2014dimension,wang2012nonparametric}
establish age-associated variation in tree structure, while branching models
\citep{chang2013generalized,choudhury2018branch} find that branching probabilities
generally decrease with age, conditional on vessel characteristics and branch order.
Together with scale-space analysis \citep{shen2014functional} and correspondence-based
geometry \citep{skwerer2014tree}, these results show that the age association depends on
structural scale and on which geometric information the representation retains.

Persistent-homology analyses provide a complementary multiscale view.
\citet{bendich2016persistent,agerberg2025algebraic} found strong age associations and
showed that intermediate, not only longest, persistence ranges can carry substantial age
information. \citet{matuk2024topogeometric} found that an age association survives
normalization and elastic alignment of persistence landscapes, in contrast to a much less
stable sex effect. These works establish a reproducible multiscale signal but do not
combine inferential region selection with exact generator-level localization on the
three-dimensional centrelines.

Our raw contrasts agree with the broad finding that older subjects have fewer radial
reversals and less persistence mass. The system-length-standardized analysis isolates a
more specific feature: an age-associated decrease in ordinary, outward-to-inward
centreline reversals with \(71.2\le q\le85.9\) mm and \(0\le p\le14.7\) mm. At the other
end of the descriptor, the global class records a redistribution from larger to smaller
maximum root radii. These are root-relative geometric statements, not assignments to
registered anatomical vessels.

The peripheral radii are qualitatively consistent with the cortical-proximity effect of
\citet{skwerer2014tree}, which was found under cortical-landmark but not descendant
correspondence. The closest methodological precedent is \citet{guo2022statistical}, who
retained embedded geometry through elastic shape graphs and, after normalization for total
arterial length, visualized the edges with the largest fitted age-associated deformations.
That analysis requires cross-subject registration of graph elements. Our approach avoids
segment correspondence by defining radial reversals within subject and comparing them in
persistence space, but consequently cannot name an anatomical vessel. Longitudinal or
registered data would permit a more direct comparison of the two localization strategies.
\FloatBarrier

\bibliographystyle{abbrvnat}
\bibliography{paper_20260831}

\end{document}